\documentclass[11pt, leqno]{amsart}
\usepackage{latexsym}
\usepackage{amssymb}
\usepackage{amsmath}

\usepackage{tikz}
\usepackage{tkz-euclide}
\usetikzlibrary{decorations.fractals}
\usetikzlibrary{decorations.footprints}
\usepackage[colorlinks=true, pdfstartview=FitV, linkcolor=red, citecolor=red, urlcolor=blue, backref=page]{hyperref} 

\usepackage{tikz,amsthm,amsmath,amstext,amssymb,amscd,epsfig,euscript,pspicture,multicol,graphpap,graphics,graphicx,enumerate,subfig,sidecap,wrapfig,color,pict2e}

\usepackage[utf8]{inputenc}

\usepackage[top=1in, bottom=1in, left=1in, right=1in, a4paper]{geometry}

\calclayout
\allowdisplaybreaks

\theoremstyle{plain}
\newtheorem{theorem}[equation]{Theorem}
\newtheorem{lemma}[equation]{Lemma}

\newtheorem{proposition}[equation]{Proposition}

\theoremstyle{definition}
\newtheorem{definition}[equation]{Definition}

\theoremstyle{remark}
\newtheorem{remark}[equation]{Remark}

\numberwithin{equation}{section}

\newcommand{\vertiii}[1]{{\left\vert\kern-0.25ex\left\vert\kern-0.25ex\left\vert #1 
		\right\vert\kern-0.25ex\right\vert\kern-0.25ex\right\vert}}

\newcommand{\ZZ}{{\mathbb{Z}}}

\newcommand{\dist}{\operatorname{dist}}

\newcommand{\re}{\mathbb{R}}

\newcommand{\ree}{\mathbb{R}^{n+1}}
\newcommand{\N}{\mathbb{N}}
\newcommand{\dd}{\mathbb{D}}

\newcommand{\F}{\mathcal{F}}

\newcommand{\W}{\mathcal{W}}

\newcommand{\pom}{\partial\Omega}

\renewcommand{\emptyset}{\mbox{\textup{\O}}}
\newcommand{\tinyemptyset}{\mbox{\tiny \textup{\O}}}

\DeclareMathOperator{\supp}{supp}

\DeclareMathOperator{\diam}{diam}
\DeclareMathOperator{\interior}{int}

\DeclareMathOperator*{\Lip}{Lip}

\def\div{\mathop{\operatorname{div}}\nolimits}
\def\Cap{\mathop{\operatorname{Cap}_2}\nolimits}

\def\Xint#1{\mathchoice
	{\XXint\displaystyle\textstyle{#1}}%
	{\XXint\textstyle\scriptstyle{#1}}%
	{\XXint\scriptstyle\scriptscriptstyle{#1}}%
	{\XXint\scriptscriptstyle\scriptscriptstyle{#1}}%
	\!\int}
\def\XXint#1#2#3{{\setbox0=\hbox{$#1{#2#3}{\int}$}
		\vcenter{\hbox{$#2#3$}}\kern-.5\wd0}}
\def\aver#1{\Xint-_{#1}}

\renewcommand*{\backref}[1]{}
\renewcommand*{\backrefalt}[4]{%
 \ifcase #1 (Not cited.)%
   \or        (Cited on page~#2.)%
    \else      (Cited on pages~#2.)%
    \fi}

\begin{document} 

	\allowdisplaybreaks[3]

\title[Elliptic operators on rough domains]{Square function and non-tangential maximal function estimates for elliptic operators in 1-sided NTA domains satisfying the capacity density condition}

\author[M. Akman]{Murat Akman}
\address{{Murat Akman}\\
Department of Mathematical Sciences
\\
University of Essex
\\
Colchester CO4 3SQ, United Kingdom}	\email{murat.akman@essex.ac.uk}

\author[S. Hofmann]{Steve Hofmann}

\address{{Steve Hofmann}\\
Department of Mathematics \\University of Missouri \\ Columbia, MO 65211, USA} \email{hofmanns@missouri.edu}

\author[J. M. Martell]{José María Martell}
\address{{José María Martell} \\
Instituto de Ciencias Matemáticas CSIC-UAM-UC3M-UCM \\ 
Consejo Superior de Investigaciones Científicas \\ C/ Nico\-lás Cabrera, 13-15 \\
E-28049 Madrid, Spain} \email{chema.martell@icmat.es}

\author[T. Toro]{Tatiana Toro}

\address{{Tatiana Toro}\\
University of Washington \\
Department of Mathematics \\
Box 354350\\
Seattle, WA 98195-4350} \email{toro@uw.edu}

\thanks{The second author was partially supported by NSF grants DMS-1664047
and DMS-2000048. 
The third author acknowledges financial support from the Spanish Ministry of Science and Innovation, through the ``Severo Ochoa Programme for Centres of Excellence in R\&D'' (CEX2019-000904-S) and through the grant MTM PID2019-107914GB-I00. The third author also acknowledges that the research leading to these results has received funding from the European Research Council under the European Union's Seventh Framework Programme (FP7/2007-2013)/ ERC agreement no. 615112 HAPDEGMT. The fourth author was partially supported by the Craig McKibben \& Sarah Merner Professor in Mathematics, by NSF grant number DMS-1664867 and DMS-1954545, and by the Simons Foundation Fellowship 614610.}

\thanks{This material is based upon work supported by the National Science Foundation under Grant No. DMS-1440140 while the authors were in residence at the Mathematical Sciences Research Institute in Berkeley, California, during the Spring 2017 semester.}

\date{\today}

\subjclass[2010]{31B05, 35J08, 35J25, 42B37, 42B25, 42B99}

\keywords{Uniformly elliptic operators, elliptic measure, the Green function, 1-sided non-tangentially accessible domains, 1-sided chord-arc domains,  capacity density condition, Ahlfors-regularity, $A_\infty$ Muckenhoupt weights, Reverse Hölder, Carleson measures, square function estimates, non-tangential maximal function estimates, dyadic analysis, sawtooth domains, perturbation}
\begin{abstract}
Let $\Omega\subset\mathbb{R}^{n+1}$, $n\ge 2$, be a 1-sided non-tangentially accessible domain (aka uniform domain),  that is, $\Omega$ satisfies the interior Corkscrew and Harnack chain conditions, which are respectively scale-invariant/quantitative versions of openness and path-connectedness. Let us assume also that $\Omega$ satisfies the so-called capacity density condition, a quantitative version of the fact that all boundary points are Wiener regular. Consider $L_0 u=-\mathrm{div}(A_0\nabla u)$, $Lu=-\mathrm{div}(A\nabla u)$,  two real (non-necessarily symmetric) uniformly elliptic operators in $\Omega$, and write $\omega_{L_0}$, $\omega_L$ for the respective associated elliptic measures. The goal of this program is to find sufficient conditions guaranteeing that $\omega_L$ satisfies an $A_\infty$-condition or a $RH_q$-condition with respect to $\omega_{L_0}$. In this  paper we are interested in obtaining square function and non-tangential estimates for solutions of operators as before.  We establish that bounded weak null-solutions satisfy Carleson measure estimates, with respect to the associated elliptic measure. We also show that for every weak null-solution, the associated square function can be controlled by the non-tangential maximal function in any Lebesgue space with respect to the associated elliptic measure. These results extend previous work of Dahlberg-Jerison-Kenig and are fundamental for the proof of the perturbation results in \cite{AHMT-II}.
\end{abstract}

\maketitle
\setcounter{tocdepth}{3}
\tableofcontents


\section{Introduction and Main results}

The purpose of this program is to study some perturbation problems for second order divergence form real elliptic operators with bounded measurable coefficients in domains with rough boundaries. Let $\Omega\subset\mathbb{R}^{n+1}$, $n\geq 2$, be an open set and let $Lu=-\div(A\nabla u)$ be a second order divergence form real  elliptic operator 
defined in $\Omega$. Here the coefficient matrix $A=(a_{i,j}(\cdot))_{i,j=1}^{n+1}$ is real (not
necessarily symmetric) and uniformly elliptic, with $a_{i,j}\in L^\infty(\Omega)$,
that is, there exists a constant $\Lambda\geq 1$ such that 
\begin{align}
\label{uniformlyelliptic}
\Lambda^{-1} |\xi|^{2} \leq A(X) \xi \cdot \xi,
\qquad\qquad
|A(X) \xi \cdot\eta|\leq \Lambda |\xi|\,|\eta| 
\end{align}
for all $\xi,\eta \in\mathbb{R}^{n+1}$ and for almost every $X\in\Omega$. Associated with $L$ one can construct a family of positive Borel measures  $\{\omega_{L}^{X}\}_{X\in\Omega}$, defined on $\partial\Omega$ with $\omega^X(\pom)\le 1$ for every $X\in\Omega$, so that for each $f\in C_c(\pom)$ one can define its associated weak-solution
\begin{equation}\label{hm-sols}
u(X)=\int_{\partial\Omega} f(z)d\omega_{L}^{X}(z),\quad \mbox{whenever}\, \, X\in\Omega,
\end{equation}
which satisfies $Lu=0$ in $\Omega$ in the weak sense. 
 In principle, unless we assume some further condition, $u$ need not 
be continuous all the way to the boundary, 
but still we think of $u$ as the solution to the continuous Dirichlet problem with boundary data $f$. 
We call $\omega_{L}^{X}$ the elliptic measure of $\Omega$ associated with the operator $L$ with pole at $X\in\Omega$.  For convenience,  we will sometimes write $\omega_{L}$ and call it simply the elliptic measure, dropping the dependence on the  pole. 

Given two such operators $L_0u=-\div(A_0 \nabla u)$ and $Lu=-\div(A \nabla u)$, one 
may wonder whether one can find conditions on the matrices $A_0$ and $A$ so that  
some ``good estimates'' for the Dirichlet problem or for the elliptic measure for $L_0$ might be 
transferred to the operator $L$. Similarly, one may 
try to see whether $A$ being ``close'' to $A_0$ in some sense gives some relationship between $\omega_{L}$ and $\omega_{L_0}$. In this direction, a celebrated result of Littman, Stampacchia, 
and Weinberger in \cite{LSW} states that the continuous Dirichlet problem for the Laplace operator $L_0=\Delta$, (i.e., $A_0$ is the identity) is solvable 
if and only if it is solvable for any real elliptic operator $L$. By solvability here we mean that the elliptic measure solutions as in \eqref{hm-sols} are indeed continuous in $\overline{\Omega}$. It is well known that solvability in this sense is in fact equivalent to the fact that all boundary points are regular in the sense of Wiener, a 
condition which entails some capacitary thickness of the complement of $\Omega$. Note that, for this result, one does not need to know that $L$ is ``close'' to the Laplacian in any sense (other than the fact that both operators are uniformly elliptic).

On the other hand, if $\Omega=\mathbb{R}^2_+$ is the upper-half plane and $L_0=\Delta$, 
then the harmonic measure associated with $\Delta$ 
 is mutually absolutely continuous with respect to the 
surface measure on the boundary, and its Radon-Nykodym derivative
is the classical Poisson kernel. 
However, Caffarelli, Fabes, and Kenig in \cite{CFK} constructed a uniformly real  
elliptic operator $L$ in the plane (the pullback of the Laplacian via a quasiconformal mapping of the upper half plane to itself) for which the associated elliptic measure $\omega_L$ is not even 
absolutely continuous with respect to the surface measure (see also \cite{MM} for another example). Hence, in principle the ``good behavior'' of harmonic measure does not always transfer to any elliptic measure even in a nice domain such as the upper-half plane. Consequently, it is natural to see if those good properties can be transferred by assuming some conditions reflecting the fact that  $L$ is ``close'' to $L_0$ or, in other words, by imposing some conditions on the disagreement of $A$ and $A_0$.

The goal of this program is to solve some perturbation problems that 
go beyond \cite{D, F, FKP, MPT, CHM,CHMT}. Our setting is that of 
1-sided NTA domains satisfying the so called capacity density condition (CDC for short), see 
Section \ref{section:prelim} for the precise definitions. The latter is a quantitative version of the well-known Wiener criterion and it is weaker than the Ahlfors regularity of the boundary or the existence of exterior 
Corkscrews (see Definition \ref{def1.cork}). 
This setting guarantees among other things that any elliptic measure is doubling in some appropriate sense, hence one can see that a suitable portion of the boundary of the domain endowed with the Euclidean distance and with a given elliptic measure $\omega_{L_0}$ is a space of homogeneous type. In particular, classes like $A_\infty(\omega_{L_0})$ or $RH_p(\omega_{L_0})$ have the same good features of the corresponding ones in the Euclidean setting. However, our assumptions do not guarantee that the surface measure has any 
good behavior and it
could even be locally infinite. 
In one of our main results, we
consider the case in which a certain disagreement 
condition, originating in \cite{FKP}, holds either with small or large  constant. 
The small constant case can be seen as an extension of \cite{FKP, MPT} to a setting in 
which surface measure is not a good object. The large constant case is 
new even in nice domains such as balls, upper-half spaces, Lipschitz 
domains or chord-arc domains. To the best of our knowledge,  our work is
the first to establish perturbation results on sets with bad 
surface measures, and our large perturbation results are the first of their type. 
Finally, we do not require the operators to be symmetric.  
The precise results, along with its context in the historical developments, 
will be stated in the sequel to the present paper \cite{AHMT-II}.

In the present article we develop some of the needed tools, and present
some other results which are of independent interest. 
Key to our argument is the construction of certain 
sawtooth domains adapted to
a dyadic grid on the boundary and to the Whitney decomposition of the domain. These domains are shown to inherit the main geometrical/topological features of the original domain (see Proposition~\ref{prop:CDC-inherit}). With this in hand we obtain a discrete sawtooth lemma for projections improving \cite[Main Lemma]{DJK}, see Lemma \ref{lemma:DJK-sawtooth} and Lemma \ref{lemma:DJK-sawtooth:new}. These ingredients are crucial for the main results of the papers which we state next.  First we establish that bounded weak-solutions satisfy Carleson measure estimates adapted to the elliptic measure. 

\begin{theorem}\label{THEOR:CME}
Let $\Omega\subset\mathbb{R}^{n+1}$, $n\ge 2$, be a 1-sided NTA domain  \textup{(}cf. Definition \ref{def1.1nta}\textup{)} satisfying the capacity density condition \textup{(}cf. Definition \ref{def-CDC}\textup{)}.  Let $Lu=-\div(A\nabla u)$ be a 
real \textup{(}not necessarily symmetric\textup{)} elliptic operator and write $\omega_L$ and $G_L$ to denote, respectively, the associated elliptic measure and the Green function.  There exists $C$ depending only on dimension $n$, the 1-sided NTA constants, the CDC constant, and the ellipticity constant of $L$, such that for every $u\in W^{1,2}_{\rm loc} (\Omega)\cap L^\infty(\Omega)$ with $Lu=0$ in the weak-sense in $\Omega$ there holds
	\begin{equation}\label{eq:CME}
		\sup_B\sup_{B'} \frac1{\omega_L^{X_{\Delta}}(\Delta')}\iint_{B'\cap\Omega} |\nabla u(X)|^2\, G_L(X_\Delta,X)\,dX
		\le
		\,C\,\|u\|_{L^\infty(\Omega)}^2,
	\end{equation}
	where $\Delta=B\cap\pom$, $\Delta'=B'\cap\pom$, $X_\Delta$ is a corkscrew point relative to $\Delta$ \textup{(}cf. Definition \ref{def1.cork}\textup{)}, and the sups are taken respectively over all balls $B=B(x,r)$ with $x\in \pom$ and $0<r<\diam(\pom)$,
	and $B'=B(x',r')$ with $x'\in2\Delta$ and $0<r'<r c_0/4$, and $c_0$ is the Corkscrew constant.
\end{theorem} 

This result is in turn the main ingredient to 
obtain that the conical square function can be locally 
controlled by the non-tangential maximal function in norm with respect to the elliptic measure, 
allowing us to extend some estimates from \cite{DJK} to our general setting.

\begin{theorem}\label{THEOR:S<N}
Let $\Omega\subset\mathbb{R}^{n+1}$, $n\ge 2$, be a 1-sided NTA domain  \textup{(}cf. Definition \ref{def1.1nta}\textup{)} satisfying the capacity density condition \textup{(}cf. Definition \ref{def-CDC}\textup{)}.  Let $Lu=-\div(A\nabla u)$ be a 
	real \textup{(}non-necessarily symmetric\textup{)} elliptic operator and write $\omega_L$ to denote the associated elliptic measure and the Green function. For every $0<q<\infty$, there exists $C_q$ depending only on dimension $n$, the 1-sided NTA constants, the CDC constant, the ellipticity constant of $L$, and $q$, such that for every $u\in W^{1,2}_{\rm loc} (\Omega)$ with $Lu=0$ in the weak-sense in $\Omega$, for every $Q_0\in\dd$,  there holds 
	\begin{equation}\label{S<N}
		\|\mathcal{S}_{Q_0}u\|_{L^q(Q_0,\omega_L^{X_{Q_0}})}
		\le C_q\,
		\|\mathcal{N}_{Q_0}u\|_{L^q(Q_0,\omega_L^{X_{Q_0}})},
	\end{equation}
where $\mathcal{S}_{Q_0}$ and $\mathcal{N}_{Q_0}$ are the localized  dyadic conical square function and non-tangential maximal function respectively \textup{(}cf. \eqref{def:SF} and \eqref{def:NT}\textup{)}, and $X_{Q_0}$ is a corkscrew point relative to $Q_0$ \textup{(}see Section \ref{subsection:sawtooth}\textup{)}.
\end{theorem}

We note that the estimate \eqref{S<N} is written for the localized dyadic conical square function and non-tangential maximal function. It is not difficult to see that, as a consequence, one can obtain a similar estimate for the regular localized (or truncated) conical square function and non-tangential maximal function with arbitrary apertures (see \cite[Lemma~4.8]{CDMT}), precise statements are left to the interested reader.

The plan of this paper is as follows. Section \ref{section:prelim} presents some of the preliminaries, definitions, and tools which will be used throughout the paper. Section \ref{section:DJK-proj} contains a dyadic version of the main lemma of \cite{DJK}. In Section \ref{section:SF-NT} we prove our main results, Theorem \ref{THEOR:CME} and Theorem \ref{THEOR:S<N}.

We would like to mention that after an initial version of this work was posted on arXiv \cite{AHMT-full}, Feneuil and Poggi in \cite{FP} obtained results related to ours, compare for instance Theorem \ref{THEOR:CME} with \cite[Theorem 1.27]{FP}. Also, the recent work \cite{CDMT} complements this paper and its companion \cite{AHMT-II}, see for instance \cite[Corollary 1.4]{CDMT}.

\section{Preliminaries}
\label{section:prelim}

\subsection{Notation and conventions}

\begin{list}{$\bullet$}{\leftmargin=0.4cm  \itemsep=0.2cm}
	
	\item We use the letters $c,C$ to denote harmless positive constants, not necessarily the same at each occurrence, which depend only on dimension and the
	constants appearing in the hypotheses of the theorems (which we refer to as the ``allowable parameters'').  We shall also sometimes write $a\lesssim b$ and $a \approx b$ to mean, respectively, that $a \leq C b$ and $0< c \leq a/b\leq C$, where the constants $c$ and $C$ are as above, unless
	explicitly noted to the contrary.   Unless otherwise specified upper case constants are greater than $1$  and lower case constants are smaller than $1$. In some occasions it is important to keep track of the dependence on a given parameter $\gamma$, in that case we write $a\lesssim_\gamma b$ or $a\approx_\gamma b$ to emphasize  that the implicit constants in the inequalities depend on $\gamma$.
	
	\item  Our ambient space is $\ree$, $n\ge 2$. 
	
	\item Given $E\subset\ree$ we write $\diam(E)=\sup_{x,y\in E}|x-y|$ to denote its diameter.
	
	\item Given a domain $\Omega \subset \ree$, we shall
	use lower case letters $x,y,z$, etc., to denote points on $\partial \Omega$, and capital letters
	$X,Y,Z$, etc., to denote generic points in $\ree$ (especially those in $\ree\setminus \partial\Omega$).
	
	\item The open $(n+1)$-dimensional Euclidean ball of radius $r$ will be denoted
	$B(x,r)$ when the center $x$ lies on $\partial \Omega$, or $B(X,r)$ when the center
	$X \in \ree\setminus \partial\Omega$.  A {\it surface ball} is denoted
	$\Delta(x,r):= B(x,r) \cap\partial\Omega$, and unless otherwise specified it is implicitly assumed that $x\in\pom$.
	
	\item If $\pom$ is bounded, it is always understood (unless otherwise specified) that all surface balls have radii controlled by the diameter of $\pom$, that is, if $\Delta=\Delta(x,r)$ then $r\lesssim \diam(\pom)$. Note that in this way $\Delta=\pom$ if $\diam(\pom)<r\lesssim \diam(\pom)$.
	
	

	\item For $X \in \ree$, we set $\delta(X):= \dist(X,\partial\Omega)$.
	
	\item We let $\mathcal{H}^n$ denote the $n$-dimensional Hausdorff measure. 
	
	\item For a Borel set $A\subset \ree$, we let $\mathbf{1}_A$ denote the usual
	indicator function of $A$, i.e. $\mathbf{1}_A(X) = 1$ if $X\in A$, and $\mathbf{1}_A(X)= 0$ if $X\notin A$.

	
	
%
	
	\item We shall use the letter $I$ (and sometimes $J$)
	to denote a closed $(n+1)$-dimensional Euclidean cube with sides
	parallel to the coordinate axes, and we let $\ell(I)$ denote the side length of $I$.
	We use $Q$ to denote  dyadic ``cubes''
	on $E$ or $\partial \Omega$.  The
	latter exist as a consequence of Lemma \ref{lemma:dyadiccubes} below.
	
\end{list}

\subsection{Some definitions}\label{ssdefs} 

\begin{definition}[\bf Corkscrew condition]\label{def1.cork}
	Following \cite{JK}, we say that an open set $\Omega\subset \ree$
	satisfies the {\it Corkscrew condition} if for some uniform constant $0<c_0<1$ and
	for every $x\in \partial\Omega$ and $0<r<\diam(\partial\Omega)$, if we write $\Delta:=\Delta(x,r)$, there is a ball
	$B(X_\Delta,c_0r)\subset B(x,r)\cap\Omega$.  The point $X_\Delta\subset \Omega$ is called
	a {\it Corkscrew point relative to} $\Delta$ (or, relative to $B$). We note that  we may allow
	$r<C\diam(\pom)$ for any fixed $C$, simply by adjusting the constant $c_0$.
	We say that $\Omega$ satisfies the {\em exterior Corkscrew condition} if
 $\Omega_{ext}:= \ree\setminus \overline{\Omega}$ 
	satisfies the Corkscrew condition.
\end{definition}


\begin{definition}[\bf Harnack Chain condition]\label{def1.hc}
Again following \cite{JK}, we say
	that $\Omega$ satisfies the {\it Harnack Chain condition} if there are uniform constants $C_1,C_2>1$ such that for every pair of points $X, X'\in \Omega$
	there is a chain of balls $B_1, B_2, \dots, B_N\subset \Omega$ with $N \leq  C_1(2+\log_2^+ \Pi)$ 
	where
	\begin{equation}\label{cond:Lambda}
	\Pi:=\frac{|X-X'|}{\min\{\delta(X), \delta(X')\}}.
	\end{equation} 
such that $X\in B_1$, $X'\in B_N$, $B_k\cap B_{k+1}\neq\emptyset$ and for every $1\le k\le N$
	\begin{equation}\label{preHarnackball}
	C_2^{-1} \diam(B_k) \leq \dist(B_k,\partial\Omega) \leq C_2 \diam(B_k).
	\end{equation}
	The chain of balls is called a {\it Harnack Chain}.
\end{definition}

We note that in the context of the previous definition if $\Pi\le 1$ we can trivially form the Harnack chain $B_1=B(X,3\delta(X)/5)$ and $B_2=B(X', 3\delta(X')/5)$ where \eqref{preHarnackball} holds with $C_2=3$. Hence the Harnack chain condition is non-trivial only when $\Pi> 1$.

\begin{definition}[\bf 1-sided NTA and NTA]\label{def1.1nta}
	We say that a domain $\Omega$ is a {\it 1-sided non-tangentially accessible domain} (1-sided NTA)  if it satisfies both the Corkscrew and Harnack Chain conditions.
	Furthermore, we say that $\Omega$ is a {\it non-tangentially accessible domain}	(NTA  domain)	if it is a 1-sided NTA domain and if, in addition, $\Omega_{\rm ext}:= \ree\setminus \overline{\Omega}$ also satisfies the Corkscrew condition.
\end{definition}
\begin{remark} 
	In the literature, 1-sided NTA domains are also called \textit{uniform domains}. We remark that the 1-sided NTA condition is a quantitative form of path connectedness.
\end{remark}

\begin{definition}[\bf Ahlfors  regular]\label{def1.ADR}
	We say that a closed set $E \subset \ree$ is {\it $n$-dimensional Ahlfors regular} (AR for short) if
	there is some uniform constant $C_1>1$ such that
	\begin{equation} \label{eq1.ADR}
	C_1^{-1}\, r^n \leq \mathcal{H}^n(E\cap B(x,r)) \leq C_1\, r^n,\qquad x\in E, \quad 0<r<\diam(E).
	\end{equation}
\end{definition}

\begin{definition}[\bf 1-sided CAD and CAD]\label{defi:CAD}
	A \emph{1-sided chord-arc domain} (1-sided CAD) is a 1-sided NTA domain with AR boundary.
	A \emph{chord-arc domain} (CAD) is an NTA domain with AR boundary.
\end{definition}

We next recall the definition of the capacity of a set.  Given an open set $D\subset \ree$ (where we recall that we always assume that $n\ge 2$) and a compact set $K\subset D$  we define the capacity of $K$ relative to $D$ as 
\[
\Cap(K, D)=\inf\left\{\iint_{D} |\nabla v(X)|^2 dX:\, \, v\in C^{\infty}_{0}(D),\, v(x)\geq  1 \mbox{ in }K\right\}.
\]

\begin{definition}[\textbf{Capacity density condition}]\label{def-CDC}
	An open set $\Omega$ is said to satisfy the \textit{capacity density condition} (CDC for short) if there exists a uniform constant $c_1>0$ such that
\begin{equation}\label{eqn:CDC}
	\frac{\Cap(\overline{B(x,r)}\setminus \Omega, B(x,2r))}{\Cap(\overline{B(x,r)}, B(x,2r))} \geq c_1
\end{equation}
	for all $x\in \partial\Omega$ and $0<r<\diam(\pom)$.
\end{definition}

The CDC is also known as the uniform 2-fatness as studied by Lewis in \cite{L88}. Using \cite[Example 2.12]{HKM} one has that 
\begin{equation}\label{cap-Ball}
\Cap(\overline{B(x,r)}, B(x,2r))\approx r^{n-1}, \qquad \mbox{for all $x\in\ree$ and $r>0$},
\end{equation}
and hence the CDC is a quantitative version of the Wiener regularity, in particular every $x\in\pom$ is Wiener regular. It is easy to see that the exterior Corkscrew condition implies CDC. Also, it was proved in \cite[Section 3]{Zhao} and \cite[Lemma 3.27]{HLMN} that a set with Ahlfors regular boundary satisfies the capacity density condition with constant $c_1$ depending only on $n$ and the Ahlfors regular constant.

\subsection{Existence of a dyadic grid}\label{ss-dyadic}

In this section we introduce  a dyadic grid along the lines of that obtained in \cite{C}. More precisely,  we will use the dyadic structure from \cite{HK1, HK2}, with a modification from \cite[Proof of Proposition 2.12]{HMMM}:

\begin{lemma}[\textbf{Existence and properties of the ``dyadic grid''}]\label{lemma:dyadiccubes}
Let $E\subset\re^{n+1}$ be a closed set. Then there exists a constant $C\ge 1$ depending just on $n$ such that for each $k\in\mathbb{Z}$ there is a collection of Borel sets  (called ``cubes'')
	$$
	\mathbb{D}_k:=\big\{Q_j^k\subset E:\ j\in\mathfrak{J}_k\big\},
	$$
	where $\mathfrak{J}_k$ denotes some (possibly finite) index set depending on $k$ satisfying:
	\begin{list}{$(\theenumi)$}{\usecounter{enumi}\leftmargin=1cm \labelwidth=1cm \itemsep=0.2cm \topsep=.2cm \renewcommand{\theenumi}{\alph{enumi}}}
		\item $E=\bigcup_{j\in \mathfrak{J}_k}Q_j^k$ for each $k\in\mathbb{Z}$.
		\item If $m\le k$ then either $Q_j^k \subset Q_i^m$ or $Q_i^m\cap Q_j^k=\emptyset$.
		\item For each $k\in\mathbb{Z}$, $j\in\mathfrak{J}_k$, and $m<k$, there is a unique $i\in\mathfrak{J}_m $ such that $Q_j^k\subset Q_i^m$.
		\item For each  $k\in\mathbb{Z}$, $j\in\mathfrak{J}_k$ there is $x_j^k\in E$ such that
		\[B(x_j^k, C^{-1}2^{-k})\cap E\subset Q_j^k \subset B(x_j^k, C 2^{-k})\cap E.\]

	\end{list}
\end{lemma}

\begin{proof}
We first note that $E$ is geometric doubling. That is, there exists $N$ depending just on $n$ such that for every $x\in E$ and $r>0$ one can cover the surface ball $B(x,r)\cap E$ with at most $N$ surface balls of the form $B(x_i,r/2)\cap E$ with $x_i\in E$ ---observe that geometric doubling for $E$ is inherited from the corresponding property on $\ree$ and that is why $N$ depends only on $n$ and it is independent of $E$. Besides, letting $\eta=\frac1{16}$, for every $k\in \ZZ$ it is easy to find a countable collection $\{x_j^k\}_{j\in \mathfrak{J}_k}\subset E$ such that
\[
|x_j^k-x_{j'}^k|\ge \eta^k, \qquad j,j'\in\mathfrak{J}_k,\ j\neq j';
\qquad
\min_{j\in\mathfrak{J}_k} |x-x_{j}|<\eta^k, \qquad\forall\,x\in E.
\]
Invoking then \cite{HK1, HK2} on $E$ with the Euclidean distance and $c_0=C_0=1$ one can construct a family of dyadic cubes associated with these families of points, say $\mathfrak{D}_k$ for $k\in\ZZ$. These satisfy $(a)$--$(d)$ in the statement with the only difference that we have to replace $2^{-k}$ by $\eta^{k}$ in $(d)$. 

At this point we follow the argument in  \cite[Proof of Proposition 2.12]{HMMM} with $\eta=\frac1{16}$. For any $k\in\ZZ$ we set $\dd_j=\mathfrak{D}_{k}$ for every $4k\le j<4(k+1)$. It is straightforward to show that properties $(a)$, $(b)$ and $(c)$ for the families $\dd_k$ follow at once from those for the families $\mathfrak{D}_k$. Regarding  $(d)$, let $Q^i\in \dd_j$ and let $k\in\ZZ$ such that $4k\le j<4(k+1)$ so that $Q^i\in\dd_j=\mathfrak{D}_k$. Writing $x^i\in E$ for the corresponding point associated with $Q^i\in\mathfrak{D}_k$ and invoking $(d)$ for $\mathfrak{D}_k$ we conclude
\[
B(x^i, {C}^{-1} 2^{-j})\cap E
\subset
B(x^i, C^{-1} \eta^{k})\cap E
\subset
Q^i
\subset 
B(x^i, C\eta ^{k})\cap E
\subset 
B(x^i, 16C 2^{-j})\cap E,
\]
hence $(d)$ holds.  
\end{proof}

A few remarks are in order concerning this lemma.  Note that by construction,  within the same generation (that is, within each $\dd_k$) the cubes are pairwise disjoint (hence, there are no repetitions). On the other hand, repetitions are allowed in the different generations, that is, one could have that $Q\in\dd_k$ and $Q'\in\dd_{k-1}$ agree. Then, although $Q$ and $Q'$ are the same set,  as cubes we understand that they are different. In short, it is then understood that $\dd$ is an indexed collection of sets where repetitions of sets are allowed in the different generations but not within the same generation. With this in mind, we can give a proper definition of the ``length'' of a cube (this concept has no geometric meaning for the moment). For every $Q\in\mathbb{D}_k$, we set $\ell(Q)=2^{-k}$, which is called the ``length'' of $Q$. Note that the ``length'' is well defined when considered on $\dd$, but it is not well-defined on the family of sets induced by $\dd$. It is important to observe that the ``length'' refers to the way the cubes are organized in the dyadic grid and in general may not have a geometrical meaning.  It is clear from $(d)$ that $\diam(Q)\lesssim \ell(Q)$ (we will see below that in our setting the converse hold, see Remark  \ref{remark:diam-radius}).

Let us observe that the generations run for all $k\in\ZZ$. However, as we are about to see, sometimes it is natural to truncate the generations. If $E$ is bounded and $k\in\ZZ$ is such that $\diam(E)<C^{-1}2^{-k}$, then there cannot be two distinct cubes in $\dd_k$. Thus, $\dd_k=\{Q^k\}$ with $Q^k=E$. 
Therefore, we are going to ignore those $k\in\mathbb{Z}$ such that $2^{-k}\gtrsim\diam(E)$. Hence, we shall denote by $\mathbb{D}(E)$ the collection of all relevant $Q_j^k$, i.e., $\mathbb{D}(E):=\bigcup_k\mathbb{D}_k$, where, if $\diam(E)$ is finite, the union runs over those $k\in\mathbb{Z}$ such that $2^{-k}\lesssim\diam(E)$. 

In what follows given $B=B(x,r)$ with $x\in E$ we will denote $\Delta=\Delta(x,r)=B\cap E$. We write $\Xi=2C^2$, with $C$ being the constant in Lemma \ref{lemma:dyadiccubes}, which is a purely dimensional. For $Q\in\mathbb{D}(E)$ we will set $k(Q)=k$ if $Q\in\mathbb{D}_k$. Property $(d)$ implies that for each cube $Q\in\mathbb{D}$, there exist $x_Q\in E$ and $r_Q$, with $\Xi^{-1}\ell(Q)\leq r_Q\leq\ell(Q)$ (indeed $r_Q= (2C)^{-1}\ell(Q)$), such that
    \begin{equation}\label{deltaQ}
    \Delta(x_Q,2r_Q)\subset Q\subset\Delta(x_Q,\Xi r_Q).
    \end{equation} 
    We shall denote these balls and surface balls by
    \begin{equation}\label{deltaQ2}
    B_Q:=B(x_Q,r_Q),\qquad\Delta_Q:=\Delta(x_Q,r_Q),
    \end{equation}
    \begin{equation}\label{deltaQ3}
    \widetilde{B}_Q:=B(x_Q,\Xi r_Q),\qquad\widetilde{\Delta}_Q:=\Delta(x_Q,\Xi r_Q),
    \end{equation}
    and we shall refer to the point $x_Q$ as the ``center'' of $Q$.
    
Let $Q\in\dd_k$ and consider  the family of its dyadic children $\{Q'\in \dd_{k+1}: Q'\subset Q\}$. Note that for any two distinct children $Q', Q''$, one has $|x_{Q'}-x_{Q''}|\ge r_{Q'}=r_{Q''}=r_Q/2$, otherwise $x_{Q''}\in Q''\cap \Delta_{Q'}\subset Q''\cap Q'$, contradicting the fact that $Q'$ and $Q''$ are disjoint. Also $x_{Q'}, x_{Q''}\in Q\subset \Delta(x_Q,r_Q)$, hence by the geometric doubling property we have a purely dimensional bound for the number of such $x_{Q'}$ and hence the number of dyadic children of a given dyadic cube is uniformly bounded.

\begin{lemma}\label{lemma:thin-boundaries}
Let $E\subset\ree$ be a closed set and let $\dd(E)$ be the dyadic grid as in Lemma \ref{lemma:dyadiccubes}. Assume that there is a Borel measure $\mu$ which is doubling, that is, there exists $C_\mu\ge 1$ such that $\mu(\Delta(x,2r))\le C_\mu  \mu(\Delta(x,r))$ for every $x\in E$ and $r>0$. Then $\mu(\partial Q)=0$ for every $Q\in\dd(E)$. Moreover, there exist $0<\tau_0<1$, $C_1$, and $\eta>0$ depending only on dimension and $C_\mu$ such that for every $\tau\in(0,\tau_0)$ and $Q\in\dd(E)$
\begin{equation}\label{eqn:thin-boundary}
\mu \big(\big\{x\in Q:\,\dist(x,E\setminus Q)\leq\tau \ell(Q)\big\}\big)
\leq C_1
\tau^\eta \mu(Q).
\end{equation}
\end{lemma}    

\begin{proof}
The argument is a refinement of that in \cite[Proposition 6.3]{HM1} (see also \cite[p. 403]{GR} where the Euclidean case was
treated).
Fix an integer $k$, a cube $Q\in \dd_k$,
and a positive integer $m$ to be chosen.   Fix $\tau>0$ small enough to be chosen and write
\[
\Sigma_\tau
=
\big\{x\in \overline{Q}: \dist(x,E\setminus Q)<\tau \ell(Q)
\big\}.
\]
We set $$\{Q^1_i\}:=\dd^1:=\dd_Q\cap \dd_{k+m}\,,$$
and make the
disjoint decomposition
$Q=\bigcup Q^1_i.$  We then split $\dd^1=\dd^{1,1} \cup \dd^{1,2}$, where
$Q^1_i\in \dd^{1,1}$ if $\overline{Q_i^1}$ meets $\Sigma_\tau$, and
$Q^1_i\in \dd^{1,2}$ otherwise.  We then write $\overline{Q}=R^{1,1}\cup R^{1,2}$, where
\[
R^{1,1}:= \bigcup_{\dd^{1,1}} \widehat{Q}^1_i\,,
\qquad R^{1,2}:=
\bigcup_{\dd^{1,2}} Q^1_i,
\]
and for each cube $Q^1_i\in\dd^{1,1}$, we construct $\widehat{Q}^1_i$ as follows.
We enumerate the elements in $\dd^{1,1}$ as $Q^1_{i_1},Q^1_{i_2},\dots,Q^1_{i_N}$, and then set $(Q^1_i)^*=Q^1_i\cup(\partial Q^1_i\cap \partial Q)$ and
\[
\widehat{Q}^1_{i_1}:=(Q^1_{i_1})^*,
\quad
\widehat{Q}^1_{i_2}:=(Q^1_{i_2})^*\setminus(Q^1_{i_1})^*,
\quad
\widehat{Q}^1_{i_3}:=(Q^1_{i_3})^*\setminus((Q^1_{i_1})^*\cup (Q^1_{i_2})^*),\ \dots
\]
so that $R^{1,1}$ covers $\Sigma_\tau$ and the modified cubes $\widehat{Q}^1_i$
are pairwise disjoint. 

We also note from \eqref{deltaQ} that if $2^{-m}<\Xi^{-2}/4$ then
\[
\dist\big(\Delta_Q,E\setminus Q\Big) \geq r_Q \ge \Xi^{-1}\ell(Q),
\qquad
\diam(Q^1_i)\le 2\Xi r_{Q^1_i}\le  2\Xi\ell(Q^1_i)<\frac{\Xi^{-1}}{2}\ell(Q).
\]
Then $R^{1,1}$ misses $\Delta_Q$ provided $\tau<\Xi^{-1}/2$. Otherwise, we can find $x\in \overline{Q^1_i}\cap\Delta_Q$ with $Q^1_i\in \dd^{1,1}$. 
The latter implies that there is $y\in \overline{Q^1_i}\cap \Sigma_\tau$. All these yield a contradiction:
\begin{align*}
\Xi^{-1}\ell(Q)
\le
\dist\big(\Delta_Q,E\setminus Q\Big)
\le
|x-y|+
\dist\big(y,E\setminus Q\big)
\le
\diam (\overline{Q^1_i})+\tau\ell(Q)
<
\Xi^{-1}\ell(Q).
\end{align*}
Consequently, by the doubling property,
\[
\mu(\overline{Q})
\le
\mu(2\widetilde{\Delta}_Q)
\leq 
C_{\mu}'\,\mu(\Delta_Q)
\leq 
C_{\mu}'\,\mu(R^{1,2}).
\]
Since $R^{1,1}$ and $R^{1,2}$ are disjoint, the latter estimate yields
$$\mu(R^{1,1})\leq \Big(1-\frac1{C_{\mu}'}\Big)\,\mu(\overline{Q})=:\theta\,\mu(\overline{Q}),$$
where we note that $0<\theta<1$.

Let us now repeat this procedure, decomposing $\widehat{Q}^1_i$ for each $Q_i^1\in\dd^{1,1}$. We set $\dd^2(Q^1_i)=\dd_{Q^1_i}\cap \dd_{k+2m}$ and split it into $\dd^{2,1}(Q^1_i)$ and $\dd^{2,2}(Q^1_i)$ where $Q'\in \dd^{2,1}(Q^1_i)$ if
$\overline{Q'}$ meets $\Sigma_\tau$. Associated to any $Q'\in \dd^{2,1}(Q^1_i)$ we set $(Q')^*=(Q'\cap \widehat{Q}^1_i)\cup(\partial Q'\cap (\partial Q\cap \widehat{Q}^1_i))$. Then we make these sets disjoint as before and we have that $R^{2,1}(Q^1_i)$ is defined as the disjoint union of the corresponding $\widehat{Q'}$. Note that $\widehat{Q}^1_i= R^{2,1}(Q^1_i)\cup R^{2,2}(Q^1_i)$ and this is a disjoint union. As before, $R^{2,1}(Q^1_i)$ misses
$\Delta_{Q_i^1}$ provided $\tau< 2^{-m} \Xi^{-1}/2$
so that  by the doubling property
\[
\mu(\widehat{Q}^1_i)
\le
\mu(2\widetilde{\Delta}_{Q^1_i})
\leq 
C_{\mu}'\,\mu(\Delta_{Q_i^1})
\leq 
C_{\mu}'\,\mu(R^{2,2}(Q^1_i))
\]
and then $\mu(R^{2,1}(Q^1_i))\leq \theta\,\mu(\widehat{Q}^1_i).$ Next we set $R^{2,1}$ and $R^{2,2}$ as the union of the corresponding $R^{2,1}(Q^1_i)$ and $R^{2,2}(Q^1_i)$ with $Q_i^1\in\dd^{1,1}$. Then,
\begin{multline*}
\mu(R^{2,1})
:=
\mu \Big(\bigcup_{Q_i^1\in\dd^{1,1}} R^{2,1}(Q^1_i)\Big)
=
\sum_{Q_i^1\in\dd^{1,1}} \mu\big(R^{2,1}(Q^1_i)\big)
\\
\leq
\theta
\sum_{Q_i^1\in\dd^{1,1}} \mu(\widehat{Q}^1_i)
=
\theta\,\mu(R^{1,1})
\le
\theta^2\,\mu(\overline{Q}).
\end{multline*}

Iterating this procedure we obtain that for every $k=0,1,\dots$, if $\tau< 2^{-km} \Xi^{-1}/2$ then $\mu(R^{k+1,1})\le \theta^{k+1}\mu(\overline{Q})$. Let us see that this leads to the desired estimates. Fix $\tau<\Xi^{-1}/2$ and  find $k\ge 0$ such that $2^{-(k+1)m} \Xi^{-1}/2\le  \tau< 2^{-km} \Xi^{-1}/2 $. By construction $\Sigma_\tau\subset R^{k+1,1}$ and then
\[
\mu(\Sigma_\tau)
\le
\mu (R^{k+1,1})
\le
\theta^{k+1}\mu(\overline{Q})
\le
(2\Xi)^{\frac{\log_2 \theta^{-1}}{m}}\,
\tau^{\frac{\log_2 \theta^{-1}}{m}}\mu(\overline{Q}),
\]
which easily gives \eqref{eqn:thin-boundary} with $C_1=(2\Xi)^{\frac{\log_2 \theta^{-1}}{m}}$ and $\eta={\frac{\log_2 \theta^{-1}}{m}}$. On the other hand, note that
\[
\partial Q\subset \bigcap_{j: 2^{-j}<\Xi^{-1}/2} \Sigma_{2^{-j}},
\]
also $\Sigma_{2^{-(j+1)}}\subset \Sigma_{2^{-j}}$. Thus clearly,
\[
0\le \mu(\partial Q)
\le
\lim_{j\to\infty}  \mu(\Sigma_{2^{-j}})
\le
\lim_{j\to\infty}  C_1 2^{-j\eta}\mu(Q)=0,
\]
yielding that $\mu(\partial Q) = 0$.   
\end{proof}
    
\begin{remark}\label{remark:thin-boundaries} 
Note that the previous argument is local in the sense that if we just want to obtain the desired estimates for a fixed $Q_0$ we would only need to assume that $\mu$ is doubling in $2\widetilde{\Delta}_{Q_0}$. Indeed we would just need to know that $\mu(\Delta(x,2r))\le C\, \mu(\Delta(x,r))$ for every $x\in Q_0$ and $0<r<\Xi \ell(Q_0)$, and the involved constants in the resulting estimates will depend only on dimension and $C_\mu$. Further details are left to the interested reader.
\end{remark}

We next introduce the ``discretized Carleson region'' relative to $Q$, $\mathbb{D}_{Q}=\{Q'\in\dd:Q'\subset Q\}$. Let $\mathcal{F}=\{Q_i\}\subset\mathbb{D}$ be a family of pairwise disjoint cubes. The ``global discretized sawtooth'' relative to $\mathcal{F}$ is the collection of cubes $Q\in\mathbb{D}$ that are not contained in any $Q_i\in\mathcal{F}$, that is,
$$
\mathbb{D}_\mathcal{F}:=\mathbb{D}\setminus\bigcup_{Q_i\in\mathcal{F}}\mathbb{D}_{Q_i}.
$$
For a given $Q\in\mathbb{D}$, the ``local discretized sawtooth'' relative to $\mathcal{F}$ is the collection of cubes in $\mathbb{D}_Q$ that are not contained in any $Q_i\in\mathcal{F}$ or, equivalently,
$$
\mathbb{D}_{\mathcal{F},Q}:=\mathbb{D}_{Q}\setminus\bigcup_{Q_i\in\mathcal{F}}\mathbb{D}_{Q_i}=\mathbb{D}_\mathcal{F}\cap\mathbb{D}_Q.
$$
We also allow $\F$ to be the null set in which case $\mathbb{D}_{\tinyemptyset}=\dd$ and $\mathbb{D}_{\tinyemptyset,Q}=\dd_Q$.

With a slight abuse of notation,  let $Q^0$ be either $E$, and in that case $\dd_{Q^0}:=\dd$, or a fixed cube in $\dd$, hence $\dd_{Q^0}$ is the family of dyadic subcubes of $Q^0$. Let $\mu$ be a non-negative Borel measure on $Q^0$ so that $0<\mu(Q)<\infty$ for every $Q\in\mathbb{D}_{Q^0}$.  For the rest of the section we will be working with $\mu$ which is dyadically doubling in $Q^0$. This means that there exists $C_\mu$ such that
$\mu(Q)\le C_\mu \mu(Q')$ for every $Q, Q'\in\dd_{Q^0}$ with $\ell(Q)=2\ell(Q')$.

\begin{definition}[$A_{\infty}^{\rm dyadic}$]\label{def:Ainfty-dyadic}
	
	Given $Q^0$ and $\mu$, a non-negative dyadically doubling measure in $Q^0$, a non-negative Borel measure $\nu$ defined on $Q^0$ is said to belong to $A_\infty^{\rm dyadic}(Q^0,\mu)$ if there exist constants $0<\alpha,\beta<1$ such that for every $Q\in\mathbb{D}_{Q^0}$ and for every Borel set $F\subset Q$, we have that
	\begin{equation}\label{cond-Ainfty-dyadic}
		\frac{\mu(F)}{\mu(Q)}>\alpha
		\qquad\implies\qquad
		\frac{\nu(F)}{\nu(Q)}>\beta.
	\end{equation}
\end{definition}

It is well known (see \cite{CF-1974,GR}) that since $\mu$ is a dyadically doubling measure in $Q^0$, $\nu\in A_\infty^{\rm dyadic}(Q^0,\mu)$ if and only if $\nu\ll\mu$ in $Q^0$  and there exists $1<p<\infty$ such that $\nu\in RH_p^{\rm dyadic}(Q^0,\mu)$, that is, there is a constant $C\ge 1$ such that
$$
\bigg(\aver{Q} k(x)^{p}\,d\mu (x)\bigg)^{\frac{1}{p}}
\leq
C\aver{Q} k(x)\,d\mu(x)
= 
C\,
\frac{\nu(Q)}{\mu(Q)},
$$
for every $Q\in\dd_{Q^0}$, and where $k=d\nu/d\mu$ is the Radon-Nikodym derivative.

For each $\mathcal{F}=\{Q_i\}\subset\mathbb{D}_{Q^0}$, a family of pairwise disjoint dyadic cubes, and each $f\in L^1_{\rm loc}(\mu)$, we define the projection operator
$$
\mathcal{P}_{\mathcal{F}}^\mu f(x)
=
f(x)\mathbf{1}_{E\setminus(\bigcup_{Q_i\in\mathcal{F}} Q_i)}(x)+\sum_{Q_i\in\mathcal{F}}\Big(\aver{Q_i}f(y)\,d\mu(y)\Big)\mathbf{1}_{Q_i}(x).
$$
If $\nu$ is a non-negative Borel measure on $Q^0$, we may naturally then define the measure $\mathcal{P}_\mathcal{F}^\mu\nu$ as $\mathcal{P}_{\mathcal{F}}^\mu\nu(F)=\int_{E}\mathcal{P}_{\mathcal{F}}^\mu\mathbf{1}_F\,d\nu$, that is,
\begin{equation}\label{defprojection}
	\mathcal{P}_{\mathcal{F}}^\mu\nu(F)=\nu\Big(F\setminus\bigcup_{Q_i\in\mathcal{F}}Q_i\Big)+\sum_{Q_i\in\mathcal{F}}\frac{\mu(F\cap Q_i)}{\mu(Q_i)}\nu(Q_i),
\end{equation}
for each Borel set $F\subset Q^0$.

\subsection{Sawtooth domains}\label{subsection:sawtooth}

In the sequel, $\Omega\subset\re^{n+1}$, $n\geq 2$, will be a 1-sided NTA domain satisfying the CDC. Write $\dd=\dd(\pom)$ for the dyadic grid obtained from Lemma \ref{lemma:dyadiccubes} with $E=\pom$. In Remark \ref{remark:diam-radius} below we shall show that under the present assumptions one has that $\diam(\Delta)\approx r_{\Delta}$ for every surface ball $\Delta$. In particular $\diam(Q)\approx\ell(Q)$ for every $Q\in\dd$ in view of \eqref{deltaQ}. Given $Q\in\mathbb{D}$ we define the ``Corkscrew point relative to $Q$'' as $X_Q:=X_{\Delta_Q}$. We note that
    $$
    \delta(X_Q)\approx\dist(X_Q,Q)\approx\diam(Q).
    $$

As done above, given $Q\in\mathbb{D}$ and $\F$ a possibly empty family of pairwise disjoint dyadic cubes, we can define $\mathbb{D}_Q$, the ``discretized Carleson region''; $\mathbb{D}_\mathcal{F}$, the  ``global discretized sawtooth'' relative to $\mathcal{F}$; and $\mathbb{D}_{\mathcal{F},Q}$, the ``local discretized sawtooth'' relative to $\mathcal{F}$. Note that if $\F$ to be the null set in which case $\mathbb{D}_{\tinyemptyset}=\dd$ and $\mathbb{D}_{\tinyemptyset,Q}=\dd_Q$. We would like to introduce the ``geometric'' Carleson regions and sawtooths.  

Let $\mathcal{W}=\mathcal{W}(\Omega)$ denote a collection of (closed) dyadic Whitney cubes of $\Omega\subset\re^{n+1}$, so that the cubes in $\mathcal{W}$  
form a covering of $\Omega$ with non-overlapping interiors, and satisfy
\begin{equation}\label{constwhitney}
4\diam(I)\leq\dist(4I,\partial\Omega)\leq\dist(I,\partial\Omega)\leq 40\diam(I),\qquad\forall I\in\mathcal{W},
\end{equation}
and
$$
\diam(I_1)\approx\diam(I_2),\,\text{ whenever }I_1\text{ and }I_2\text{ touch}.
$$
Let $X(I)$ denote the center of $I$, let $\ell(I)$ denote the side length of $I$, and write $k=k_I$ if $\ell(I)=2^{-k}$.

Given $0<\lambda<1$ and $I\in\mathcal{W}$ we write $I^*=(1+\lambda)I$ for the ``fattening'' of $I$. By taking $\lambda$ small enough, we can arrange matters, so that, first, $\dist(I^*,J^*)\approx\dist(I,J)$ for every $I,J\in\mathcal{W}$. Secondly, $I^*$ meets $J^*$ if and only if $\partial I$ meets $\partial J$ (the fattening thus ensures overlap of $I^*$ and $J^*$ for any pair $I,J\in\mathcal{W}$ whose boundaries touch, so that the Harnack Chain property then holds locally in $I^*\cup J^*$, with constants depending upon $\lambda$). By picking $\lambda$ sufficiently small, say $0<\lambda<\lambda_0$, we may also suppose that there is $\tau\in(\frac12,1)$ such that for distinct $I,J\in\mathcal{W}$, we have that $\tau J\cap I^*=\emptyset$. In what follows we will need to work with dilations $I^{**}=(1+2\lambda)I$ or $I^{***}=(1+4\lambda)I$, and in order to ensure that the same properties hold we further assume that $0<\lambda<\lambda_0/4$.

For every $Q\in\mathbb{D}$ we can construct a family $\mathcal{W}_Q^*\subset\mathcal{W}(\Omega)$, and define
$$
U_Q:=\bigcup_{I\in\mathcal{W}_Q^*}I^*,
$$
satisfying the following properties: $X_Q\in U_Q$ and there are uniform constants $k^*$ and $K_0$ such that
\begin{align}
\label{kstar_K0}
\begin{split}
k(Q)-k^*\leq k_I\leq k(Q)+k^*,\quad\forall I\in\mathcal{W}_Q^*,
\\[4pt]
X(I)\rightarrow_{U_Q} X_Q,\quad\forall I\in\mathcal{W}_Q^*,
\\[4pt]
\dist(I,Q)\leq K_0 2^{-k(Q)},\quad\forall I\in\mathcal{W}_Q^*.
\end{split}
\end{align}
Here, $X(I)\rightarrow_{U_Q} X_Q$ means that the interior of $U_Q$ contains all balls in a Harnack Chain (in $\Omega$) connecting $X(I)$ to $X_Q$, and moreover, for any point $Z$ contained in any ball in the Harnack Chain, we have $\dist(Z,\partial\Omega)\approx\dist(Z,\Omega\setminus U_Q)$ with uniform control of the implicit constants. The constants $k^*, K_0$ and the implicit constants in the condition $X(I)\rightarrow_{U_Q} X_Q$, depend on the allowable parameters and on $\lambda$. Moreover, given $I\in\mathcal{W}(\Omega)$ we have that $I\in\mathcal{W}_{Q_I}^*$, where $Q_I\in\dd$ satisfies $\ell(Q_I)=\ell(I)$, and contains any fixed $\widehat{y}\in\partial\Omega$ such that $\dist(I,\partial\Omega)=\dist(I,\widehat{y})$. The reader is referred to \cite{HM1, HMT1} for full details.

For a given $Q\in\mathbb{D}$, the ``Carleson box'' relative to $Q$ is defined by
$$
T_Q:=\interior\bigg(\bigcup_{Q'\in\mathbb{D}_Q}U_{Q'}\bigg).
$$
For a given family $\mathcal{F}=\{Q_i\}\subset\dd$ of pairwise disjoint cubes and a given $Q\in\mathbb{D}$, we define the ``local sawtooth region'' relative to $\mathcal{F}$ by
\begin{equation}
\label{defomegafq}
\Omega_{\mathcal{F},Q}=\interior\bigg(\bigcup_{Q'\in\mathbb{D}_{\mathcal{F},Q}}U_{Q'}\bigg)=\interior\bigg(\bigcup_{I\in\mathcal{W}_{\mathcal{F},Q}}I^*\bigg),
\end{equation}
where $\mathcal{W}_{\mathcal{F},Q}:=\bigcup_{Q'\in\mathbb{D}_{\mathcal{F},Q}}\mathcal{W}_Q^*$. Note that in the previous definition we may allow $\F$ to be empty in which case clearly $\Omega_{\tinyemptyset ,Q}=T_Q$. Similarly, the ``global sawtooth region'' relative to $\mathcal{F}$ is defined as
\begin{equation}
\label{defomegafq-global}
\Omega_{\mathcal{F}}=\interior\bigg(\bigcup_{Q'\in\mathbb{D}_{\mathcal{F}}}U_{Q'}\bigg)=\interior\bigg(\bigcup_{I\in\mathcal{W}_{\mathcal{F}}}I^*\bigg),
\end{equation}
where $\mathcal{W}_{\mathcal{F}}:=\bigcup_{Q'\in\mathbb{D}_{\mathcal{F}}}\mathcal{W}_Q^*$. If $\F$ is the empty set clearly $\Omega_{\tinyemptyset}=\Omega$.
For a given $Q\in\dd$  and $x\in \pom$ let us introduce the ``truncated dyadic cone'' 
\[
\Gamma_{Q}(x) := \bigcup_{x\in Q'\in\mathbb{D}_{Q}}  U_{Q'},
\]
where it is understood that $\Gamma_{Q}(x)=\emptyset$ if $x\notin Q$. 
Analogously, we can slightly fatten the Whitney boxes and use $I^{**}$ to define new fattened Whitney regions and sawtooth domains. More precisely, for every $Q\in\dd$,
\[
T_Q^*:=\interior\bigg(\bigcup_{Q'\in\mathbb{D}_Q}U_{Q'}^*\bigg),\quad\Omega^*_{\mathcal{F},Q}:=\interior\bigg(\bigcup_{Q'\in\mathbb{D}_{\F,Q}}U_{Q'}^*\bigg), \quad
\Gamma^*_{Q}(x) := \bigcup_{x\in Q'\in\mathbb{D}_{Q_0}}  U_{Q'}^*
\]
where $U_{Q}^*:=\bigcup_{I\in\mathcal{W}_Q^*}I^{**}$.
Similarly, we can define $T_Q^{**}$, $\Omega^{**}_{\mathcal{F},Q}$, $\Gamma_Q^{**}(x)$, and $U^{**}_{Q}$ by using $I^{***}$ in place of $I^{**}$.

Given $Q$ we next define the ``localized dyadic non-tangential maximal function''
\begin{equation}\label{def:NT}
\mathcal{N}_{Q}u(x) 
: = 
\sup_{Y\in \Gamma^*_{Q}(x)} |u(Y)|,
\qquad x\in \pom,
\end{equation}
for every $u\in C(T_{Q}^*)$, where it is understood that $\mathcal{N}_{Q}u(x)= 0$ for every $x\in\pom\setminus Q$ (since $\Gamma_Q^*(x)=\emptyset$ in such a case). 
Finally, let us introduce the ``localized  dyadic conical square function''
\begin{equation}\label{def:SF}
\mathcal{S}_{Q}u(x):=\bigg(\iint_{\Gamma_{Q}(x)}|\nabla u(Y)|^2\delta(Y)^{1-n}\,dY\bigg)^{\frac12}, \qquad x\in \pom,
\end{equation}
for every $u\in W^{1,2}_{\rm loc}  (T_{Q_0})$. Note that again $\mathcal{S}_{Q}u(x)=0$ for every $x\in\pom\setminus Q$.

To define  the ``Carleson box'' $T_\Delta$ associated with a surface ball $\Delta=\Delta(x,r)$, let $k(\Delta)$ denote the unique $k\in\mathbb{Z}$ such that $2^{-k-1}<200r\leq 2^{-k}$, and set
\begin{equation}\label{D-delta}
\mathbb{D}^{\Delta}:=\big\{Q\in\mathbb{D}_{k(\Delta)}:\:Q\cap 2\Delta\neq\emptyset\big\}.
\end{equation}
We then define
\begin{equation}
\label{def:T-Delta}
T_{\Delta}:=\interior\bigg(\bigcup_{Q\in\mathbb{D}^\Delta}\overline{T_Q}\bigg).
\end{equation}
We can also consider fattened versions of $T_\Delta$ given by
$$
T_{\Delta}^*:=\interior\bigg(\bigcup_{Q\in\mathbb{D}^\Delta}\overline{T_Q^*}\bigg),\qquad T_{\Delta}^{**}:=\interior\bigg(\bigcup_{Q\in\mathbb{D}^\Delta}\overline{T_Q^{**}}\bigg).
$$

Following \cite{HM1, HMT1}, one can easily see that there exist constants $0<\kappa_1<1$ and $\kappa_0\geq 16\Xi$ (with $\Xi$ the constant in \eqref{deltaQ}), depending only on the allowable parameters, so that
\begin{gather}\label{definicionkappa12}
\kappa_1B_Q\cap\Omega\subset T_Q\subset T_Q^*\subset T_Q^{**}\subset \overline{T_Q^{**}}\subset\kappa_0B_Q\cap\overline{\Omega}=:\tfrac{1}{2}B_Q^*\cap\overline{\Omega},
\\[6pt]
\label{definicionkappa0}
\tfrac{5}{4}B_\Delta\cap\Omega\subset T_\Delta\subset T_\Delta^*\subset T_\Delta^{**}\subset\overline{T_\Delta^{**}}\subset\kappa_0B_\Delta\cap\overline{\Omega}=:\tfrac{1}{2}B_\Delta^*\cap\overline{\Omega},
\end{gather}
and also
\begin{equation}\label{propQ0}
Q\subset\kappa_0B_\Delta\cap\partial\Omega=\tfrac{1}{2}B_\Delta^*\cap\partial\Omega=:\tfrac{1}{2}\Delta^*,\qquad\forall\,Q\in\mathbb{D}^{\Delta},
\end{equation}
where $B_Q$ is defined as in \eqref{deltaQ2}, $\Delta=\Delta(x,r)$ with $x\in\partial\Omega$, $0<r<\diam(\partial \Omega)$, and $B_{\Delta}=B(x,r)$ is so that $\Delta=B_\Delta\cap\partial\Omega$. From our choice of the parameters one also has that $B_Q^*\subset B_{Q'}^*$ whenever $Q\subset Q'$.

In the remainder of this section we show that if $\Omega$ is a 1-sided NTA domain satisfying the CDC then Carleson boxes and local and global sawtooth domains are also 1-sided NTA domains satisfying the CDC. We next present some of the properties of the capacity which will be used in our proofs. From the definition of capacity one can easily see that given a ball $B$ and compact sets $F_1\subset F_2\subset \overline{B}$ then
\begin{equation}\label{cap:prop1}
\Cap(F_1, 2B)\le \Cap (F_2, 2B).
\end{equation}
Also, given two balls $B_1\subset B_2$ and a compact set $F\subset \overline{B_1}$ then 
\begin{equation}\label{cap:prop2}
\Cap(F, 2B_2)\le \Cap (F, 2B_1).
\end{equation}
On the other hand, \cite[Lemma 2.16]{HKM} gives that if $F$ is a compact with $F\subset \overline{B}$ then there is a dimensional constant $C_n$ such that
\begin{equation}\label{cap:prop3}
C_n^{-1}\Cap(F, 2B)\le \Cap (F, 4B)\le \Cap (F, 2B).
\end{equation}

\begin{proposition}\label{prop:CDC-inherit}
	Let $\Omega\subset\mathbb{R}^{n+1}$, $n\ge 2$, be a 1-sided NTA domain satisfying the CDC. Then all of its Carleson boxes $T_Q$ and $T_\Delta$, and sawtooth regions $\Omega_\F$, and $\Omega_{\F,Q}$ are 1-sided NTA domains and satisfy the CDC with uniform implicit constants depending only on dimension and on the corresponding
	constants for $\Omega$.
\end{proposition}

\begin{proof}
A careful examination of the proofs in \cite[Appendices A.1-A.2]{HM1} reveals that if $\Omega$ is a 1-sided NTA domain then all Carleson boxes $T_Q$ and $T_\Delta$, and local and global sawtooth domains $\Omega_{\F,Q}$ and $\Omega_\F$ inherit the interior Corkscrew and Harnack chain conditions, hence they are also 1-sided NTA domains. Therefore, we only need to prove the CDC. We are going to consider only the case $\Omega_{\F,Q}$ (which in particular gives the desired property for  $T_{Q}$ by allowing $\F$ to be the null set). The other proofs require minimal changes which are left to the interested reader. To this end, fix $Q\in\dd$ and $\F\subset\dd_{Q}$ a (possibly empty) family of pairwise disjoint dyadic cubes. Let $x\in\partial \Omega_{\F,Q}$ and $0<r<\diam(\Omega_{\F,Q})\approx \ell(Q)$. 
	
	\noindent {\bf Case 1:} $\delta(x)=0$. In that case we have that $x\in\pom$ and we can use that  $\Omega$ satisfies the CDC with constant $c_1$, \eqref{cap:prop1} and the fact that  $\Omega_{\F,Q}\subset\Omega$ to obtain the desired estimate
	\[
	c_1 r^{n-1}
	\lesssim 
	\Cap(\overline{B(x,r)}\setminus \Omega, B(x,2r))
	\le
	\Cap(\overline{B(x,r)}\setminus \Omega_{\F,Q}, B(x,2r)).
	\]

	\medskip
	
	\noindent {\bf Case 2:} $0<\delta(x)<r/M$ with $M$ large enough to be chosen. In this case $x\in\Omega\cap\partial\Omega_{\F,Q}$ and  hence there exist $Q'\in \dd_{\F,Q}$ and $I\in\W^*_{Q'}$ such that $x\in\partial I^*$. Note that by \eqref{kstar_K0}
	\[
	|x-x_{Q'}|
	\le
	\diam(I^*)+\dist(I,Q')+\diam(Q')
	\lesssim
	\ell(Q')
	\approx
	\ell(I)
	\approx
	\delta(x)
	\lesssim\frac{r}{M}.
	\]
	Let $Q''\in\dd_{Q}$ be such that $x_{Q'}\in Q''$ and  $\frac{r}{2M}\le \ell(Q'')<\frac{r}{M}<\ell(Q)$ provided that $M$ is taken large enough. If $Z\in B_{Q''}$ then taking $M$ large enough 
	\[
	|Z-x|
	\le
	|Z-x_{Q''}|+ |x_{Q''}-x_{Q'}|+|x_{Q'}-x|
	\lesssim
	\ell(Q'')+\frac{r}{M}
	\lesssim
	\frac{r}{M}
	<r
	\]
	and $B_{Q''}\subset B(x,r)$. On the other hand, if $Z\in B(x,2r)$, we analogously have provided $M$ is large enough
	\[
	|Z-x_{Q''}|
	\le 
	|Z-x|+ |x-x_{Q'}|+ |x_{Q'}-x_{Q''}|
	<
	2r+C\frac{r}{M}+ \Xi r_{Q''}
	<
	6M\Xi r_{Q''}
	\]
	and thus $B(x,2r)\subset 6M\Xi B_{Q''}$. Once $M$ has been fixed so that the previous estimates hold, we use them in conjunction  with the fact that $\Omega$ satisfies the CDC with constant $c_1$, \eqref{cap:prop1}--\eqref{cap:prop3}, and that $\Omega_{\F,Q}\subset\Omega$ to obtain
	\begin{multline*}
	\frac{c_1}{(2M\Xi)^{n-1}} r^{n-1}
	\le
	c_1 r_{Q''}^{n-1}
	\lesssim 
	\Cap(\overline{B_{Q''}}\setminus \Omega, 2 B_{Q''})
	\lesssim
	\Cap(\overline{B_{Q''}}\setminus \Omega, 6M\Xi B_{Q''})
	\\
	\le 
	\Cap(\overline{B_{Q''}}\setminus \Omega, B(x,2r))
	\le
	\Cap(\overline{B(x,r)}\setminus \Omega_{\F,Q}, B(x,2r)),
	\end{multline*}
	which gives us the desired lower bound in the present case. 
	
	\medskip

	\begin{figure}[!ht]
		\centering
		\begin{tikzpicture}[scale=.6]

		\begin{scope}[shift={(-5.25,0)}]
		\draw plot [smooth] coordinates {(-5,2)(-3,1.6) (-2,1.3)(-1,1) (0,1)(1,1)(2,1.4)(3,1.7)(5,2)};
		
		\draw (-3.8, 4)--(-3.8,2.6)--(-3.4,2.6)--(-3.4, 2.2)--(-3.2, 2.2)--(-3.2,2)--(-3,2)--(-1,1)--(1,1)--(2.6,2.2)--(2.8,2.2)--(2.8, 2.4)--(3,2.4)--(3,2.8)--(3.4,2.8)--(3.4,4)--cycle;

		\node[below] at (-5,2) {$\partial\Omega$};
		\draw[red,line width=0.5mm] (-1,1)--(1,1);
		\node[below, red] at (-1,1) {$Q$};

		\begin{scope}
		\node at (0,3) {$T_Q$};
		
		\draw (.5,1) circle (1cm);
		\node[circle, fill=black, inner sep=.1pt,minimum size=3pt, label=above:{$x$}] (a) at (.5,1) {};
		
		
		
		\begin{scope}
		
		\path[clip] (-3.8, 0)--(-3.8,2.6)--(-3.4,2.6)--(-3.4, 2.2)--(-3.2, 2.2)--(-3.2,2)--(-3,2)--(-1,1)--(1,1)--(2.6,2.2)--(2.8,2.2)--(2.8, 2.4)--(3,2.4)--(3,2.8)--(3.4,2.8)--(3.4,0)--cycle;

		\draw[fill=gray, opacity=.5] (.5,1) circle (1cm);
		
		
		\clip (.5,1) circle (1cm);
		
		\clip (-1,-1) rectangle (3,3) plot [smooth] coordinates {(-5,2)(-3,1.6) (-2,1.3)(-1,1) (0,1)(1,1)(2,1.4)(3,1.7)(5,2)};
		
		\draw[blue,line width=0.5mm] (-1,1)--(1,1);

		\draw[fill=purple,opacity=.5] (.5,1) circle (1cm);
		
		\end{scope}
		
		\end{scope}
		
		\end{scope}

		
		\begin{scope}[shift={(5.25,0)}]
		\draw plot [smooth] coordinates {(-5,2)(-3,1.6) (-2,1.3)(-1,1) (0,1)(1,1)(2,1.4)(3,1.7)(5,2)};
		\node[below] at (-5,2) {$\partial\Omega$};
		\draw[red,line width=0.5mm] (-1,1)--(1,1);
		\node[below, red] at (-1,1) {$Q$};

		\begin{scope}
		\draw (-3.8, 4)--(-3.8,2.6)--(-3.4,2.6)--(-3.4, 2.2)--(-3.2, 2.2)--(-3.2,2)--(-3,2)--(-1,1)--(1,1)--(2.6,2.2)--(2.8,2.2)--(2.8, 2.4)--(3,2.4)--(3,2.8)--(3.4,2.8)--(3.4,4)--cycle;
		
		\node at (0,3) {$T_Q$};
		\draw (1.3,1.25) circle (1.25cm);
		\node[circle, fill=black, inner sep=.1pt,minimum size=3pt, label=above:{$x$}] (a) at (1.3,1.25) {};

		\node[circle, fill=black, inner sep=.1pt,minimum size=3pt] (a) at (.55,1) {};
		
		\node at (1.3,.6) {\tiny $B_{Q''}$};
		
		\draw (.55,1) circle (.35cm);
		
		\begin{scope}
		
		\path[clip] (-3.8, 0)--(-3.8,2.6)--(-3.4,2.6)--(-3.4, 2.2)--(-3.2, 2.2)--(-3.2,2)--(-3,2)--(-1,1)--(1,1)--(2.6,2.2)--(2.8,2.2)--(2.8, 2.4)--(3,2.4)--(3,2.8)--(3.4,2.8)--(3.4,0)--cycle;
		\draw[fill=gray, opacity=.5] (1.3,1.25) circle (1.25cm);
		
		
		\clip (1.3,1.25) circle (1.25cm);
		\draw[blue,line width=0.5mm] (-1.1,1)--(1,1); 
		\draw[fill=purple, opacity=.5] (.55,1) circle (.35cm);

		\end{scope}
		
		\end{scope}
		
		\end{scope}

		\end{tikzpicture}
		\caption{\textbf{Case 1} and \textbf{Case 2} for $T_Q$.}
	\end{figure}
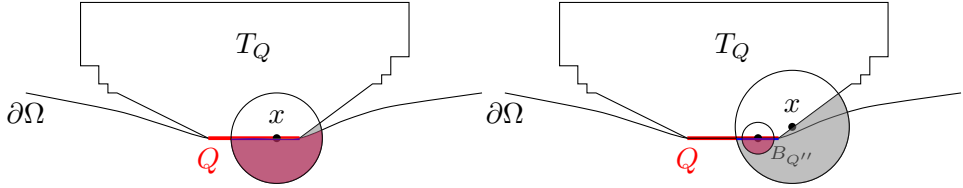

	\noindent {\bf Case 3:} $\delta(x)>r/M$. In this case $x\in\Omega\cap\partial \Omega_{\F,Q}$ and  hence there exists $Q'\in \dd_{\F,Q}$ and $I\in\W^*_{Q'}$ such that $x\in\partial I^*$ and $\interior(I^*)\subset \Omega_{\F,Q}$. Also there exists $J\in\W$, with $J\ni x$ such that $J\notin\W^*_{Q''}$ for any $Q''\in \dd_{\F,Q}$ which implies that $\tau J\subset \Omega\setminus \Omega_{\F,Q}$ for some $\tau\in(\frac12,1)$ (see Section \ref{subsection:sawtooth}). Note that $\ell(I)\approx\ell(J)\approx\delta(x)\gtrsim r$, and more precisely $r/M<\delta(x)<41\diam(J)$ by \eqref{constwhitney}.

	\begin{figure}[!ht]
		\centering
		\begin{tikzpicture}[scale=.7]

		\draw plot [smooth] coordinates {(-5,2)(-3,1.6) (-2,1.3)(-1.5,1) (0,1)(1.5,1)(2,1.4)(3,1.7)(5,2)};
		\node[below] at (-5,2) {$\partial\Omega$};
		\draw[red,line width=0.5mm] (-1.5,1)--(1.5,1);
		\node[below, red] at (-1,1) {$Q$};
		
		\node at (0,3) {$T_Q$};
		
		\node at (-3.7,2.4) {\tiny $B'$};

		\draw (-1.5,1)--(-2,1.5)--(-2,2)--(-2.5,2)--(-2.5,3)--(-3.5,3)--(-3.5,4)--(-4.5,4);
		\draw (1.5,1)--(2,2)--(2.5,2)--(2.5,2.5)--(3,2.5)--(3,3)--(4,3)--(4, 4);

		\node[circle, fill=black, inner sep=.1pt,minimum size=3pt, label=above:{\ $x$}] (a) at (-3,3) {};
		
		\draw (-3,3) circle (1.5cm);
		
		\draw[fill=purple, opacity=.8] (-3,2.5) circle (.33cm);

		
		\draw[red, line width=0.5mm] (-3.5, 3)--(-2.5,3);
		
		
		\begin{scope}
		
		\path[clip] (-1,1)--(-1.5,1)--(-2,2)--(-2.5,2)--(-2.5,3)--(-3.5,3)--(-3.5,4)--(-5,4)--(-5,0)--(-1,0);
		
		\draw[fill=gray, opacity=.5] (-3,3) circle (1.5cm);

		\end{scope}

		\end{tikzpicture}
		\caption{\textbf{Case 3} for $T_Q$.}
	\end{figure}
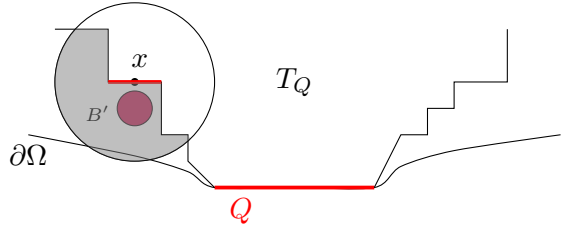

	Let $B'=B(x',s)$ with $s=r/(300M)$ and $x'$ being the point in the segment joining $x$ and the center of $J$ at distance $2s$ from $x$. It is easy to see that $B'\subset B(x,r)\subset B(x,2r)\subset 1000M B'$ and also $\overline{B'}\subset \interior(J)\setminus \Omega_{\F,Q}$. We can then use
	\eqref{cap-Ball} and \eqref{cap:prop1}--\eqref{cap:prop3} to obtain the desired estimate: 
	\begin{multline*}
	\frac1{(300M)^{n-1}} r^{n-1}
	=
	s^{n-1}
	\approx
	\Cap(\overline{B'}, 2 B')
	\lesssim
	\Cap(\overline{B'}, 1000M B')
	\\
	\le 
	\Cap(\overline{B'}, B(x,2r))
	\le
	\Cap(\overline{B(x,r)}\setminus \Omega_{\F,Q}, B(x,2r)).
	\end{multline*}

	Collecting the 3 cases and using \eqref{cap-Ball} we have been able to show that 
	\begin{equation}\label{eqn:CDC-TQ}
	\frac{\Cap(\overline{B(x,r)}\setminus\Omega_{\F,Q}, B(x,2r))}{\Cap(\overline{B(x,r)}, B(x,2r))} \gtrsim 1,
	\qquad \forall\, x\in\partial \Omega_{\F,Q}, \ 0<r<\diam(\Omega_{\F,Q}),
	\end{equation}
	which eventually gives that $\Omega_{\F,Q}$ satisfies the CDC. This completes the proof.
\end{proof}

Our next auxiliary result adapts \cite[Lemma 4.44]{HMT-NTA} to our current setting: 

\begin{lemma}\label{lemma:approx-saw}
	Let $\Omega\subset\mathbb{R}^{n+1}$ be a 1-sided NTA domain satisfying the CDC.	Given $Q_0\in\dd$ and $N\ge 4$
	consider the family of pairwise disjoint cubes $\mathcal{F}_N=\{Q\in\dd_{Q_0}: \ell(Q)=2^{-N}\,\ell(Q_0)\}$ and let $\Omega_N:=\Omega_{\F_N,Q_0}$ and $\Omega_N^*:=\Omega_{\F_N,Q_0}^*$. There exists $\Psi_N\in C_c^\infty(\ree)$ and a constant $C\ge 1$ depending only on dimension $n$, the 1-sided NTA constants, the CDC constant, and independent of $N$ and $Q_0$ such that the following hold:
	
	\begin{list}{$(\theenumi)$}{\usecounter{enumi}\leftmargin=.8cm
			\labelwidth=.8cm\itemsep=0.2cm\topsep=.1cm
			\renewcommand{\theenumi}{\roman{enumi}}}
		
		\item $C^{-1}\,\mathbf{1}_{\Omega_{N}}\le \Psi_N\le \mathbf{1}_{\Omega_{N}^{*}}$.
		
		\item $\sup_{X\in \Omega} |\nabla \Psi_N(X)|\,\delta(X)\le C$.
		
		\item Setting 
		\begin{equation}
			\label{eq:defi-WN}
			\W_N:=\bigcup_{Q\in\dd_{\F_{N},Q_0}} \W_Q^*,
			\quad
			\W_N^\Sigma:=
			\big\{I\in \W_{N}:\, \exists\,J\in \W\setminus \W_N\ \mbox{with}\ \partial I\cap\partial J\neq\emptyset
			\big\}.
		\end{equation}
	\end{list}
	one has 
	\begin{equation}
		\nabla \Psi_N\equiv 0
		\quad
		\mbox{in}
		\quad
		\bigcup_{I\in \W_N \setminus \W_N^\Sigma }I^{**}
		\label{eq:fregtgtr}
	\end{equation}
	and there exists a family $\{\widehat{Q}_I\}_{I\in 	\W_N^\Sigma}$ so that 
	\begin{equation}\label{new-QI}
		C^{-1}\,\ell(I)\le \ell(\widehat{Q}_I)\le C\,\ell(I),
		\qquad
		\dist(I, \widehat{Q}_I)\le C\,\ell(I),
		\qquad
		\sum_{I\in 	\W_N^\Sigma} \mathbf{1}_{\widehat{Q}_I} \le C.
	\end{equation}
\end{lemma}

\begin{proof}
	We proceed as in \cite[Lemma 4.44]{HMT-NTA}. Recall that given $I$, any closed dyadic cube in $\ree$, we set $I^{*}=(1+\lambda)I$ and $I^{**}=(1+2\,\lambda)I$. Let us introduce $\widetilde{I^{*}}=(1+\frac32\,\lambda)I$ so that
	\begin{equation}
		I^{*}
		\subsetneq
		\interior(\widetilde{I^{*}})
		\subsetneq \widetilde{I^{*}}
		\subset
		\interior(I^{**}).
		\label{eq:56y6y6}
	\end{equation}
	Given $I_0:=[-\frac12,\frac12]^{n+1}\subset\ree$, fix $\phi_0\in C_c^\infty(\ree)$ such that
	$1_{I_0^{*}}\le \phi_0\le 1_{\widetilde{I_0^{*}}}$ and $|\nabla \phi_0|\lesssim 1$ (the implicit constant depends on the parameter $\lambda$). For every $I\in \W=\W(\Omega)$ we set $\phi_I(\cdot)=\phi_0\big(\frac{\,\cdot\,-X(I)}{\ell(I)}\big)$ so that $\phi_I\in C^\infty(\ree)$, $1_{I^{*}}\le \phi_I\le 1_{\widetilde{I^{*}}}$ and $
	|\nabla \phi_I|\lesssim \ell(I)^{-1}$ (with implicit constant depending only on $n$ and $\lambda$).
	
	For every $X\in\Omega$, we let $\Phi(X):=\sum_{I\in \W} \phi_I(X)$. It then follows that $\Phi\in C_{\rm loc}^\infty(\Omega)$ since for every compact subset of $\Omega$, the previous sum has finitely many non-vanishing terms. Also, $1\le \Phi(X)\le C_{\lambda}$ for every $X\in \Omega$ since the family $\{\widetilde{I^{*}}\}_{I\in \W}$ has bounded overlap by our choice of $\lambda$. Hence we
	can set $\Phi_I=\phi_I/\Phi$ and one can easily see that $\Phi_I\in C_c^\infty(\ree)$, $C_\lambda^{-1}1_{I^{*}}\le \Phi_I\le 1_{\widetilde{I^{*}}}$ and $
	|\nabla \Phi_I|\lesssim \ell(I)^{-1}$. With this in hand set 
	\[
	\Psi_N(X)
	:=
	\sum_{I\in \W_N} \Phi_I(X)
	=
	\frac{\sum\limits_{I\in \W_N} \phi_I(X)}{\sum\limits_{I\in \W} \phi_I(X)},
	\qquad
	X\in\Omega.
	\]
	We first note that the number of terms in the sum defining $\Psi_N$ is bounded depending on $N$. Indeed, if $Q\in \dd_{\F_N, Q_0}$ then $Q\in \dd_{Q_0}$ and $2^{-N}\ell(Q_0)<\ell(Q)\le \ell(Q_0)$ which implies that $\dd_{\F_N, Q_0}$ has finite cardinality with bounds depending only on dimension and $N$ (here we recall that the number of dyadic children of a given cube is uniformly controlled). Also, by construction $\W_Q^*$ has cardinality depending only on the allowable parameters. Hence, $\# \W_N\lesssim C_N<\infty$. This and the fact that each $\Phi_I\in C_c^\infty(\ree)$ yield that $\Psi_N\in C_c^\infty(\ree)$. Note also that \eqref{eq:56y6y6} and the definition of $\W_N$  give
	\begin{multline*}
		\supp \Psi_I
		\subset
		\bigcup_{I\in \W_N} \widetilde{I^{*}}
		=
		\bigcup_{Q\in\dd_{\F_{N},{Q}_0}}
		\bigcup_{I\in \W_Q^*} \widetilde{I^{*}}
		\subset
		\interior\Big(
		\bigcup_{Q\in\dd_{\F_{N},{Q}_0}}
		\bigcup_{I\in \W_Q^*} I^{**}
		\Big)
		\\
		=
		\interior\Big(
		\bigcup_{Q\in\dd_{\F_{N},{Q}_0}}
		U_Q^{*}
		\Big)
		=
		\Omega_{N}^{*}
	\end{multline*}
	This, the fact that $\W_N\subset \W$, and the definition of $\Psi_N$ immediately give that
	$\Psi_N\le \mathbf{1}_{\Omega_{N}^{*}}$. On the other hand if $X\in \Omega_N=\Omega_{\F_{N},Q_0}$ then the exists $I\in \W_N$ such that $X\in I^{*}$ in which case $\Psi_N(X)\ge \Phi_I(X)\ge C_\lambda^{-1}$. All these imply $(i)$. 
	Note that $(ii)$  follows by observing that for every $X\in \Omega$
	$$
	|\nabla \Psi_N(X)|
	\le
	\sum_{I\in \W_N} |\nabla\Phi_I(X)|
	\lesssim
	\sum_{I\in \W} \ell(I)^{-1}\,1_{\widetilde{I^{*}}}(X)
	\lesssim
	\delta(X)^{-1}
	$$
	where we have used that if $X\in \widetilde{I^{*}}$ then $\delta(X)\approx \ell(I)$ and also that the family $\{\widetilde{I^{*}}\}_{I\in \W}$ has bounded overlap.
	
	To see $(iii)$ fix $I\in\W_N\setminus \W^{\Sigma}_N$ and  $X\in I^{**}$, and set $\W_X:=\{J\in \W: \phi_J(X)\neq 0\}$ so that $I\in \W_X$. We first note that  $\W_X\subset \W_N$. Indeed, if $\phi_J(X)\neq 0$ then $X\in \widetilde{J^{*}}$.
	Hence $X\in I^{**}\cap J^{**}$ and our choice of $\lambda$ gives that $\partial I$ meets $\partial J$, this in turn implies that $J\in \W_N$ since $I\in\W_N\setminus \W^{\Sigma}_N$. All these yield
	$$
	\Psi_N(X)
	=
	\frac{\sum\limits_{J\in \W_N} \phi_J(X)}{\sum\limits_{J\in \W} \phi_J(X)}
	=
	\frac{\sum\limits_{J\in \W_N\cap \W_X} \phi_J(X)}{\sum\limits_{J\in \W_X} \phi_J(X)}
	=
	\frac{\sum\limits_{J\in \W_N\cap \W_X} \phi_J(X)}{\sum\limits_{J\in \W_N\cap \W_X} \phi_J(X)}
	=
	1.
	$$
	Hence $\Psi_N\big|_{I^{**}}\equiv 1$  for every $I\in\W_N\setminus \W^{\Sigma}_N$. This and the fact that $\Psi_N\in C_c^\infty(\ree)$ immediately give that $\nabla \Psi_N\equiv 0$ in $\bigcup_{I\in \W_N \setminus \W_N^\Sigma }I^{**}$.

	We are left with showing the last part of $(iv)$ and for that we borrow some ideas from \cite[Appendix A.2]{HMM-UR}. Fix $I\in  \W_N^\Sigma$ and let $J$ be so that $J\in \W\setminus \W_N$ with $\partial I\cap\partial J\neq\emptyset$, in particular $\ell(I)\approx\ell(J)$. Since $I\in  \W_N^\Sigma$ there exists $Q_I\in\dd_{\F_N,Q_0}$ (that is, $Q_I\subset Q_0$ with $2^{-N}\,\ell(Q_0)<\ell(Q_I)\le\ell(Q_0)$ so that $I\in \W_{Q_I}^*$). Pick $Q_J\in\dd$ so that $\ell(Q_J)=\ell(J)$ and it contains any fixed $\widehat{y}\in\partial\Omega$ such that $\dist(J,\partial\Omega)=\dist(J,\widehat{y})$. Then, as observed in Section \ref{subsection:sawtooth}, one has $J\in \W_{Q_J}^*$.
	But, since $J\in \W\setminus \W_N$, we necessarily have $Q_J\notin \dd_{\F_N,Q_0}=\dd_{\F_N}\cap \dd_{Q_0}$. Hence, $\W_N^\Sigma=\W_N^{\Sigma,1}\cup \W_N^{\Sigma,2}\cup \W_N^{\Sigma,3}$ where
	\begin{align*}
		\W_N^{\Sigma,1}:&=\{I\in \W_N^{\Sigma}: Q_0\subset Q_J\},
		\\
		\W_N^{\Sigma,2}:&=\{I\in \W_N^{\Sigma}: Q_J\subset Q_0,\ \ell(Q_J)\le 2^{-N}\,\ell(Q_0)\},
		\\
		\W_N^{\Sigma,3}:&=\{I\in \W_N^{\Sigma}: Q_J\cap Q_0=\emptyset\}.
	\end{align*}
	For later use it is convenient to observe that 
	\begin{equation}\label{dist:QJ-I}
		\dist(Q_J,I)\le \dist(Q_J,J) +\diam(J)+\diam(I)\approx\ell(J)+\ell(I)\approx\ell(I).
	\end{equation}

	Let us first consider $\W_N^{\Sigma,1}$. If $I\in \W_N^{\Sigma,1}$ we clearly have
	\[
	\ell(Q_0)\le \ell(Q_J)=\ell(J)\approx\ell(I)\approx\ell(Q_I)\le \ell(Q_0)
	\]
	and since $Q_I\in\dd_{Q_0}$
	\[
	\dist(I,x_{Q_0})
	\le
	\dist(I,Q_I)+ \diam(Q_I)
	\approx \ell(I).
	\]
	In particular, $\# \W_N^{\Sigma,1}\lesssim 1$. Thus if we set $\widehat{Q}_I:=Q_J$ it follows from \eqref{dist:QJ-I} that the two first conditions in \eqref{new-QI} hold and also $\sum_{I\in 	\W_N^{\Sigma,1}} \mathbf{1}_{\widehat{Q}_I} \le
	\# \W_N^{\Sigma,1}\lesssim 1$.

	Consider next $\W_N^{\Sigma,2}$. For any $I\in \W_N^{\Sigma,2}$ we also set $\widehat{Q}_I:=Q_J$ so that from \eqref{dist:QJ-I} we clearly see that the two first conditions in \eqref{new-QI} hold. It then remains to estimate the overlap. With this goal in mind we first note that if $I\in \W_N^{\Sigma,2}$, the fact that  $Q_I\in\dd_{\F_N,Q_0}$ yields
	\[
	2^{-N}\,\ell(Q_0)
	<
	\ell(Q_I)
	\approx
	\ell(I)
	\approx
	\ell(J)
	\approx
	\ell(Q_J)
	\le
	2^{-N}\,\ell(Q_0),
	\]
	hence $\ell(I)\approx 2^{-N}\,\ell(Q_0)$. Suppose next that $Q_J\cap Q_J'=\widehat{Q}_I\cap \widehat{Q}_I\neq\emptyset$ for $I, I'\in \W_N^{\Sigma,2}$. Then since $I$ touches $J$ and $I'$ touches $J'$
	\begin{multline*}
		\dist(I,I')
		\le
		\diam(J)+\dist(J, Q_J)+\diam(Q_J)+\diam(Q_J')+\diam(J')
		\\
		\approx
		\ell(J)+\ell(J')
		\approx
		2^{-N}\,\ell(Q_0).
	\end{multline*}
	Hence fixed  $I\in \W_N^{\Sigma,2}$ there is a uniformly bounded number of $I'\in \W_N^{\Sigma,2}$ with $\widehat{Q}_I\cap \widehat{Q}_{I'}\neq\emptyset$, and, in particular,
	$\sum_{I\in 	\W_N^{\Sigma,2}} \mathbf{1}_{\widehat{Q}_I} \lesssim 1$.

	We finally take into consideration the most delicate collection $\W_N^{\Sigma,3}$. In this case for every $I\in \W_N^{\Sigma,3}$ we pick $\widehat{Q}_I\in\dd$ so that $ \widehat{Q}_I\ni x_{Q_J}$ and $\ell(\widehat{Q}_I)=2^{-M'}\,\ell(Q_J)$ with $M'\ge 3$ large enough so that $2^{M'}\ge 2 \Xi^2$ (cf. \eqref{deltaQ}). Note that since $M'\ge 3$  we have that $\widehat{Q}_I\subset Q_J$ which, together with \eqref{dist:QJ-I}, implies
	\[
	\dist(I,\widehat{Q}_I)
	\le
	\dist(I,Q_J)
	+\diam(Q_J)
	\lesssim 
	\ell(I).
	\]
	Hence the first two conditions in \eqref{new-QI} hold in the current situation. 
	
	On the other hand, the choice of $M'$ and \eqref{deltaQ} guarantee that
	\begin{equation}\label{4q43tfg3f}
		\diam(\widehat{Q}_I)
		\le
		2\,\Xi\,r_{\widehat{Q}_I}
		\le
		2\,\Xi\,\ell(\widehat{Q}_I)
		=
		2^{-M'+1}\,\Xi\,\ell(Q_J)
		\le
		\Xi^{-1}\,\ell(Q_J).
	\end{equation}
	Also, since $2\Delta_{Q_J}\subset Q_J$, it follows that $Q_0\cap2\Delta_{Q_J}=\emptyset$ and therefore $2\Xi^{-1}\,\ell(Q_J)\le \dist(x_{Q_J}, Q_0)$. Besides, since $Q_I\subset Q_0$
	\begin{multline*}
		\dist(x_{Q_J}, Q_0)
		\le
		\diam(Q_J)+\dist(Q_J, J)+\diam(J)
		\\
		+\diam(I)+\dist(I,Q_I)+\diam(Q_I)
		\approx
		\ell(J)\approx\ell(I).
	\end{multline*}
	Thus, $2\,\Xi^{-1}\,\ell(Q_J)\le \dist(x_{Q_J}, Q_0) \le C\,\ell(J)$. Suppose next that $I, I'\in \W_N^{\Sigma,3}$ are so that $\widehat{Q}_I\cap\widehat{Q}_{I'}\neq\emptyset$ and assume without loss of generality that $\widehat{Q}_{I'}\subset\widehat{Q}_{I}$, hence $\ell(J')\le \ell(J)$. Then, since 
	$x_{Q_J}\in \widehat{Q}_I$ and $x_{Q_{J'}}\in \widehat{Q}_{I'} \subset \widehat{Q}_I$ we get from \eqref{4q43tfg3f}
	\begin{multline*}
		2\,\Xi^{-1}\,\ell(Q_{J})
		\le 
		\dist(x_{Q_{J}}, Q_0)
		\le
		|x_{Q_{J}}- x_{Q_{J'}}|+
		\dist(x_{Q_{J'}}, Q_0)
		\\
		\le
		\diam(\widehat{Q}_{I})+C\ell(J')
		\le
		\Xi^{-1}\,\ell(Q_J)+C\ell(J')
	\end{multline*}
	and therefore $\Xi^{-1}\,\ell(Q_{J})\le C\,\ell(J)$ which in turn gives $\ell(I)\approx \ell(J)\approx \ell(J')\approx\ell(I')$. Note also that since $I$ touches $J$, $I'$ touches $J'$, and $\widehat{Q}_I\cap\widehat{Q}_{I'}\neq\emptyset$ we obtain
	\begin{multline*}
		\dist(I,I')
		\le
		\diam(J)+\dist(J, Q_J)+\diam(Q_J)+\diam(Q_{J'})
		\\+\dist(Q_{J'}, J')+\diam(J')
		\approx
		\ell(J)+\ell(J')\approx\ell(I).
	\end{multline*}
	Consequently, fixed  $I\in \W_N^{\Sigma,3}$ there is a uniformly bounded number of $I'\in \W_N^{\Sigma,3}$ with $\widehat{Q}_I\cap \widehat{Q}_{I'}\neq\emptyset$. As a result, $\sum_{I\in 	\W_N^{\Sigma,3}} \mathbf{1}_{\widehat{Q}_I} \lesssim 1$.  This clearly completes the proof of $(iii)$ and hence that of Lemma \ref{lemma:approx-saw}.
\end{proof}

\subsection[Elliptic operators, elliptic measure and the Green function]{Uniformly elliptic operators, elliptic measure and the Green function}
Next, we recall several facts concerning elliptic measure and the Green functions. To set the stage let $\Omega\subset\re^{n+1}$ be an open set. Throughout we consider  
elliptic operators $L$ of the form $Lu=-\div(A\nabla u)$ with $A(X)=(a_{i,j}(X))_{i,j=1}^{n+1}$ being a real (non-necessarily symmetric) matrix such that $a_{i,j}\in L^{\infty}(\Omega)$ and there exists $\Lambda\geq 1$ such that the following uniform ellipticity condition holds 
\begin{align}
\label{e:elliptic}
\Lambda^{-1} |\xi|^{2} \leq A(X) \xi \cdot \xi,
\qquad\qquad
|A(X) \xi \cdot\eta|\leq \Lambda |\xi|\,|\eta| 
\end{align}
for all $\xi,\eta \in\mathbb{R}^{n+1}$ and for almost every $X\in\Omega$. We write $L^\top$ to denote the transpose of $L$, or, in other words, $L^\top u = -\div(A^\top
\nabla u)$ with $A^\top$ being the transpose matrix of $A$.

We say that $u$ is a weak solution to $Lu=0$ in $\Omega$ provided that $u\in W_{\rm loc}^{1,2}(\Omega)$ satisfies
\[
\iint A(X)\nabla u(X)\cdot \nabla\phi(X) dX=0  \quad\mbox{whenever}\,\, \phi\in C^{\infty}_{0}(\Omega).
\]
Associated with $L$ one can construct an elliptic measure $\{\omega_L^X\}_{X\in\Omega}$ and a Green function $G_L$ (see \cite{HMT1} for full details). Sometimes, in order to emphasize the dependence on $\Omega$, we will write $\omega_{L,\Omega}$ and $G_{L,\Omega}$. If $\Omega$ satisfies the CDC then it follows that all boundary points are Wiener regular and hence for a given $f\in C_c(\partial\Omega)$ we can define
\[
u(X)=\int_{\partial\Omega} f(z)d\omega^{X}_{L}(z), \quad \mbox{whenever}\, \, X\in\Omega,
\]
so that  $u\in W^{1,2}_{\rm loc}(\Omega)\cap C(\overline{\Omega})$ satisfies $u=f$ on $\partial\Omega$ and $Lu=0$ in the weak sense in $\Omega$. Moreover, if $\Omega$ is bounded and $f\in \Lip(\Omega)$ then $u\in W^{1,2}(\Omega)$. In the same context the Green function satisfies the following properties which will be used along the paper: 
\begin{gather}\label{sizestimate}
	0\le G_L(X,Y)\leq C|X-Y|^{1-n},\quad\forall X,Y\in\Omega,\quad X\neq Y;
	\\[0.15cm] 
	G_L(\cdot,Y)\in  W_{\rm loc}^{1,2}(\Omega\setminus\{Y\})\cap C\big(\overline{\Omega}\setminus\{Y\}\big)\quad\text{and}\quad G_L(\cdot,Y)|_{\partial\Omega}\equiv 0\quad\forall Y\in\Omega;
	\\[0.15cm]\label{G-G-top}
	G_L(X,Y)=G_{L^\top}(Y,X),\quad\forall X,Y\in\Omega,\quad X\neq Y;
	\\[0.15cm]
	\label{eq:G-delta}
	\iint_{\Omega}A(X)\nabla_X G_L(X,Y)\cdot\nabla\varphi(X)\,dX=\varphi(Y),\qquad\forall\, \varphi\in  C_c^{\infty}(\Omega).
\end{gather}

We first define the reverse Hölder class and the $A_\infty$ classes with respect to fixed elliptic measure in $\Omega$.  One reason we take this approach is that we do not know whether $\mathcal{H}^{n}|_{\partial\Omega}$ is well-defined since we do not assume any Ahlfors regularity. Hence we have to develop these notions in terms of elliptic measures. To this end, let $\Omega$ satisfy the CDC and let $L_0$ and $L$ be two real (non-necessarily symmetric) elliptic operators associated with $L_0u=-\div(A_0\nabla u)$ and $L u=-\div(A\nabla u)$ where $A$ and $A_0$ satisfy \eqref{e:elliptic}. Let $\omega^{X}_{L_0}$ and $\omega_{L}^{X}$ be the elliptic measures of $\Omega$ associated with the operators $L_0$ and $L$ respectively with pole at $X\in\Omega$. Note that if we further assume that $\Omega$ is connected then $\omega_{L}^{X}\ll\omega_L^{Y}$ on $\pom$ for every $X,Y\in\Omega$. Hence if $\omega_L^{X_0}\ll\omega_{L_0}^{Y_0}$ on $\pom$  for some $X_0,Y_0\in\Omega$ then $\omega_L^{X}\ll\omega_{L_0}^{Y}$ on $\pom$ for every $X,Y\in\Omega$ and thus  we can simply write $\omega_{L}\ll \omega_{L_0}$ on $\pom$. In the latter case we will use the notation
\begin{equation}\label{def-RN}
h(\cdot\,;L, L_0, X)=\frac{d\omega_L^{X}}{d\omega_{L_0}^{X}}
\end{equation}
to denote the Radon-Nikodym derivative of $\omega_{L}^{X}$ with respect to $\omega_{L_0}^{X}$,
which is a well-defined function $\omega_{L_0}^{X}$-almost everywhere on $\pom$.

\begin{definition}[Reverse Hölder and $A_\infty$ classes]\label{d:RHp}
	Fix $\Delta_0=B_0\cap \pom$ where $B_0=B(x_0,r_0)$ with $x_0\in\pom$ and $0<r_0<\diam(\pom)$. Given $p$, $1<p<\infty$, we say that $\omega_L\in RH_p(\Delta_0,\omega_{L_0})$, provided that $\omega_L\ll \omega_{L_0}$ on $\Delta_0$, and there exists $C\geq 1$ such that 
\begin{equation*}
	\left(\aver{\Delta}h(y;L,L_0,X_{\Delta_0} )^p d \omega_{L_0}^{X_{\Delta_0}}(y)\right)^{\frac1p} 
	\leq 
	C 
	\aver{\Delta} h(y;L,L_0,X_{\Delta_0} ) d \omega_{L_0}^{X_{\Delta_0}}(y)
	=
	C\frac{\omega_L^{X_{\Delta_0}}(\Delta)}{\omega_{L_0}^{X_{\Delta_0}}(\Delta)},
\end{equation*}
	for every $\Delta=B\cap \partial\Omega$ where $B\subset B(x_0,r_0)$, $B=B(x,r)$ with  $x\in \partial\Omega$, $0<r<\diam(\partial\Omega)$. The infimum of the constants $C$ as above is denoted by $[\omega_{L}]_{RH_p(\Delta_0,\omega_{L_0})}$. 
	
	Similarly, we say that $\omega_L\in RH_p(\pom,\omega_{L_0})$ provided that for every $\Delta_0=\Delta(x_0,r_0)$ with $x_0\in\pom$ and $0<r_0<\diam(\pom)$ one has $\omega_L\in RH_p(\Delta_0,\omega_{L_0})$ uniformly on $\Delta_0$, that is, 
	\[
	[\omega_{L}]_{RH_p(\pom,\omega_{L_0})}
	:=\sup_{\Delta_0} [\omega_{L}]_{RH_p(\Delta_0,\omega_{L_0})}<\infty.
	\]

	Finally,
	\[
	A_\infty(\Delta_0,\omega_{L_0})=\bigcup_{p>1} RH_p(\Delta_0,\omega_{L_0})
	\quad\mbox{and}\quad
	A_\infty(\partial\Omega,\omega_{L_0})=\bigcup_{p>1} RH_p(\partial\Omega,\omega_{L_0})
	.\]
\end{definition}

The following result lists a number of properties which will be used throughout the paper, proofs may be found in \cite{HMT1}: 

\begin{lemma}\label{lemma:proppde}
	Suppose that $\Omega\subset\re^{n+1}$, $n\ge 2$, is a 1-sided NTA domain satisfying the CDC. Let $L_0=-\div(A_0\nabla)$ and $L=-\div(A\nabla)$ be two real (non-necessarily symmetric) elliptic operators, there exist $C_1\ge 1$, $\rho\in (0,1)$ (depending only on dimension, the 1-sided NTA constants, the CDC constant, and the ellipticity of $L$) and $C_2\ge 1$ (depending on the same parameters and on the ellipticity of $L_0$), such that for every $B_0=B(x_0,r_0)$ with $x_0\in\partial\Omega$, $0<r_0<\diam(\partial\Omega)$, and $\Delta_0=B_0\cap\partial\Omega$ we have the following properties:
	\begin{list}{$(\theenumi)$}{\usecounter{enumi}\leftmargin=1cm \labelwidth=1cm \itemsep=0.1cm \topsep=.2cm \renewcommand{\theenumi}{\alph{enumi}}}
		
		\item $\omega_L^Y(\Delta_0)\geq C_1^{-1}$ for every $Y\in C_1^{-1}B_0\cap\Omega$ and $\omega_L^{X_{\Delta_0}}(\Delta_0)\ge C_1^{-1}$.

		\item If $B=B(x,r)$ with $x\in\partial\Omega$ and $\Delta=B\cap\partial\Omega$ is such that $2B\subset B_0$, then for all $X\in\Omega\setminus B_0$ we have that $
	{C_1^{-1}}\omega_L^X(\Delta)\leq r^{n-1} G_L(X,X_\Delta)\leq C_1\omega_L^X(\Delta)$.

		\item  If $X\in\Omega\setminus 4B_0$, then $ \omega_{L}^X(2\Delta_0)\leq C_1\omega_{L}^X(\Delta_0)$.

		\item  For every $X\in\Omega\setminus 2\kappa_0B_0$ with $\kappa_0$ as in \eqref{definicionkappa0}, we have that
		\[
		\frac1C_1 \frac1{\omega_L^X(\Delta_0)}\le  \frac{d\omega_L^{{X_{\Delta_0}}}}{d\omega_L^X}(y)\le C_1 \frac1{\omega_L^X(\Delta_0)},
		\qquad\mbox{for $\omega_L^X$-a.e. $y\in\Delta_0$}.
		\]
		
%
%
%
%
%
		\item For every $X\in B_0\cap\Omega$ and for any $j\ge 1$
		\[
		\frac{d\omega_L^X}{d\omega_L^{X_{2^j\Delta_0}}}(y)\le C_1\,\bigg(\frac{\delta(X)}{2^j\,r_0}\bigg)^{\rho},
		\qquad\mbox{for $\omega_L^X$-a.e. $y\in\pom\setminus 2^j\,\Delta_0$}.
		\]
	\end{list}
\end{lemma}

\medskip

\begin{remark}\label{remark:chop-dyadic}
We note that from $(d)$ in the previous result, Harnack's inequality, and \eqref{deltaQ} one can easily see that 
\begin{equation}\label{chop-dyadic:densities}
\frac{d\omega_L^{X_{Q'}}}{d\omega_L^{X_{Q''}}}(y)
\approx 
\frac1{\omega_L^{X_{Q''}}(Q')},
\qquad
\mbox{ for $\omega_L^{X_{Q''}}$-a.e. }y\in Q',
\mbox{whenever }Q'\subset Q''\in\dd.
\end{equation}		
Observe that since $\omega_L^{X_{Q''}}\ll \omega_L^{X_{Q'}}$  an analogous inequality for the reciprocal of the Radon-Nikodym derivative follows immediately. 
\end{remark}

\begin{remark}\label{remark:diam-radius}
Given $\Omega$, a 1-sided NTA domain satisfying the CDC, we claim that if $\Delta=\Delta(x,r)$ with $x\in\pom$ and $0<r<\diam(\pom)$ then $\diam(\Delta)\approx r$. To see this we first observe that $\diam(\Delta)\le 2r$. If $\diam(\Delta)\ge c_0 r/4$ ---$c_0$ is the Corkscrew constant--- then clearly $\diam(\Delta)\approx r$. Hence, we may assume that $\diam(\Delta)<c_0 r/4$. Let $s=2\diam(\Delta)$ so that $\diam(\Delta)<s<r$ and note that one can easily see that $\Delta=\Delta':=\Delta(x,s)$. Associated with $\Delta$ and $\Delta'$ we can consider $X_{\Delta}$ and $X_{\Delta'}$ the corresponding Corkscrew points. These are different, despite the fact that 
$\Delta=\Delta(x,r)$. Indeed,
\[
c_0r
\le 
\delta(X_{\Delta})\le |X_\Delta-X_{\Delta'}|+|X_{\Delta'}-x|
\le
|X_\Delta-X_{\Delta'}|+s
<
|X_\Delta-X_{\Delta'}|+
\frac{c_0}2 r
\]
which yields that $ |X_\Delta-X_{\Delta'}|\ge \frac{c_0}2 r$. Note that $X_\Delta\notin 2B':=B(x,2s)$ since otherwise we would get a contradiction:
$c_0r\le \delta(X_\Delta)\le |X_\Delta-x|<2s<c_0 r$. Hence we can invoke Lemma \ref{lemma:proppde} parts $(a)$ and $(b)$ and \eqref{sizestimate} to see that
\[
1\approx \omega_L^{X_\Delta}(\Delta)
=
\omega_L^{X_\Delta}(\Delta')
\approx
s^{n-1} G_L(X_{\Delta},X_{\Delta'})
\lesssim
s^{n-1} |X_{\Delta}-X_{\Delta'}|^{1-n}
\lesssim
(s/r)^{n-1}.
\]
This and the fact that $n\ge 2$ easily yields that $r\lesssim s$ as desired. 
\end{remark}	

We close this section by establishing an estimate for the non-tangential maximal function for elliptic-measure solutions.

\begin{proposition}\label{prop:NT-S-max} 	Let $\Omega\subset\mathbb{R}^{n+1}$ be a 1-sided NTA domain satisfying the CDC.
	Given $Q_0\in\dd$ and $f\in C(\pom)$ with $\supp f\subset 2\widetilde{\Delta}_{Q_0}$
	let
	\[
	u(X)=\int_{\pom} f(y)\,d\omega_{L}^X(y),\qquad X\in\pom.
	\]
	Then for every $x\in Q_0$,	
	\begin{equation}\label{eqn:NT-M}
		\mathcal{N}_{Q_0}u(x) 
		\lesssim
		\sup_{\substack{\Delta\ni x \\ 0<r_\Delta< 4\Xi  r_{Q_0}}} \aver{\Delta} |f(y)|\,d\omega_{L}^{X_{Q_0}}(y),
	\end{equation}
	and, as a consequence, for every $1<q\le\infty$
	\begin{equation}\label{NT-data}
		\|\mathcal{N}_{Q_0}u\|_{L^q(Q_0,\omega_L^{X_{Q_0}})}
		\lesssim
		\|f\|_{L^q(2\widetilde{\Delta}_{Q_0},\omega_L^{X_{Q_0}})}.
	\end{equation}
	Moreover, the implicit constants depend just on dimension $n$, the 1-sided NTA constants, the CDC constant, and the ellipticity constant of $L$ and on $q$ in  \eqref{NT-data}.
\end{proposition}

\begin{proof}
By decomposing $f$ into its positive and negative parts we may assume that $f$ is non-negative with $\supp f\subset 2\widetilde{\Delta}_{Q_0}$ and construct the associated $u$ as in the statement which is non-negative. Fix $x\in Q_0$ and let $X\in\Gamma_{Q_0}^*(x)$. Then, by definition there are $Q\in\dd_{Q_0}$ and $I\in \W_Q^*$ such that $x\in Q$  and $X\in I^{**}$. Hence using Harnack's inequality and the notation introduced in \eqref{deltaQ}--\eqref{deltaQ3}
\begin{multline*}
	u(X)
	=
	\int_{\pom} f(y)\,d\omega_L^X(y)
	\approx
	\int_{\pom} f(y)\,d\omega_L^{X_Q}(y)
	\\
	\le
	\int_{4\,\widetilde{\Delta}_Q} f(y)\,d\omega_L^{X_Q}(y)
	+
	\sum_{j=3}^\infty \int_{2^j\,\widetilde{\Delta}_Q\setminus 2^{j-1}\,\widetilde{\Delta}_Q} f(y)\,d\omega_L^{X_Q}(y)
	=:\sum_{j=2}^\infty \mathcal{I}_j.
\end{multline*}
Let $k_0\ge 0$ be such that $\ell(Q)=2^{-k_0}\ell (Q_0)$. Observe that for every $j\ge k_0+3$ one has that $2\widetilde{\Delta}_{Q_0}\setminus 2^{j-1}\widetilde{\Delta}_{Q}=\emptyset$. Otherwise there is $z\in 2\widetilde{\Delta}_{Q_0}\setminus 2^{j-1}\widetilde{\Delta}_{Q}$ and hence we get a contradiction:
\[
4\,\Xi\,r_{Q_0} 
\le 
2^{j-1-k_0} \,\Xi\,r_{Q_0} 
=
2^{j-1} \,\Xi\,r_{Q} 
\le
|z-x_Q|
\le
|z-x_{Q_0}|+ |x_Q-x_{Q_0}|
\le
3\,\Xi\,r_{Q_0}. 
\]
With this in hand, and since $\supp f\subset 2\widetilde{\Delta}_{Q_0}$, we clearly see that $\mathcal{I}_j=0$ for $j\ge k_0+3$. 

In order to estimate the $\mathcal{I}_j$'s we need some preparatives. 
Note that for every $2\le j\le k_0+2$ one has $2^{j}\widetilde{B}_Q\subset 5\widetilde{B}_{Q_0}$.  We claim that  
\begin{equation}\label{aux-chop}
	\frac{d\omega_L^{X_{2^{j}\widetilde{\Delta}_Q}}}{d\omega_L^{X_{Q_0}}}(y)\lesssim  \frac1{\omega_{L}^{X_{Q_0}}(2^{j}\widetilde{\Delta}_Q)},
	\qquad\mbox{for $\omega_L^{X_{Q_0}}$-a.e. $y\in2^j\,\widetilde{\Delta}_Q$, $2\le j\le k_0+2$.}
\end{equation}
Indeed, this estimate follows from Harnack's inequality and Lemma \ref{lemma:proppde} part $(a)$ when $j\approx k_0$ since $2^{j}\,\ell(Q)\approx \ell(Q_0)$,  and from Lemma \ref{lemma:proppde} part $(d)$ whenever $j\ll k_0$. We also observe that Lemma \ref{lemma:proppde} part $(a)$ and Harnack's inequality readily give that 
\begin{equation}\label{aux-Bourgain}
	\omega_{L}^{X_{2^{j}\widetilde{\Delta}_Q}}(2^{j}\widetilde{\Delta}_Q)\approx 1, \qquad \text{for every $2\le j\le k_0+2$.}
\end{equation}
Finally, by Lemma \ref{lemma:proppde} part $(e)$ and Harnack's inequality it follows that 
\begin{equation}\label{aux-decay-kernelfn}
	\frac{d\omega_L^{X_Q}}{d\omega_L^{X_{2^{j-1}\widetilde{\Delta}_Q}}}(y)\lesssim  2^{-j\,\rho},
	\qquad\mbox{for $\omega_L^{X_Q}$-a.e. $y\in\pom\setminus 2^{j-1}\,\widetilde{\Delta}_Q$},\ j\ge 3.
\end{equation}

Let us start estimating $\mathcal{I}_2$. Use Harnack's inequality and \eqref{aux-Bourgain}, \eqref{aux-chop} with $j=2$, to conclude that 
\[
\mathcal{I}_2
\approx
\aver{4\,\widetilde{\Delta}_Q} f(y)\,d\omega_L^{X_{\widetilde{\Delta}_Q}}(y)
\approx
\aver{4\,\widetilde{\Delta}_Q} f(y)\,d\omega_L^{X_{Q_0}}(y).
\]
On the other hand, for $3\le j\le k_0+2$, we employ \eqref{aux-decay-kernelfn}, Harnack's inequality, \eqref{aux-Bourgain}, and \eqref{aux-chop}
\begin{multline*}
	\mathcal{I}_j
	\lesssim
	2^{-j\,\rho} \int_{2^j\,\widetilde{\Delta}_Q\setminus 2^{j-1}\,\widetilde{\Delta}_Q} f(y)\,d\omega_L^{X_{2^{j-1}\,\widetilde{\Delta}_Q}}(y)
	\lesssim
	2^{-j\,\rho} \aver{2^j\,\widetilde{\Delta}_Q} f(y)\,d\omega_L^{X_{2^{j}\,\widetilde{\Delta}_Q}}(y)
	\\
	\approx
	2^{-j\,\rho} \aver{2^j\,\widetilde{\Delta}_Q} f(y)\,d\omega_L^{X_{Q_0}}(y).
\end{multline*}
If we now collect all the obtained estimates we conclude as desired \eqref{eqn:NT-M}:
\begin{multline*}
	u(X)
	\lesssim
	\sum_{j=2}^{k_0+2} \mathcal{I}_j
	\lesssim
	\sum_{j=2}^{k_0+2} 2^{-j\,\rho} \aver{2^j\,\widetilde{\Delta}_Q} f(y)\,d\omega_L^{X_{Q_0}}(y)
	\\
	\le
	\sup_{\substack{\Delta\ni x \\ 0<r_\Delta< 8\Xi  r_{Q_0}}} \aver{\Delta} |f(y)|\,d\omega_{L}^{X_{Q_0}}(y)
	\,
	\sum_{j=2}^{\infty} 2^{-j\,\rho}
	\lesssim
	\sup_{\substack{\Delta\ni x \\ 0<r_\Delta< 4\Xi  r_{Q_0}}} \aver{\Delta} |f(y)|\,d\omega_{L}^{X_{Q_0}}(y).
\end{multline*}

To complete the proof we just need to obtain \eqref{NT-data} but this follows at once upon using \eqref{eqn:NT-M} and observing that the local Hardy-Littlewood maximal function on its right hand side is bounded on ${L^q(20\,\widetilde{\Delta}_{Q_0},\omega_L^{X_{Q_0}})}$ since $\omega_L^{X_{Q_0}}$ is a doubling measure in $20\,\widetilde{\Delta}_{Q_0}$ by Lemma \ref{lemma:proppde} parts $(a)$ and $(c)$.
\end{proof}

\section{Dyadic sawtooth lemma for projections}\label{section:DJK-proj}

In this section, we shall prove two dyadic versions of the main lemma in \cite{DJK}. To set the stage we sate a result which is partially proved in \cite[Proposition 6.7]{HM1} under the further assumption that $\partial\Omega$ is Ahlfors regular

\begin{proposition}\label{prop:Pi-proj}
	Let $\Omega\subset\ree$, $n\ge 2$, be a 1-sided NTA domain satisfying the CDC. Fix $Q_0\in \dd $ and let $\mathcal{F}=\{Q_k\}_k \subset \mathbb{D}_{Q_0}$ be a family of pairwise disjoint dyadic cubes. There exists $Y_{Q_0}\in \Omega\cap \Omega_{\F,Q_0}\cap \Omega_{\F,Q_0}^*$ so that
	\begin{equation}\label{eq:common-cks}
		\dist(Y_{Q_0},\pom)\approx \dist(Y_{Q_0},\pom_{\F,Q_0})\approx \dist(Y_{Q_0},\pom_{\F,Q_0}^*)\approx \ell(Q_0),
	\end{equation}
	where the implicit constants depend only on dimension, the 1-sided NTA constants, the CDC constant, and is  independent of $Q_0$ and $\F$. 
	Additionally, for each $Q_j\in\mathcal{F}$, there is an $n$-dimensional cube $P_j\subset\partial\Omega_{\mathcal{F},Q_0}$, which is contained in a face of $I^*$ for some $I\in\mathcal{W}$, and  which satisfies
	\begin{equation}\label{props:Pj}
		\ell(P_j)\approx \dist(P_j,Q_j)\approx \dist(P_j,\partial\Omega)\approx \ell(I)\approx \ell(Q_j),
	\end{equation}
	and
	\begin{equation}\label{eqn:overlap-Pj}
		\sum_{j} 1_{P_j} \lesssim 1,
	\end{equation}
	where the implicit constants depend on allowable parameters.  
\end{proposition}

\begin{proof}
Note first that $\Omega_{\F, Q_0}$ is a 1-sided NTA domain satisfying the CDC (see Proposition \ref{prop:CDC-inherit}). Pick an arbitrary $x_0\in \pom_{\F, Q_0}$ and let $Y_0$ be a Corkscrew point relative to the surface ball  $B(x_0,\diam(\pom_{\F, Q_0})/2)\cap\pom_{\F, Q_0}$ for the bounded domain $\Omega_{\F, Q_0}$ (recall that one has $\diam(\pom_{\F, Q_0})\approx \ell(Q_0)<\infty$ by \eqref{definicionkappa12}). Note that $Y_0\in\Omega_{\F, Q_0}\subset\Omega$, which is comprised of fattened Whitney boxes,  then $Y_0\in I^{**}$ for some $I\in\W$, with $\interior(I^{**})\subset \Omega_{\F, Q_0}$. Let $Y_{Q_0}=X(I)$ be the center of $I$ so that $\delta(Y_0)\approx \ell(I)\approx \delta(Y_{Q_0})$. Then,  
\begin{multline*}
	\ell(Q_0)\approx \diam(\pom_{\F, Q_0}) \approx \dist(Y_0,\pom_{\F, Q_0})\le \dist(Y_0,\pom_{\F, Q_0^*})\le\delta(Y_0)
	\\
	\approx \delta(Y_{Q_0})\approx \ell(I)\le \diam(\Omega_{\F, Q_0})=\diam(\pom_{\F, Q_0})\approx \ell(Q_0).
\end{multline*}

To continue we note that the existence of the family $\{P_j\}_j$ so that \eqref{props:Pj} holds has been proved in \cite[Proposition 6.7]{HM1} under the further assumption that $\partial\Omega$ is Ahlfors regular. However, a careful examination of the proof shows that the same argument applies in our scenario. We are left with showing \eqref{eqn:overlap-Pj}. To see this, observe that as in \cite[Remark 6.9]{HM1} if $P_j\cap P_k\neq\emptyset$ then $\ell(Q_j)\approx \ell(Q_k)$. Indeed from the previous result $P_j\subset I_j^*$ and $P_k\subset I_k^*$ for some $I_j,I_k\in\W$. Thus $I_j^*$ meets $I_k^*$ and by construction $I_j$ and $I_k$ meet.
Using \eqref{props:Pj} and the nature of the Whitney cubes we see that $\ell(Q_j)\approx\ell(I_j)\approx \ell(I_k)\approx \ell(Q_k)$. Using this and \eqref{props:Pj} one can also see that $\dist(Q_j,Q_k)\lesssim \ell(Q_j)\approx \ell(Q_k)$. Hence, fixing $P_{j_0}$ and $x\in P_{j_0}$ we have some constant $k_0\ge 1$ (depending on the allowable parameters) such that 
\begin{multline*}
\sum_{j} 1_{P_j}(x)
\le
\#\{P_k:\ P_k\cap P_{j_0}\neq \emptyset\}
\\
\leq 
\#\big\{Q_k:\ 2^{-k_0}\le \tfrac{\ell(Q_k)}{\ell(Q_{j_0})}\le 2^{k_0},\   \dist(Q_k,Q_{j_0})\le 2^{k_0}\ell(Q_{j_0})\big\}
\\
=
\sum_{k=-k_0}^{k_0} \#\big\{Q_k:\ \ell(Q_k)=2^{k}\ell(Q_{j_0}),\   \dist(Q_k,Q_{j_0})\le 2^{k_0}\ell(Q_{j_0})\big\}
=:\sum_{k=-k_0}^{k_0} N_k.
\end{multline*}
To estimate each of the terms in the last sum fix $k$ and note that since the cubes belong to the same generation then $Q_k$'s involved are disjoint and hence so they are the corresponding $\Delta_{Q_k}$'s which all have radius $(2C)^{-1}2^{k}\ell(Q_{j_0})$. In particular, $|x_{Q_k}-x_{Q_k'}|\gtrsim 2^{k}\ell(Q_{j_0})\ge 2^{-k_0}\ell(Q_{j_0})$ for any such cubes $Q_k$ and $Q_{k'}$. Moreover, 
\[
|x_{Q_k}-x_{Q_{j_0}}|\le
\diam(Q_k)+\dist(Q_k,Q_{j_0})+\diam(Q_{j_0})
\lesssim
2^{k_0}\ell(Q_{j_0}).
\]
Thus it is easy to see (since $\ree$ is geometric doubling) that $N_k\lesssim 2^{2k_0(n+1)}$. All these together gives us desired \eqref{eqn:overlap-Pj} ---we note in passing that the argument in \cite[Remark 6.9]{HM1} used the fact there $\pom$ is AR to estimate each $N_k$, while here we are invoking the geometric doubling property of the ambient space $\ree$. 
\end{proof}

We are now ready to state the first main result of this chapter which is a version of \cite[Lemma 6.15]{HM1} (see also \cite{DJK}) valid in our setting:


\begin{lemma}[Discrete sawtooth lemma for projections]\label{lemma:DJK-sawtooth}
	Suppose that $\Omega\subset\re^{n+1}$, $n\ge 2$, is a \textbf{bounded} 1-sided NTA domain satisfying the CDC. Let $Q_0\in\mathbb{D}$, let $\mathcal{F}=\{Q_i\}\subset\mathbb{D}_{Q_0}$ be a family of pairwise disjoint dyadic cubes, and let $\mu$ be a dyadically doubling measure in $Q_0$. Given two real (non-necessarily symmetric) elliptic $L_0$, $L$, we write  $\omega_0^{Y_{Q_0}}=\omega_{L_0,\Omega}^{Y_{Q_0}}$,  $\omega_L^{Y_{Q_0}}=\omega_{L,\Omega}^{Y_{Q_0}}$ for the elliptic measures associated with $L_0$ and $L$ for the domain $\Omega$  with fixed pole at $Y_{Q_0}\in\Omega_{\mathcal{F},Q_0}\cap\Omega$ (cf.~Lemma \ref{prop:Pi-proj}). Let $\omega_{L,*}^{Y_{Q_0}}=\omega_{L,\Omega_{\mathcal{F},Q_0}}^{Y_{Q_0}}$ be the elliptic measure associated with $L$ for the domain $\Omega_{\mathcal{F},Q_0}$ with fixed pole at  $Y_{Q_0}\in\Omega_{\mathcal{F},Q_0}\cap\Omega$. Consider $\nu_L^{Y_{Q_0}}$ the measure defined by
	\begin{equation}\label{eq:def-nu}
	\nu_L^{Y_{Q_0}}(F)=\omega_{L,*}^{Y_{Q_0}}\Big(F\setminus\bigcup_{Q_i\in\mathcal{F}}Q_i\Big)+\sum_{Q_i\in\mathcal{F}}\frac{\omega_L^{Y_{Q_0}}(F\cap Q_i)}{\omega_L^{Y_{Q_0}}(Q_i)}\omega_{L,*}^{Y_{Q_0}}(P_i),\qquad F\subset Q_0,
	\end{equation}
	where $P_i$ is the cube produced in Proposition \ref{prop:Pi-proj}. Then $\mathcal{P}_{\mathcal{F}}^{\mu}\nu_L^{Y_{Q_0}}$ (see \eqref{defprojection}) depends only on $\omega_0^{Y_{Q_0}}$ and $\omega_{L,*}^{Y_{Q_0}}$, but not on $\omega_L^{Y_{Q_0}}$. More precisely,
	\begin{equation}\label{eq:def-nu:P}
	\mathcal{P}_{\mathcal{F}}^{\mu}\nu_L^{Y_{Q_0}}(F)
	=
	\omega_{L,*}^{Y_{Q_0}}\Big(F\setminus\bigcup_{Q_i\in\mathcal{F}}Q_i\Big)+\sum_{Q_i\in\mathcal{F}}\frac{\mu(F\cap Q_i)}{\mu(Q_i)}\omega_{L,*}^{Y_{Q_0}}(P_i),\qquad F\subset Q_0.
	\end{equation}
	Moreover, there exists $\theta>0$ such that for all $Q\in\mathbb{D}_{Q_0}$ and all $F\subset Q$, we have
	\begin{equation}\label{ainfsawtooth}
	\bigg(\frac{\mathcal{P}_{\mathcal{F}}^{\mu}\omega_L^{Y_{Q_0}}(F)}{\mathcal{P}_{\mathcal{F}}^{\mu}\omega_L^{Y_{Q_0}}(Q)}\bigg)^\theta
	\lesssim
	\frac{\mathcal{P}_{\mathcal{F}}^{\mu}\nu_L^{Y_{Q_0}}(F)}{\mathcal{P}_{\mathcal{F}}^{\mu}\nu_L^{Y_{Q_0}}(Q)}
	\lesssim
	\frac{\mathcal{P}_{\mathcal{F}}^{\mu}\omega_L^{Y_{Q_0}}(F)}{\mathcal{P}_{\mathcal{F}}^{\mu}\omega_L^{Y_{Q_0}}(Q)}.
	\end{equation}
\end{lemma}

\begin{proof}
	
Our argument follows the ideas from \cite[Lemma 6.15]{HM1} and we use several auxiliary technical results from \cite[Section 6]{HM1} which were proved under the additional assumption that $\pom$ is AR. However, as we will indicate along the proof, most of them can be adapted to our setting. Those arguments that require new ideas will be explained in detail.

We first observe that \eqref{eq:def-nu:P} readily follows from the definitions of $\mathcal{P}_{\mathcal{F}}^{\mu}$ and $\nu_L^{Y_{Q_0}}$. We first establish the second estimate in \eqref{ainfsawtooth}. With this goal in mind let us fix $Q\in\dd_{Q_0}$ and $F\subset Q_0$.

\noindent\textbf{Case 1:} There exists $Q_i\in\F$ such that $Q\subset Q_i$. By \eqref{eq:def-nu:P} we have
\[
\frac{\mathcal{P}_{\mathcal{F}}^{\mu}\nu_L^{Y_{Q_0}}(F)}{\mathcal{P}_{\mathcal{F}}^{\mu}\nu_L^{Y_{Q_0}}(Q)}
=
\frac{\frac{\mu(F\cap Q_i)}{\mu(Q_i)}\omega_{L,*}^{Y_{Q_0}}(P_i)}{\frac{\mu(Q\cap Q_i)}{\mu(Q_i)}\omega_{L,*}^{Y_{Q_0}}(P_i)}
=
\frac{\mu(F)}{\mu(Q)}
=
\frac{\frac{\mu(F\cap Q_i)}{\mu(Q_i)}\omega_L^{Y_{Q_0}}(Q_i)}{\frac{\mu(Q\cap Q_i)}{\mu(Q_i)}\omega_L^{Y_{Q_0}}(Q_i)}
=
\frac{\mathcal{P}_{\mathcal{F}}^{\mu}\omega_L^{Y_{Q_0}}(F)}{\mathcal{P}_{\mathcal{F}}^{\mu}\omega_L^{Y_{Q_0}}(Q)}.
\]

\medskip

\noindent\textbf{Case 2:} $Q\not\subset Q_i$ for any $Q_i\in\F$, that is, $Q\in\dd_{\F,Q_0}$. In particular if $Q\cap Q_i\neq\emptyset$ with $Q_i\in\F$ then necessarily $Q_i\subsetneq Q$. Let $x_i^\star$ denote the center of $P_i$ and pick $r_i\approx \ell(Q_i)\approx \ell(P_i)$ so that $P_i\subset \Delta_\star(x_i^\star,r_i):=B(x_i^\star,r_i)\cap\partial\Omega_{\F,Q_0}$. Note that by Proposition \ref{prop:CDC-inherit}, Harnack's inequality and Lemma \ref{lemma:proppde} parts $(a)$ and $(c)$ we have that  $\omega_{L,*}^{Y_{Q_0}}(P_i)\approx \omega_{L,*}^{Y_{Q_0}}(\Delta_\star(x_i^\star,r_i))$. On the other hand as in \cite[Proposition 6.12]{HM1} one can see that 
\begin{equation}\label{eqn:L6-12-HM1}
\Delta_\star^Q:=B(x_Q^\star,t_Q)\cap\partial\Omega_{\F,Q_0}
\subset 
\Big(Q\setminus\bigcup_{Q_i\in\F}Q_i\Big)\bigcup
\Big(\bigcup_{Q_i\in\F: Q_i\subsetneq Q} \Delta_\star(x_i^\star,r_i) \Big)
\end{equation}
with $t_Q\approx\ell(Q)$, $x_Q^\star\in \partial\Omega_{\F,Q_0}$ and $\dist(Q, \Delta_\star^Q)\lesssim \ell(Q)$ with implicit constants depending on the allowable parameters. We note that the last expression  is slightly different to that in \cite[Proposition 6.2]{HM1}, nonetheless the one stated here follows from the proof in account of \cite[(6.14) and Proposition 6.1]{HM1} as $\partial Q_i$ is contained in $\overline{T_{Q_i}}$. Besides, Proposition \ref{prop:Pi-proj} easily yields
\begin{equation}\label{eqn:L6-12-HM1:new}
\Big(Q\setminus\bigcup_{Q_i\in\F}Q_i\Big)\bigcup
\Big(\bigcup_{Q_i\in\F: Q_i\subsetneq Q} P_i \Big)
\subset 
\Big(Q\setminus\bigcup_{Q_i\in\F}Q_i\Big)\bigcup
\Big(\bigcup_{Q_i\in\F: Q_i\subsetneq Q} \Delta_\star(x_i^\star,r_i) \Big)
\subset 
C\Delta_\star^Q,
\end{equation}
hence 
\begin{equation}
\omega_{L,*}^{Y_{Q_0}}\Big(\Big(Q\setminus\bigcup_{Q_i\in\F}Q_i\Big)\bigcup
\Big(\bigcup_{Q_i\in\F: Q_i\subsetneq Q} \Delta_\star(x_i^\star,r_i) \Big)\Big)
\lesssim
\omega_{L,*}^{Y_{Q_0}}(\Delta_\star^Q).
\end{equation}

Writing $E_0=Q_0\setminus \bigcup_{Q_i\in\F} Q_i\subset \pom\cap\partial\Omega_{\F,Q}$ (see \cite[Proposition 6.1]{HM1}) we have 
\begin{multline}\label{proj-aux-top}
\omega_{L,*}^{Y_{Q_0}}(\Delta_\star^Q)
\le
\omega_{L,*}^{Y_{Q_0}}(Q\cap E_0)
+
\sum_{Q_i\in\F: Q_i\subsetneq Q} \omega_{L,*}^{Y_{Q_0}}(\Delta_\star(x_i^\star,r_i))
\\ 
\lesssim
\omega_{L,*}^{Y_{Q_0}}(Q\cap E_0)
+
\sum_{Q_i\in\F: Q_i\subsetneq Q} \omega_{L,*}^{Y_{Q_0}}(P_i)
=
\mathcal{P}_{\mathcal{F}}^{\mu}\nu_L^{Y_{Q_0}}(Q)
\end{multline}
and, by \eqref{eqn:overlap-Pj},
\begin{multline}\label{proj-aux-top:1}
\mathcal{P}_{\mathcal{F}}^{\mu}\nu_L^{Y_{Q_0}}(Q)
=
\omega_{L,*}^{Y_{Q_0}}(Q\cap E_0)
+
\sum_{Q_i\in\F: Q_i\subsetneq Q}  \frac{\mu(Q\cap Q_i)}{\mu(Q_i)}\omega_{L,*}^{Y_{Q_0}}(P_i)
\\=
\omega_{L,*}^{Y_{Q_0}}(Q\cap E_0)
+
\sum_{Q_i\in\F: Q_i\subsetneq Q}  \omega_{L,*}^{Y_{Q_0}}(P_i)
\\
\lesssim
\omega_{L,*}^{Y_{Q_0}}\Big((Q\cap E_0)\bigcup\Big(\bigcup_{Q_i\in\F: Q_i\subsetneq Q} P_i\Big)\Big)
\lesssim
\omega_{L,*}^{Y_{Q_0}}(\Delta_\star^Q).
\end{multline}

Since $Q\in\dd_{\F,Q_0}$ we can invoke \cite[Proposition 6.4]{HM1} (which also holds in the current setting) to find $Y_Q\in\Omega_{\F,Q_0}$ which serves as a Corkscrew point simultaneously for $\Omega_{\F,Q_0}$ with respect to the surface ball $\Delta_\star(y_Q,s_Q)$ for some $y_Q\in\Omega_{\F,Q}$ and some $s_Q\approx\ell(Q)$, and for $\Omega$ with respect to each surface ball $\Delta(x,s_Q)$, for every $x\in Q$. Applying \eqref{chop-dyadic:densities} and Harnack's inequality to join  $Y_Q$ with $X_Q$ and $Y_{Q_0}$ with $Y_Q$ we have
\begin{equation}\label{chop-Proj-Omega}
\frac{d\omega_L^{Y_{Q}}}{d\omega_L^{Y_{Q_0}}}
\approx
\frac1{\omega_L^{Y_{Q_0}}(Q)},
\qquad
\mbox{$\omega_L^{Y_{Q_0}}$-a.e. in } Q.
\end{equation}
On the other hand one can see that 
\begin{equation}\label{proj-aux-cont}
\widetilde{B}_Q\bigcup \Big(\bigcup_{Q_i\in\F: Q_i\subsetneq Q} B(x_i^\star,r_i) \Big)
\subset 
B(y_Q,\widehat{s}_Q),
\end{equation}
for some $\widehat{s}_Q\approx s_Q$. Invoking then Proposition \ref{prop:CDC-inherit}, and Lemma \ref{lemma:proppde} parts $(c)$ and $(d)$ in the domain $\Omega_{\F,Q_0}$ we can analogously see 
\begin{equation}\label{chop-Proj-Omega-star}
\frac{d\omega_{L,*}^{Y_{Q}}}{d\omega_{L,*}^{Y_{Q_0}}}
\approx
\frac1{\omega_{L,*}^{Y_{Q_0}}(\Delta(y_Q,\widehat{s}_Q))}
\approx
\frac1{\omega_{L,*}^{Y_{Q_0}}(\Delta^Q_\star)},
\qquad
\mbox{$\omega_{L,*}^{Y_{Q_0}}$-a.e. in }\Delta(y_Q,\widehat{s}_Q).
\end{equation}

Next we invoke \eqref{proj-aux-top}, \eqref{proj-aux-cont}, and \eqref{chop-Proj-Omega} to obtain
\begin{multline}\label{proj-aux-111}
\frac{\mathcal{P}_{\mathcal{F}}^{\mu}\nu_L^{Y_{Q_0}}(F)}{\mathcal{P}_{\mathcal{F}}^{\mu}\nu_L^{Y_{Q_0}}(Q)}
\approx
\frac{\omega_{L,*}^{Y_{Q_0}}(F\cap E_0)}{\omega_{L,*}^{Y_{Q_0}}(\Delta_\star^Q)}
+
\sum_{Q_i\in\F: Q_i\subsetneq Q} 
\frac{\mu(F\cap Q_i)}{\mu(Q_i)} \frac{\omega_{L,*}^{Y_{Q_0}}(P_i)}{\omega_{L,*}^{Y_{Q_0}}(\Delta_\star^Q)}
\\
\approx
\omega_{L,*}^{Y_{Q}}(F\cap E_0)
+
\sum_{Q_i\in\F: Q_i\subsetneq Q} 
\frac{\mu(F\cap Q_i)}{\mu(Q_i)} \omega_{L,*}^{Y_{Q}}(P_i).
\end{multline}
We claim the following estimates hold
\begin{equation}\label{main-claim:proj}
\omega_{L,*}^{Y_{Q}}(F\cap E_0)\lesssim \omega_{L}^{Y_{Q}}(F\cap E_0), \qquad
\omega_{L,*}^{Y_{Q}}(P_i)\lesssim \omega_{L}^{Y_{Q}}(Q_i).
\end{equation}
The first estimate follows easily from the maximum principle since $\Omega_{\F,Q_0}\subset \Omega$ and $F\cap E_0\subset \pom\cap\partial \Omega_{\F,Q_0}$. For the second one, by the maximum principle we just need to see that $\omega_{L}^{X}(Q_i)\gtrsim 1$ for $X\in P_i$, but this follows from Lemma \ref{lemma:proppde} part $(a)$, \eqref{deltaQ}, Harnack's inequality, and \eqref{props:Pj}.

With the previous estimates at our disposal we can the continue with our estimate \eqref{proj-aux-111}: 
\begin{multline*}
\frac{\mathcal{P}_{\mathcal{F}}^{\mu}\nu_L^{Y_{Q_0}}(F)}{\mathcal{P}_{\mathcal{F}}^{\mu}\nu_L^{Y_{Q_0}}(Q)}
\lesssim
\omega_{L}^{Y_{Q}}(F\cap E_0)
+
\sum_{Q_i\in\F: Q_i\subsetneq Q} 
\frac{\mu(F\cap Q_i)}{\mu(Q_i)} \omega_{L}^{Y_{Q}}(Q_i)
\\
\approx
\frac{\omega_{L}^{Y_{Q_0}}(F\cap E_0)}{\omega_{L}^{Y_{Q_0}}(Q)}
+
\sum_{Q_i\in\F: Q_i\subsetneq Q} 
\frac{\mu(F\cap Q_i)}{\mu(Q_i)} \frac{\omega_{L}^{Y_{Q_0}}(Q_i)}{\omega_{L}^{Y_{Q_0}}(Q)}
\\
=
\frac{\mathcal{P}_{\mathcal{F}}^{\mu}\omega_L^{Y_{Q_0}}(F)}{\omega_{L}^{Y_{Q_0}}(Q)}
=
\frac{\mathcal{P}_{\mathcal{F}}^{\mu}\omega_L^{Y_{Q_0}}(F)}{\mathcal{P}_{\mathcal{F}}^{\mu}\omega_L^{Y_{Q_0}}(Q)},
\end{multline*}
where we have used \eqref{proj-aux-cont} and that ${\mathcal{P}_{\mathcal{F}}^{\mu}\omega_L^{Y_{Q_0}}(Q)}=\omega_L^{Y_{Q_0}}(Q)$. This proves the second estimate in \eqref{ainfsawtooth} in the current case.

Once we have shown the second estimate in \eqref{ainfsawtooth} we can invoke \cite[Lemma B.7]{HM1} (which is a purely dyadic result and hence applies in our setting) along with Lemma \ref{lemmma:P-nu-doubling} below to eventually obtain the first estimate in \eqref{ainfsawtooth}. 
\end{proof}

As a consequence of the previous result we can easily obtain a dyadic analog of the main lemma in \cite{DJK}.

\begin{lemma}[Discrete sawtooth lemma]\label{lemma:DJK-sawtooth:new}
	Suppose that $\Omega\subset\re^{n+1}$, $n\ge 2$, is a \textbf{bounded} 1-sided NTA domain satisfying the CDC. Let $Q_0\in\mathbb{D}$  and let $\mathcal{F}=\{Q_i\}\subset\mathbb{D}_{Q_0}$ be a family of pairwise disjoint dyadic cubes. Given two real (non-necessarily symmetric) elliptic $L_0$, $L$, we write  $\omega_0^{Y_{Q_0}}=\omega_{L_0,\Omega}^{Y_{Q_0}}$,  $\omega_L^{Y_{Q_0}}=\omega_{L,\Omega}^{Y_{Q_0}}$ for the elliptic measures associated with $L_0$ and $L$ for the domain $\Omega$  with fixed pole at $Y_{Q_0}\in\Omega_{\mathcal{F},Q_0}\cap\Omega$ (cf.~Lemma \ref{prop:Pi-proj}). Let $\omega_{L,*}^{Y_{Q_0}}=\omega_{L,\Omega_{\mathcal{F},Q_0}}^{Y_{Q_0}}$ be the elliptic measure associated with $L$ for the domain $\Omega_{\mathcal{F},Q_0}$ with fixed pole at  $Y_{Q_0}\in\Omega_{\mathcal{F},Q_0}\cap\Omega$. Consider $\nu_L^{Y_{Q_0}}$ the measure defined by \eqref{eq:def-nu}. Then, there exists $\theta>0$ such that for all $Q\in\mathbb{D}_{Q_0}$ and all $F\subset Q$, we have
	\begin{equation}\label{ainfsawtooth:new}
	\bigg(\frac{\omega_L^{Y_{Q_0}}(F)}{\omega_L^{Y_{Q_0}}(Q)}\bigg)^\theta
	\lesssim
	\frac{\nu_L^{Y_{Q_0}}(F)}{\nu_L^{Y_{Q_0}}(Q)}
	\lesssim
	\frac{\omega_L^{Y_{Q_0}}(F)}{\omega_L^{Y_{Q_0}}(Q)}.
	\end{equation}
	In particular, if $F\subset Q\setminus\bigcup_{Q_i\in\mathcal{F}}Q_i$, we have
	\begin{equation}\label{ainfsawtooth:new:new}
	\bigg(\frac{\omega_L^{Y_{Q_0}}(F)}{\omega_L^{Y_{Q_0}}(Q)}\bigg)^\theta
	\lesssim
	\frac{\omega_{L,*}^{Y_{Q_0}}(F)}{\omega_{L,*}^{Y_{Q_0}}(\Delta_\star^Q)}
	\lesssim
	\frac{\omega_L^{Y_{Q_0}}(F)}{\omega_L^{Y_{Q_0}}(Q)},
	\end{equation}
	where $\Delta_\star^Q:=B(x_Q^\star,t_Q)\cap\partial\Omega_{\F,Q_0}$ with $t_Q\approx\ell(Q)$, $x_Q^\star\in \partial\Omega_{\F,Q_0}$, and $\dist(Q, \Delta_\star^Q)\lesssim \ell(Q)$ with implicit constants depending on the allowable parameters (cf.~\cite[Proposition~6.12]{HM1}).
\end{lemma}

\begin{proof}
	Letting $\mu=\omega_L^{Y_{Q_0}}$, which is dyadically doubling in $Q_0$, one easily has $\mathcal{P}_{\mathcal{F}}^{\mu}\omega_L^{Y_{Q_0}}=\omega_L^{Y_{Q_0}}$ and $\mathcal{P}_{\mathcal{F}}^{\mu}\nu_L^{Y_{Q_0}}=\nu_L^{Y_{Q_0}}$. Thus \eqref{ainfsawtooth} in Lemma \ref{lemma:DJK-sawtooth}  readily yields \eqref{ainfsawtooth:new}. Next, to obtain \eqref{ainfsawtooth:new:new} we may assume that $F$ is non-empty. Observe that if $F\subset Q\setminus\bigcup_{Q_i\in\mathcal{F}}Q_i$, then $\nu_L^{Y_{Q_0}}(F)=\omega_{L,*}^{Y_{Q_0}}(F)$. On the other hand, if  $F\subset Q\setminus\bigcup_{Q_i\in\mathcal{F}}Q_i$ we must be in \textbf{Case 2} of the proof of Lemma \ref{lemma:DJK-sawtooth}, hence \eqref{proj-aux-top} and \eqref{proj-aux-top:1} hold. With all these we readily obtain \eqref{ainfsawtooth:new:new}.
\end{proof}

Our last result in this section establishes that both $\nu_L^{Y_{Q_0}}$ and $\mathcal{P}_{\mathcal{F}}^{\mu}\nu_L^{Y_{Q_0}}$ are dyadically doubling on $Q_0$.

\begin{lemma}\label{lemmma:P-nu-doubling}
	Under the assumptions of Lemma \ref{lemma:DJK-sawtooth}, $\nu_L^{Y_{Q_0}}$ and $\mathcal{P}_{\mathcal{F}}^{\mu}\nu_L^{Y_{Q_0}}$ are dyadically doubling on $Q_0$.
\end{lemma}

\begin{proof}
We follow the ideas in \cite[Lemma B.2]{HM1}. We shall first see $\nu_L^{Y_{Q_0}}$ is dyadically doubling. To this end, let $Q\in\mathbb{D}_{Q_0}$ be fixed and let $Q'$ be one of its dyadic children. We consider three cases:

\noindent \textbf{Case 1:} There exists $Q_i\in\mathcal{F}$ such that $Q\subset Q_i$. In this case we have 
\begin{align*}
\nu_L^{Y_{Q_0}}(Q)
=
\frac{\omega_L^{Y_{Q_0}}(Q)}{\omega_L^{Y_{Q_0}}(Q_i)}\omega_{L,*}^{Y_{Q_0}}(P_i)
\lesssim
\frac{\omega_L^{Y_{Q_0}}(Q')}{\omega_L^{Y_{Q_0}}(Q_i)}\omega_{L,*}^{Y_{Q_0}}(P_i)
=
\nu_L^{Y_{Q_0}}(Q')
\end{align*}
where we have used  Harnack's inequality and Lemma \ref{lemma:proppde} parts and $(a)$ and  $(c)$.

\medskip

\noindent \textbf{Case 2:} $Q'\in\mathcal{F}$. For simplicity say $Q'=Q_1\in\mathcal{F}$ and in this case $\nu_L^{Y_{Q_0}}(Q')=\omega_{L,*}^{Y_{Q_0}}(P_1)$. Note that then $Q\in\dd_{\F,Q_0}$ and we let $\mathcal{F}_1$ be the family of cubes $Q_i\in \F$ with $Q_i\cap Q\neq \emptyset$ and observe that if $Q_i\in\F_1$ then  $Q_i\subsetneq Q$. Then by \eqref{eqn:overlap-Pj}
\begin{multline}
\label{e:2}
\nu_L^{Y_{Q_0}}(Q)
=
\omega_{L,*}^{Y_{Q_0}}\Big(Q\setminus\bigcup_{Q_i\in\mathcal{F}}Q_i\Big)
+
\sum_{Q_i\in\mathcal{F}_1}\frac{\omega_L^{Y_{Q_0}}(Q\cap Q_i)}{\omega_L^{Y_{Q_0}}(Q_i)}\omega_{L,*}^{Y_{Q_0}}(P_i)
\\
=
\omega_{L,*}^{Y_{Q_0}}\Big(Q\setminus\bigcup_{Q_i\in\mathcal{F}}Q_i\Big)
+
\sum_{Q_i\in\mathcal{F}_1} \omega_{L,*}^{Y_{Q_0}}(P_i)
\lesssim
\omega_{L,*}^{Y_{Q_0}}\Big(
\Big(Q\setminus\bigcup_{Q_i\in\mathcal{F}}Q_i\Big)\bigcup \Big(\bigcup_{Q_i\in\mathcal{F}_1} P_i\Big)
\Big).
\end{multline}
Recall that in \textbf{Case 2} in the proof of Lemma \ref{lemma:DJK-sawtooth} we mentioned that $P_1\subset \Delta_\star(x_1^\star,r_1)$ with $x_1^\star$ being the center of $P_1$ and $r_1\approx \ell(P_1)\approx \ell(Q_1)\approx \ell(Q)$ since $Q$ is the dyadic parent of $Q_1$. Note that since $Q_i\in\F_1$ by \eqref{props:Pj}
\[
\ell(P_i)\approx \dist(P_i, Q)\approx \ell(Q_i)\lesssim \ell(Q)=2\ell(Q_1)\approx  \ell(P_1)\approx \dist(Q_1, P_1)\approx r_1. 
\]
Thus 
\[
\Big(Q\setminus\bigcup_{Q_i\in\mathcal{F}}Q_i\Big)\bigcup \Big(\bigcup_{Q_i\in\mathcal{F}_1} P_i\Big)\subset \Delta_{\star}(x_1^\star, Cr_1),
\]
where we here and below we use the notation $\Delta_\star$ for the surface balls with respect to $\partial\Omega_{\F,Q_0}$.
Using this, \eqref{e:2}, and Lemma \ref{lemma:proppde} parts $(a)$ and $(c)$ and Harnack's inequality we derive
\[
\nu_L^{Y_{Q_0}}(Q)
\lesssim
\omega_{L,*}^{Y_{Q_0}}(\Delta_{\star}(x_1^\star, Cr_1))
\lesssim
\omega_{L,*}^{Y_{Q_0}}(\Delta_\star(x_1^\star,r_1))
\lesssim \omega_{L,*}^{Y_{Q_0}}(P_1)
=
\nu_L^{Y_{Q_0}}(Q').
\]

\medskip

\noindent \textbf{Case 3:} None of the conditions in the previous cases happen, and necessarily  $Q,Q'\in\dd_{\F,Q_0}$. We take the same set $\mathcal{F}_1$ as in the previous case and again if $Q_i\in\F_1$ then  $Q_i\subsetneq Q$ (otherwise we are driven to \textbf{Case 1}). Introduce $\mathcal{F}_2$, the family of cubes $Q_i\in \mathcal{F}$ with $Q_i\cap Q'\neq\emptyset$. Again, if $Q_i\in\mathcal{F}_2$ we have $Q_i\subsetneq Q'$; otherwise either $Q'=Q_i$ which is \textbf{Case 2}, or $Q'\subsetneq Q_i$ which implies $Q\subset Q_i$ and we are back to \textbf{Case 1}.

Note that since $Q$ is the dyadic parent of $Q'$, using the same notation as in \eqref{eqn:L6-12-HM1} applied to $Q'\in \dd_{\F,Q_0}$ we have that
\[
\dist(x^\star_{Q'},Q)
\le
\dist(x^\star_{Q'},Q')
\lesssim 
\ell(Q')
\approx
\ell(Q)
\approx t_{Q'}.
\]
Also by \eqref{props:Pj}
\[
\dist(x^\star_{Q'},P_i)
\lesssim
\dist(x^\star_{Q'},Q)
+
\ell(Q)
+\dist(Q, P_i)
\lesssim
\ell(Q)
+
\dist(Q_i, P_i)
\lesssim
\ell(Q)
\approx
t_{Q'}.
\]
These readily give
\[
\Big(Q\setminus\bigcup_{Q_i\in\mathcal{F}}Q_i\Big)\bigcup \Big(\bigcup_{Q_i\in\mathcal{F}_1} P_i\Big)\subset \Delta_{\star}(x_{Q'}^\star, Ct_{Q'}).
\]
We can then proceed as in the previous case (see \eqref{e:2}) to obtain
\begin{align*}
\nu_L^{Y_{Q_0}}(Q)
\lesssim
\omega_{L,*}^{Y_{Q_0}}\Big(
\Big(Q\setminus\bigcup_{Q_i\in\mathcal{F}}Q_i\Big)\bigcup \Big(\bigcup_{Q_i\in\mathcal{F}_1} P_i\Big)
\Big)
\lesssim
\omega_{L,*}^{Y_{Q_0}}(\Delta_{\star}(x_{Q'}^\star, Ct_{Q'}))
\lesssim
\omega_{L,*}^{Y_{Q_0}}(\Delta_{\star}^{Q'})
\end{align*}
where $\Delta_{\star}^{Q'}=B(x_{Q'}^\star,t_{Q'})\cap\partial\Omega_{\F,Q_0}$ (see \eqref{eqn:L6-12-HM1}) and we have used 
Lemma \ref{lemma:proppde} parts $(a)$ and $(c)$ and Harnack's inequality. On the other hand, proceeding as in \eqref{proj-aux-top} with $Q'$ in place of $Q$ since $Q'\in\dd_{\F,Q_0}$:
\begin{multline*}
\omega_{L,*}^{Y_{Q_0}}(\Delta_\star^{Q'})
\le
\omega_{L,*}^{Y_{Q_0}}(Q'\cap E_0)
+
\sum_{Q_i\in\F_2} \omega_{L,*}^{Y_{Q_0}}(\Delta_\star(x_i^\star,r_i))
\\\lesssim
\omega_{L,*}^{Y_{Q_0}}(Q'\cap E_0)
+
\sum_{Q_i\in\F_2} \omega_{L,*}^{Y_{Q_0}}(P_i)
\\
=
\omega_{L,*}^{Y_{Q_0}}(Q'\cap E_0)
+
\sum_{Q_i\in\F_2} \frac{\omega_L^{Y_{Q_0}}(Q'\cap Q_i)}{\omega_L^{Y_{Q_0}}(Q_i)}\omega_{L,*}^{Y_{Q_0}}(P_i)
=
\nu_L^{Y_{Q_0}}(Q').
\end{multline*}
Eventually we obtain that $\nu_L^{Y_{Q_0}}(Q)\lesssim \nu_L^{Y_{Q_0}}(Q')$, completing the proof of the dyadic doubling property of $\nu_L^{Y_{Q_0}}$.

We next deal with $\mathcal{P}_{\mathcal{F}}^{\mu}\nu_L^{Y_{Q_0}}$. We can simply follow the previous argument replacing $\omega_L^{Y_{Q_0}}$ by $\mathcal{P}_{\mathcal{F}}^{\mu}\nu_L^{Y_{Q_0}}$ to see that in \textbf{Cases 2} and \textbf{3} we have that $\mathcal{P}_{\mathcal{F}}^{\mu}\nu_L^{Y_{Q_0}}(Q)=\nu_L^{Y_{Q_0}}(Q)$ and $\mathcal{P}_{\mathcal{F}}^{\mu}\nu_L^{Y_{Q_0}}(Q')=\nu_L^{Y_{Q_0}}(Q')$, hence the doubling condition follows from the previous calculations and the constant depend on that of $\omega_{L,*}^{Y_{Q_0}}$. With regard to \textbf{Cases 1},  on which $Q\subset Q_i$ for some $Q_i\in\F$, one can easily see that
\begin{align*}
\mathcal{P}_{\mathcal{F}}^{\mu} \nu_L^{Y_{Q_0}}(Q)
=
\frac{\mu(Q)}{\mu(Q_i)}\omega_{L,*}^{Y_{Q_0}}(P_i)
\lesssim
\frac{\mu(Q')}{\mu(Q_i)}\omega_{L,*}^{Y_{Q_0}}(P_i)
=
\mathcal{P}_{\mathcal{F}}^{\mu}\nu_L^{Y_{Q_0}}(Q'),
\end{align*}
which uses that $\mu$ is dyadically doubling in $Q_0$. Eventually we have seen that doubling constant depend on that of $\omega_{L,*}^{Y_{Q_0}}$ and $\mu$ as desired. This completes the proof. 
\end{proof}

\section{Proof of the main results}\label{section:SF-NT}

\subsection{Proof of Theorem \ref{THEOR:CME}}

By renormalization we may assume without loss of generality that $\|u\|_{L^\infty(\Omega)}=1$.	
We will first prove a dyadic version of \eqref{eq:CME}. Let $\dd=\dd(\pom)$ the dyadic grid from Lemma \ref{lemma:dyadiccubes} with $E=\pom$. Our goal is to show that 
\begin{equation}\label{eq:CME:dyadic}
M_0:=\sup_{Q^0\in\dd} \sup_{\substack{Q_0\in\dd_{Q^0}\\ \ell(Q_0)\le \frac{\ell(Q^0)}{M}}} \frac1{\omega_L^{X_{Q^0}}(Q_0)}\iint_{T_{Q_0}} |\nabla u(X)|^2\, G_L(X_{Q^0},X)\,dX
\lesssim 1
\end{equation}
with $M\ge 4$ large enough. Assuming this momentarily let us see how to derive \eqref{eq:CME}. Fix $B$ and $B'$ as in the suprema in \eqref{eq:CME}. Let $k, k'\in\ZZ$ be so that $2^{k-1}<r\le 2^k$ and  $2^{k'-1}<r'\le 2^{k'}$, and define $k'':=\min\{k', k-10k_M\}$ where $k_M\ge 1$ is large enough to be chosen depending on $M$ and the allowable parameters. Set 
\begin{equation*}
\W'
:=
\{I\in\W: I\cap B'\neq\emptyset,\ell(I)<2^{k''}\}
\bigcup
\{I\in\W: I\cap B'\neq\emptyset,\ell(I)\ge 2^{k''}\}
=:
\W'_1\cup\W'_2.
\end{equation*}
Note that for every $I\in \W$ with $I\cap B'\neq\emptyset$ we have
\[
\ell(I)<\diam(I)\le \frac{\dist(I,\pom)}{4}< \frac{r'}4\le 2^{k'-2}.
\]
As a consequence, if $\W_2'\neq\emptyset$, then $k''=k-10\,k_M$, and picking $I\in \W_2'\neq\emptyset$ one has
\[
r\approx 2^k\approx_M 2^{k''}\le  \ell(I)\le   2^{k'-2}\approx r'\lesssim r.
\]
This gives $r'\approx_M r$ and $\#\W_2'\lesssim_M 1$.

To proceed, let us write
\begin{multline*}
\iint_{B'\cap\Omega} |\nabla u(X)|^2\, G_L(X_\Delta,X)\,dX
\le
\iint_{\bigcup\limits_{I\in\W'_1} I} |\nabla u(X)|^2\, G_L(X_\Delta,X)\,dX
\\
+
\sum_{I\in\W'_2} \iint_{I} |\nabla u(X)|^2\, G_L(X_\Delta,X)\,dX
=:
\mathcal{I}+\mathcal{II},
\end{multline*}
and we estimate each term in turn. 

To estimate $\mathcal{II}$ we may assume that $\W'_2\neq\emptyset$, hence $k''=k-10\,k_M$, $r'\approx r$ and $\#\W_2'\lesssim 1$. Then Lemma \ref{lemma:proppde}, the fact that  $\omega(\pom)\le 1$, Caccioppoli's inequality, the normalization $\|u\|_{L^\infty(\Omega)}=1$, and Harnack's inequality give
\begin{multline*}
\mathcal{II}
=
\sum_{I\in\W'_2} \iint_{I} |\nabla u(X)|^2\, G_L(X_\Delta,X)\,dX
\lesssim
\sum_{I\in\W'_2} \ell(I)^{1-n}\iint_{I} |\nabla u(X)|^2\,dX
\\
\lesssim\# \W'_2
\lesssim
1
\approx
\omega_L^{X_{\Delta}}(\Delta').
\end{multline*}

Next we deal with $\mathcal{I}$. Introduce the disjoint family $\F'=\{Q\in\dd: \ell(Q)=2^{k''-1}, Q\cap 3\,B'\neq\emptyset\}$. 
 Given $I\in\W'_1$, let  $X_I\in B'\cap I$, and $Q_I\in\dd$  be so that $\ell(Q_I)=\ell(I)$ and  it contains some fixed $y_I\in\partial\Omega$ such that $\dist(I,\partial\Omega)=\dist(I,y_I)$. Then, as observed in Section \ref{subsection:sawtooth}, one has $I\in \W_{Q_I}^*$. Note that
\begin{equation*}
|y_I-x'|\le\dist(y_I, I)+\diam(I)+|X_I-x'|\le \frac54 \dist(I,\pom)+|X_I-x'|
\le \frac94|X_I-x'|<3\,r',
\end{equation*}
hence $y_I\in Q_I\cap 3\,\Delta'$. This and the fact that, as observed before, $\ell(Q_I)=\ell(I)<2^{k''}$  imply that $Q_I\subset Q$ for some $Q\in\F'$. Hence,
$I\subset (1+\lambda)\,I\subset U_{Q_I}\subset \overline{T_{Q}}$ for some $Q\in\F'$. This eventually show that $\bigcup_{I\in\W'_1} I\subset\bigcup_{Q\in\F'} T_Q$ and therefore
\[
\mathcal{I}
\le
\sum_{Q\in\F'}
\iint_{T_Q} |\nabla u(X)|^2\, G_L(X_\Delta,X)\,dX.
\]
For any $Q\in \F'$ pick the unique (ancestor) $\widehat{Q}\in \dd$ with $\ell(\widehat{Q})=2^{k-1}$ and $Q\subset \widehat{Q}$. Note that $\delta(X_\Delta)\approx r$, $\delta(X_{\widehat{Q}})\approx \ell(\widehat{Q})=2^{k-1}\approx r$. Also,
\begin{multline*}
|X_\Delta-X_{\widehat{Q}}|
\le
|X_\Delta-x|
+
|x-x'|+
|x'-x_{Q}|
+
|x_Q-x_{\widehat{Q}}|
+|x_{\widehat{Q}}-X_{\widehat{Q}}|
\\
<
3\,r+ 3\,r'+\diam(Q)+\diam(\widehat{Q})+\ell(\widehat{Q})
\lesssim
r+2^{k''}+2^k
\lesssim 
r.
\end{multline*}
Hence by the Harnack chain condition one obtains that $G_L(X_\Delta,X)\approx G_L(X_{\widehat{Q}},X)$ for every $X\in T_Q$ (in doing that we need to make sure that $k_M$ is large enough so that the Harnack chain joining $X_\Delta$ and $X_{\widehat{Q}}$, which is $c\,r$-away from $\pom$, does not get near $T_Q$, which is $\kappa_0\,\ell(Q)$-close to $\pom$).  Note also that $\frac{\ell(Q)}{\ell(\widehat{Q})}=2^{k''-k}\le 2^{-k_M}< M^{-1}$, provided $k_M$ is large enough depending on $M$. All theres and \eqref{eq:CME:dyadic} yield
\begin{multline*}
\mathcal{I}
\lesssim
\sum_{Q\in\F'}
\iint_{T_Q} |\nabla u(X)|^2\, G_L(X_{\widehat{Q}},X)\,dX
\lesssim
M_0\,\sum_{Q\in\F'} \omega_L^{X_{\widehat{Q}}}(Q)
\\
\lesssim
M_0\,
\sum_{Q\in\F'} \omega_L^{X_{\Delta}}(Q)
\le
M_0\,\omega_L^{X_{\Delta}}\Big( \bigcup_{Q\in\F'} Q  \Big )
\le
M_0\,\omega_L^{X_{\Delta}}(C\Delta')
	\lesssim
M_0\,\omega_L^{X_{\Delta}}(\Delta'),
\end{multline*}
where we have used Lemma \ref{lemma:proppde}. This completes the prof of the fact that \eqref{eq:CME:dyadic} implies \eqref{eq:CME}.

We next focus on showing \eqref{eq:CME:dyadic}. With this goal in mind we fix $Q^0\in\dd=\dd(\pom)$ and let $Q_0\in\dd_{Q^0}$ with $\ell(Q_0)\le \ell(Q^0)/M$ with $M$ large enough so that $X_{Q^0}\notin 4\, B_Q^*$ (cf. \eqref{definicionkappa12}). Write $\omega_L=\omega_L^{X_{Q^0}}$ and $\mathcal{G}_L=G_L(X_{Q^0},\cdot)$ and note that our choice of $M$,  \eqref{G-G-top}, and \eqref{eq:G-delta} guarantee that 
$L^\top \mathcal{G}_L= L^\top G_{L^\top}(\cdot,X_{Q^0}) =0$ in the weak sense in $4\,B_Q^*$.

Fix $N\gg 1$  and consider the family of pairwise disjoint cubes $\mathcal{F}_N=\{Q\in\dd_{Q_0}: \ell(Q)=2^{-N}\,\ell(Q_0)\}$ and let $\Omega_N=\Omega_{\F_N,Q_0}$ (cf. \eqref{defomegafq}). Note that by construction $\Omega_{N}\subset T_{Q_{0}}$ is an increasing sequence of sets  converging to  $T_{Q_0}$. 
Our goal is to show that for every $N\gg 1$ there holds
\begin{equation}\label{eq:CME-N-dya}
\iint_{\Omega_N} |\nabla u(X)|^2\, \mathcal{G}_L(X)\,dX\le M_0\, \omega_L(Q_0),
\end{equation}
with $M_0$ independent of $Q^0$, $Q_0$, and $N$. Hence the monotone convergence theorem yields
\[
\iint_{T_{Q_0}} |\nabla u(X)|^2\, \mathcal{G}_L(X)\,dX
=
\lim_{N\to\infty} \iint_{\Omega_N} |\nabla u(X)|^2\, \mathcal{G}_L(X)\,dX\le M_0\, \omega_L(Q_0),
\]
which is \eqref{eq:CME:dyadic}.

Let us next start estimating \eqref{eq:CME-N-dya}. Using $\Psi_N$ from Lemma \ref{lemma:approx-saw} and the ellipticity of the matrix $A$ we have
\begin{align*}
&\iint_{\Omega_N} |\nabla u(X)|^2\, \mathcal{G}_L(X)\,dX
\lesssim
\iint_{\ree} |\nabla u(X) |^2\, \mathcal{G}_L(X)\,\Psi_N(X)\,dX
\\
&\qquad\lesssim
\iint_{\ree} A(X)\nabla u(X)\cdot \nabla u(X)\,\mathcal{G}_L(X)\,\Psi_N(X)\,dX
\\
&
\qquad=
\iint_{\ree} A(X)\nabla u(X)\cdot \nabla (u \,\mathcal{G}_L\,\Psi_N)(X)\,dX
\\
&
\qquad\qquad
-
\frac12\iint_{\ree} A(X)\nabla (u^2\, \Psi_N)(X)\cdot \nabla \mathcal{G}_L(X)\,dX 
\\
&
\qquad\qquad-
\frac12\iint_{\ree} A(X)\nabla (u^2)(X)\cdot \nabla\Psi_N(X)\, \mathcal{G}_L(X)\,dX 
\\
&
\qquad\qquad+
\frac12\iint_{\ree} A(X)\nabla\Psi_N(X)\cdot \nabla\mathcal{G}_L(X)\,u(X)^2\,dX 
\\
&
\qquad=:
\mathcal{I}_1+\mathcal{I}_2+\mathcal{I}_3+\mathcal{I}_4.
\end{align*}
We observe that $u \,\mathcal{G}_L\,\Psi_N$ and $u^2\, \Psi_N$ belong to $W^{1,2}(\Omega)$ since $u\in W^{1,2}_{\rm loc} (\Omega)\cap L^\infty(\Omega)$, $\supp \Psi_N\subset \Omega_N^*$, $\delta(X)\gtrsim 2^{-N}\,\ell(Q_0)$ for every $X\in \Omega_N^*$, the properties of $G_L$, and the fact that $X_{Q^0}$ is away from $\Omega_N^*$ ---$\delta(X_{Q^0})\approx \ell(Q^0)$ and by \eqref{definicionkappa12} one has $\delta(X)\lesssim \ell(Q_0)\le \ell(Q^0)/M\le \delta(X_{Q^0})/2$ for every $X\in {\Omega}_N^*$ and provided $M$ is large enough. Using all these one can easily see via a limiting argument that the fact that $Lu=0$ in the weak sense in $\Omega$ implies that $\mathcal{I}_1=0$. Likewise, one can easily show that $\mathcal{I}_2=0$  by recalling that $\supp \Psi_N\subset \Omega_N^*\subset \frac12\,{B_Q}^*\cap\Omega$
(see \eqref{definicionkappa12}) and that as mentioned above $L^\top \mathcal{G}_L=0$ in the weak sense in $4\,B_Q^*$. Thus we are left with estimating the terms $\mathcal{I}_3$ and $\mathcal{I}_4$. By $(iii)$ in Lemma \ref{lemma:approx-saw} and the fact that  $\|u\|_{L^\infty(\Omega)}=1$ we obtain
\begin{align*}
|\mathcal{I}_3|+|\mathcal{I}_4|
&\lesssim
\iint_{\bigcup_{I\in\W_N^\Sigma} I^{**}} \big(|\nabla u|\,\mathcal{G}_L+ |\nabla \mathcal{G}_L|\big)\,\delta(\cdot)^{-1}\,dX
\\
&
\lesssim
\sum_{I\in\W_N^\Sigma} \ell(I)^{\frac{n-1}{2}}\,\Big(\Big(\iint_{I^{**}} |\nabla u|^2\,dX\Big)^\frac12\, \mathcal{G}_L(X(I))+ \Big(\iint_{I^{**}} |\nabla \mathcal{G}_L|^2\,dX\Big)^\frac12\Big)
\\
&
\lesssim
\sum_{I\in\W_N^\Sigma} \ell(I)^{\frac{n-3}{2}}\,\Big(\Big(\int_{I^{***}} |u|^2\,dX\Big)^\frac12 +\ell(I)^{\frac{n+1}{2}}\,\Big) \mathcal{G}_L(X(I))
\\
&
\lesssim
\sum_{I\in\W_N^\Sigma} \ell(I)^{n-1}\,\mathcal{G}_L(X(I)),
\end{align*}
where $X(I)$ denotes the center of $I$, and we have used Harnack's and Caccioppoli's inequalities, that $L^\top\mathcal{G}_L=0$ and $Lu=0$ in the weak sense in $I^{***}\subset \frac12\,B_Q^*\cap\Omega$ (see \eqref{definicionkappa12}).  Invoking Lemmas \ref{lemma:proppde} and Lemma \ref{lemma:approx-saw} one can see that  $\ell(I)^{n-1}\,\mathcal{G}_L(X(I))\lesssim \omega_L(\widehat{Q}_I)$ for every $I\in\W_N^\Sigma$. This together with Lemma \ref{lemma:approx-saw} allows us to conclude
\[
|\mathcal{I}_3|+|\mathcal{I}_4|
\lesssim
\sum_{I\in\W_N^\Sigma} \omega_L(\widehat{Q}_I)
\lesssim
\omega_L\Big(\bigcup_{I\in\W_N^\Sigma}\widehat{Q}_I\Big).
\]
Note that if $y\in \widehat{Q}_I$ with $I\in\W_N^\Sigma$ one has
\[
|y-x_{Q_0}|
\le
\diam(\widehat{Q}_I)
+
\dist(\widehat{Q}_I,I)+\diam(I)+\dist(I,x_{Q_0})
\lesssim
\ell(I)+\ell(Q_0)
\lesssim
\ell(Q_0)
\]
where we have used \eqref{new-QI} and \eqref{definicionkappa12}. Thus, Lemma \ref{lemma:proppde} gives 
\[
|\mathcal{I}_3|+|\mathcal{I}_4|
\lesssim 
\omega_L(C\,\Delta_{Q_0})
\lesssim
\omega_L(Q_0).
\]
This allows us to complete the proof of Theorem \ref{THEOR:CME}. \qed

\subsection{Proof of Theorem \ref{THEOR:S<N}}
We borrow some ideas from \cite{HMM-trans}. Given $k\in\N$ introduce the truncated localized conical square function: for every $Q\in\dd_{Q_0}$ and $x\in Q$, let 
\[
\mathcal{S}_{Q}^ku(x):=\bigg(\iint_{\Gamma_{Q}^k(x)}|\nabla u(Y)|^2\delta(Y)^{1-n}\,dY\bigg)^{\frac12}, \quad 
\text{where}\  \Gamma_{Q}^k(x) := \bigcup_{\substack{x\in Q'\in\mathbb{D}_{Q} \\ \ell(Q')\ge 2^{-k}\,\ell(Q_0)}}  U_{Q'},
\]
where if $\ell(Q)<2^{-k}\,\ell(Q_0)$ it is understood that $\Gamma_{Q}^k(x)=\emptyset$ and $\mathcal{S}_Q^ku(x)= 0$. Note that by the monotone convergence theorem $\mathcal{S}_{Q}^k u(x)\nearrow \mathcal{S}_{Q}u(x)$ as $k\to\infty$ for every $x\in Q$.

Fixed $k_0$ large enough (eventually, $k_0\to\infty$), our goal is to show that we can find $\vartheta>0$ (independent of $k_0$)  such that for every $\beta,\gamma, \lambda >0$ we have
\begin{multline}\label{good-lambda:goal}
\omega_L^{X_{Q_0}}\big(\big\{x\in Q_0: \mathcal{S}_{Q_0}^{k_0}u(x)>(1+\beta)\,\lambda,\ \mathcal{N}_{Q_0}u(x)\le \gamma\,\lambda \big\}\big)
\\
\lesssim
\Big(\frac{\gamma}{\beta}\Big)^\vartheta\, \omega_L^{X_{Q_0}}\big(\big\{x\in Q_0: \mathcal{S}_{Q_0}^{k_0}u(x)>\beta\,\lambda\big\}\big),
\end{multline}
where the implicit constant depend on the allowable parameters and it is independent of $k_0$.
To prove this we fix $\beta,\gamma, \lambda >0$ and set
\[
E_\lambda:=\big\{x\in Q_0: \mathcal{S}_{Q_0}^{k_0} u(x)>\lambda\big\}.
\]

Consider first the case $E_\lambda\subsetneq Q_0$. Note that if $x\in E_\lambda$, by definition $\mathcal{S}_{Q_0}^{k_0}u (x)>\lambda$. Let $Q_x\in\dd_{Q_0}$ be the unique dyadic cube such that $Q_x\ni x$ and $\ell(Q_x)=2^{-k_0}\ell(Q_0) $. Then it is clear from construction that for every $y\in Q_x$ one has 
\[
\Gamma_{Q_0}^{k_0}(x)=\bigcup_{Q_x\subset Q\subset Q_0} U_Q =\Gamma_{Q_0}^{k_0}(y)
\qquad\mbox{and}\qquad
\lambda< \mathcal{S}_{Q_0}^{k_0}u (x)= \mathcal{S}_{Q_0}^{k_0}u (y).
\]
Hence,  $Q_x\subset E_\lambda$ and we have shown that for every $x\in E_\lambda$ there exists $Q_x\in\dd_{Q_0}$ such that $Q_x\ni x$ and $Q_x\subset E_{\lambda}$. 
We then take the ancestors of $Q_x$,  and look for the one with maximal side length $Q_x^{\rm max}\supset Q_x$ which is contained in $E_\lambda$. That is, $Q\subset E_\lambda$ for every $Q_x\subset Q\subset Q_x^{\rm max}$
and $\widehat{Q}_x^{\rm max}\cap Q_0\setminus E_\lambda\neq\emptyset$ where $\widehat{Q}_x^{\rm max}$ is the dyadic parent of $Q_x^{\rm max}$ (during this proof we will use $\widehat{Q}$ to denote the dyadic parent of $Q$, that is, the only dyadic cube containing it with double side length). Note that the assumption $E_\lambda\subsetneq Q_0$ guarantees that $Q_x^{\rm max}\in\dd_{Q_0}\setminus\{Q_0\}$. Let $\F_0=\{Q_j\}_j$ be the collection of such maximal cubes as $x$ runs in $E_\lambda$ and we clearly have that the family is pairwise disjoint and also $E_\lambda=\bigcup_{Q_j\in\F_0} Q_j$. Also, by construction $\ell(Q_j)\ge 2^{-k_0}\ell(Q_0)$ and by the maximality of each $Q_j$ we can select $x_j\in \widehat{Q}_j\setminus E_\lambda$.

On the other hand, for any  $x\in Q_j$ we have, using that $x_j\in \widehat{Q}_j\setminus E_\lambda$,
\[
\Gamma_{Q_0}^{k_0}(x) 
= 
\bigcup_{\substack{x\in Q\in\mathbb{D}_{Q_0} \\ \ell(Q)\ge 2^{-k_0}\,\ell(Q_0)}} U_{Q}
=
\Gamma_{Q_j}^{k_0}(x) \bigcup \Big(\bigcup_{Q_j\subsetneq Q\subset Q_0}  U_{Q}\Big)
\subset 
\Gamma_{Q_j}^{k_0}(x) \bigcup \Gamma_{Q_0}^{k_0}(x_j)
\]
and therefore
\[
\mathcal{S}_{Q_0}^{k_0}u(x) \le \mathcal{S}_{Q_j}^{k_0}u(x)+\mathcal{S}_{Q_0}^{k_0} u(x_j) \le \mathcal{S}_{Q_j}^{k_0}u(x)+\lambda
.
\]
As a consequence, 
\[
\big\{x\in Q_j: \mathcal{S}_{Q_0}^{k_0}u(x)>(1+\beta)\,\lambda\big\}
\subset  \big\{x\in Q_j: \mathcal{S}_{Q_j}^{k_0}u(x)>\beta\lambda\big\}
\]
and
\begin{multline*}
\big\{x\in Q_0: \mathcal{S}_{Q_0}^{k_0}u(x)>(1+\beta)\,\lambda\big\}
=
\big\{x\in Q_0: \mathcal{S}_{Q_0}^{k_0}u(x)>(1+\beta)\,\lambda\big\}\cap E_\lambda
\\
=
\bigcup_{Q_j\in\F_0} \big\{x\in Q_j: \mathcal{S}_{Q_0}^{k_0}u(x)>(1+\beta)\,\lambda\big\}
\subset  
\bigcup_{Q_j\in\F_0}\big\{x\in Q_j: \mathcal{S}_{Q_j}^{k_0}u(x)>\beta\lambda\big\}.
\end{multline*}
This has been done under the assumption that $E_\lambda\subsetneq Q_0$. In the case $E_\lambda=Q_0$ we set $\F_0=\{Q_0\}$. Then in both cases we obtain 
\begin{align}\label{good-lambda:1}
\big\{x\in Q_0: \mathcal{S}_{Q_0}^{k_0}u(x)>(1+\beta)\,\lambda\big\}
\subset  
\bigcup_{Q_j\in\F_0}\big\{x\in Q_j: \mathcal{S}_{Q_j}^{k_0}u(x)>\beta\lambda\big\}.
\end{align}
Thus, to obtain \eqref{good-lambda:goal} it suffices to see that for every $Q_j\in\F_0$
\begin{align}\label{good-lambda:goal:local}
\omega_L^{X_{Q_0}}\big(\big\{x\in Q_j: \mathcal{S}_{Q_j}^{k_0}u(x)>\beta\,\lambda,\ \mathcal{N}_{Q_0}u(x)\le \gamma\,\lambda \big\}\big)
\lesssim
\Big(\frac{\gamma}{\beta}\Big)^\vartheta\, \omega_L^{X_{Q_0}}(Q_j).
\end{align}
From this we just need to sum in $Q_j\in\F_0$ to see that \eqref{good-lambda:1} together with the previous facts yield the desired estimate \eqref{good-lambda:goal}:
\begin{multline*}
\omega_L^{X_{Q_0}}\big(\big\{x\in Q_0: \mathcal{S}_{Q_0}^{k_0}u(x)>(1+\beta)\,\lambda,\ \mathcal{N}_{Q_0}u(x)\le \gamma\,\lambda \big\}\big)
\\
\le
\sum_{Q_j\in\F_0}\omega_L^{X_{Q_0}}\big(\big\{x\in Q_j: \mathcal{S}_{Q_j}^{k_0}u(x)>\beta\,\lambda,\ \mathcal{N}_{Q_0}u(x)\le \gamma\,\lambda \big\}\big)
\\
\lesssim
\Big(\frac{\gamma}{\beta}\Big)^\vartheta\, \sum_{Q_j\in\F_0}\omega_L^{X_{Q_0}}(Q_j)
=
\Big(\frac{\gamma}{\beta}\Big)^\vartheta\,
\omega_L^{X_{Q_0}}\Big( \bigcup_{Q_j\in\F_0} Q_j\Big)
=
\Big(\frac{\gamma}{\beta}\Big)^\vartheta\,\omega_L^{X_{Q_0}}(E_\lambda).
\end{multline*}

Let us then obtain \eqref{good-lambda:goal:local}. Fix $Q_j\in \F_0$ and to ease the notation write $P_0=Q_j$. Set
\begin{equation}\label{def:good-lambda-sets}
\widetilde{E}_\lambda=\big\{x\in P_0: \mathcal{S}_{P_0}^{k_0}u(x)> \beta \,\lambda \big\},
\qquad
F_\lambda=\big\{x\in P_0: \mathcal{N}_{Q_0}u(x)\le \gamma\,\lambda \big\}.
\end{equation}
If $\omega_L^{X_{Q_0}}(F_\lambda)=0$ then \eqref{good-lambda:goal:local} is trivial, hence we may assume that $\omega_L^{X_{Q_0}}(F_\lambda)>0$ so that $P_0\cap F_\lambda=F_\lambda\neq\emptyset$. We subdivide $P_0$ dyadically and stop the first time that $Q\cap F_\lambda=\emptyset$. If one never stops we write $\F_{P_0}^*=\{\emptyset\}$, otherwise $\F_{P_0}^*=\{P_j\}_j\subset\dd_{P_0}\setminus\{P_0\}$ is the family of stopping cubes which is maximal (hence pairwise disjoint) with respect to the property $F_\lambda\cap Q=\emptyset$. In particular, $F_\lambda\subset P_0\setminus(\cup_{\dd_{\F_{P_0}^*, P_0}} P_j)$.

Next we claim that
\begin{equation}\label{cones-sawtooth}
\bigcup_{x\in F_\lambda} \Gamma_{P_0}^{k_0}(x)
\subset
\bigcup_{\substack{Q\in\dd_{\F_{P_0}^*, P_0} \\ \ell(Q)\ge 2^{-k_0}\,\ell(Q_0)}} U_Q
\subset
{\rm int } \bigg(\bigcup_{Q\in\dd_{\F_{P_0}^*, P_0}}
 U_Q^*\bigg)
=
\Omega_{\F_{P_0}^*, P_0}^*
=:
\Omega_*
\end{equation}
To verify the first inclusion, we fix $Y\in \Gamma^{k_0}_{P_0}(x)$ with $x\in F_\lambda$. Then, $Y\in U_Q$  where $x\in Q\in\dd_{P_0}$. Since $x\in F_\lambda$ we must have $Q\in \dd_{\F_{P_0}^*}$ (otherwise $Q\subset P_j$ for some $P_j\in\F_{P_0}^*$ and this would imply that $x\in P_j\cap F_\lambda=\emptyset$) and therefore $Q\in \dd_{\F_{P_0}^*,P_0}$ which gives the first inclusion. The second inclusion in \eqref{cones-sawtooth} is trivial (since $U_Q\subset {\rm int}(U_Q^*)$). 

To continue we see that 
\begin{equation}\label{u-small}
|u(Y)|\le \gamma\,\lambda,\qquad \mbox{for all }Y\in\Omega_{*}.
\end{equation}
Fix such a $Y$ so that $Y\in U_Q^*$ for some $Q\in\dd_{\F_{P_0}^*,P_0}$. If $Q\cap F_\lambda=\emptyset$, by maximality of the cubes in $\F_{P_0}^*$, it follows that  $Q\subset P_j$ for some $P_j\in \F_{P_0}^*$, which contradicts the fact $Q\in\dd_{\F_{P_0}^*,P_0}$. Thus, $Q\cap F_\lambda\neq\emptyset$ and we can select $x\in Q\cap F_\lambda$ so that by definition
$|u(Y)|\le \mathcal{N}_{Q_0} u(x)\le\gamma\,\lambda$ since $Y\in U_Q^*\subset\Gamma^*_{Q_0}(x)$. 

Apply Lemma \ref{prop:Pi-proj} to find $X_*:=Y_{P_0}\in \Omega_*\cap\Omega$ so that 
\begin{align}\label{qwfafrwe}
\ell(P_0)\approx \dist(X_*,\pom_*)\approx \delta(X_*).
\end{align}
Let $\omega_L^*:=\omega_{L, \Omega_*}^{X_*}$ be the elliptic measure associated with $L$ relative to $\Omega_*$ with pole at $X_*$ and 
write $\delta_*=\dist(\cdot,\pom_*)$. Given $Y\in \Omega_*$, we choose $y_Y\in
\partial\Omega_*$ such that $|Y-y_Y|= \delta_*(Y)$. By definition, for $x\in  F_\lambda$ and $Y\in\Gamma_{P_0}(x)$,
there is a $Q\in\dd_{P_0}$ such that $Y\in U_Q$ and $x\in Q$. Thus, by the triangle inequality, and the definition of $U_Q$, we have that for $Y\in\Gamma_{P_0}(x)$,
\begin{equation}\label{qtfg3gq3g}
|x-y_Y| 
\leq 
|x-Y|+  
\delta_*(Y) 
\approx 
\delta(Y) +  \delta_*(Y) 
\approx
\delta_*(Y)
\end{equation}
where in the last step we have used that
\begin{equation}\label{delta-delta*-saw}
\delta(Y) \approx  \delta_*(Y) \qquad\text{for}\quad Y\in\bigcup_{Q\in\dd_{\F_{P_0}^*,P_0}} U_Q.
\end{equation}
On the other hand, as observed above $F_\lambda \subset P_0\setminus(\cup_\F Q_j)\subset\pom\cap\pom_*$, see  \cite[Proposition 6.1]{HM1}. Using this and the fact that if $Q\cap F_\lambda\neq\emptyset$ then $Q\in \dd_{\F_{P_0}^*,P_0}$ we have
\begin{align}\label{sigma1-2}
\int_{F_\lambda} \mathcal{S}_{P_0}^{k_0} u(x)^2\,d\omega_L^*(x)
&=
\int_{F_\lambda} \iint_{\Gamma_{P_0}^{k_0}(x)} |\nabla u(Y)|^2\,\delta(Y)^{1-n}\,dY\,d\omega_L^*(x)
\\ \nonumber
&\le
\int_{F_\lambda} 
 \sum_{\substack{x\in Q\in\dd_{P_0}\\  \ell(Q)\ge 2^{-k_0}\,\ell(Q_0)}}\iint_{U_Q} |\nabla u(Y)|^2\,\delta(Y)^{1-n}\,dY\,d\omega_L^*(x)
\\ \nonumber
&
\lesssim
\sum_{Q\in\dd_{\F_{P_0}^*,P_0}}
\Big(\iint_{U_Q} |\nabla u(Y)|^2\,dY\Big)\,\ell(Q)^{1-n}\,\omega_L^*(Q\cap F_\lambda)
\\ \nonumber
&
\lesssim
\sum_{\substack{Q\in\dd_{\F_{P_0}^*,P_0}\\ \ell(Q)\ge M^{-1}\ell(P_0)}} \dots
+
\sum_{\substack{Q\in\dd_{\F_{P_0}^*,P_0}\\ \ell(Q)<M^{-1}\ell(P_0)}} \dots
\\ \nonumber
&
=:
\Sigma_1+\Sigma_2,
\end{align}
where $M$ is a large constant to be chosen.

We start estimating $\Sigma_1$. Note first that $\#\{Q:\in\dd_{P_0}: \ell(Q)\ge M^{-1}\ell(P_0)\}\le C_M$, thus
\begin{align*}
\Sigma_1
&
\lesssim
\sum_{\substack{Q\in\dd_{\F_{P_0}^*,P_0}\\ \ell(Q)\ge M^{-1}\ell(P_0)}}\ell(Q)^{1-n}\sum_{I^\in \W_Q^*} \iint_{I^*} |\nabla u(Y)|^2\,dY
\\
&\lesssim
\sum_{\substack{Q\in\dd_{\F_{P_0}^*,P_0}\\ \ell(Q)\ge M^{-1}\ell(P_0)}}\ell(Q)^{1-n}\sum_{I\in \W_Q^*} \ell(I)^{-2}\iint_{I^{**}} |u(Y)|^2\,dY
\\
&\lesssim
(\gamma\,\lambda)^2
\sum_{\substack{Q\in\dd_{P_0}\\ \ell(Q)\ge M^{-1}\ell(P_0)}}\ell(Q)^{1-n}\sum_{I\in \W_Q^*} \ell(I)^{n-1}
\\
&\lesssim_M
(\gamma\,\lambda)^2,
\end{align*}
where we have used \eqref{u-small}, along with the fact that $\interior(I^{**})\subset \interior(U_Q^*)\subset \Omega_*$ for any $I\in \W_Q^*$ with $Q\in\dd_{\F_{P_0}^*, P_0}$, and the fact that $\W_Q^*$ has uniformly bounded cardinality. To estimate $\Sigma_2$ we note that picking $y_Q\in Q\cap F_\lambda$ we have that $Q\cap F_\lambda\subset B(y_Q, 2\,\diam(Q))\cap\pom_*=:\Delta_Q^*$. Write $X_Q^*$ for Corkscrew relative to $\Delta_Q^*$ with respect to $\Omega_*$ so that $\delta_*(X_Q^*)\approx\diam(Q)\lesssim M^{-1}\ell(P_0)$. Note that by \eqref{qwfafrwe}, we clearly have $X_*\in\Omega\setminus B(y_Q, 4\,\diam(Q))$  provided $M$ is sufficiently large. Hence, 
by Lemma \ref{lemma:proppde} part $(b)$ applied in $\Omega_*$, which is a 1-sided NTA domain satisfying the CDC by Proposition \ref{prop:CDC-inherit}, we obtain for every $Y\in U_Q$
\begin{equation}\label{Qrqf3fa}
\ell(Q)^{1-n}\, \omega_L^*(Q\cap F_\lambda)
\lesssim
\diam(Q)^{1-n}\omega_L^*(\Delta_Q^*)
\lesssim
G_{L,*}(X_*, X_Q^*)
\approx
G_{L,*}(X_*, Y),
\end{equation}
where $G_{L,*}$ is the Green function for the operator $L$ relative to the domain $\Omega_*$. Above the last estimate uses Harnack's inequality (we may need to tale $M$ slightly larger) and the fact that by \eqref{delta-delta*-saw}, one has 
$\delta_*(Y)\approx \ell(Q)\approx \diam(Q)\approx \delta_*(X_Q^*)$ (see Remark \ref{remark:diam-radius}) and that if $I\ni Y$ with $I\in \W_{Q}^*$
\[
|Y-X_*|
\le
\diam(I)+\dist(I,Q)+\diam(Q)+|y_Q-X^*|
\lesssim
\diam(Q).
\]
Write $\{P_0^i\}_i\subset \dd_{P_0}$ for the collection of dyadic cubes with $M\,\ell(P_0)\le \ell(P_0^i)<2\,M\ell(P_0)$ which has uniformly bounded  cardinality depending on $M$. Note that
\[
\{
Q\in\dd_{\F_{P_0}^*,P_0}: \ell(Q)<M^{-1}\ell(P_0)\}
\subset \bigcup_i \dd_{\F_{P_0}^*,P_0^i}.
\]
For each $i$, if $\dd_{\F_{P_0}^*,P_0^i}\neq\emptyset$ then $P_0^i\in\dd_{\F_{P_0}^*,P_0}$ and hence $P_0^i\cap F_\lambda\neq \emptyset$. Pick then $y_i\in P_0^i\cap F_\lambda$ and note that for every $Q\in \dd_{\F_{P_0}^*,P_0^i}$ by \eqref{definicionkappa12} it follows that
\[
U_Q\subset T_{P_0^i}\cap\Omega_*\subset B_{P_0^i}^*\cap\Omega_*\subset B(y_i, C\,\kappa_0\,\ell(P_0^i))\cap\Omega_*=: B_i\cap\Omega_*.
\]
Using then \eqref{Qrqf3fa} we have 
\begin{align*}
\Sigma_2
&\lesssim
\sum_{\substack{Q\in\dd_{\F_{P_0}^*,P_0}\\ \ell(Q)<M^{-1}\ell(P_0)}} \iint_{U_Q} |\nabla u(Y)|^2\,G_{L,*}(X_*, Y)\,dY
\\
&\lesssim 
\sum_i 
\sum_{\substack{Q\in\dd_{\F_{P_0}^*,P_0^i}\\ \ell(Q)<M^{-1}\ell(P_0)}} \iint_{U_Q} |\nabla u(Y)|^2\,G_{L,*}(X_*, Y)\,dY
\\
&\lesssim 
\sum_i \iint_{B_i\cap\Omega_*} |\nabla u(Y)|^2\,G_{L,*}(X_*, Y)\,dY
\\
&\lesssim
\|u\|_{L^\infty(\Omega_*)}^2\sum_i \omega_L^*(B_i\cap\pom_*)
\\
&\lesssim
(\gamma\,\lambda)^2,
\end{align*}
where we have invoked Theorem \ref{THEOR:CME} applied in $\Omega_*$, which is a 1-sided NTA domain satisfying the CDC by Proposition \ref{prop:CDC-inherit}, and we may need to take $M$ slightly larger and use Harnack's inequality;  \eqref{u-small}; and the fact that $\{P_0^i\}_i\subset \dd_{P_0}$ has uniformly bounded  cardinality. 

Using Chebyshev's inequality, \eqref{sigma1-2}, and collecting the estimates for $\Sigma_1$ and $\Sigma_2$ 
we conclude that 
\begin{align*}
\omega_L^*(\widetilde{E}_\lambda\cap F_\lambda)
\le 
\frac1{(\beta\,\lambda)^2}\,\int_{\widetilde{E}_\lambda\cap F_\lambda} (\mathcal{S}_{P_0}^{k_0} u)^2\,d\omega_L^*
\le
\frac1{(\beta\,\lambda)^2}\,
\int_{F_\lambda} \mathcal{S}_{P_0}^{k_0} u(x)^2\,d\omega_L^*(x)
\lesssim
\Big(\frac{\gamma}{\beta}\Big)^2.
\end{align*}
At this point we invoke Lemma \ref{lemma:DJK-sawtooth:new} in $P_0$ with $\F_{P_0}^*$ ---we warn the reader that $P_0$ and $\F_{P_0}^*=\{P_j\}_j$ play the role of $Q_0$ and $\{Q_j\}_j$ and that associated to each $P_j$ one finds $\widetilde{P}_j$ as in Proposition \ref{prop:Pi-proj}, which now plays the role of $P_j$ in that result, and $\mu=\omega_{L}^{X_*}$ (recall that $X_*=Y_{P_0}$) and observe that 
the fact that $F_\lambda\subset P_0\setminus(\cup_{\dd_{\F_{P_0}^*, P_0}} P_j)$ implies on account of \eqref{ainfsawtooth:new:new} that for some $\vartheta>0$ we have 
\begin{align*}
\frac{\omega_L^{X_*}(\widetilde{E}_\lambda\cap F_\lambda)}{\omega_L^{X_*}(P_0)}
\lesssim
\bigg(\frac{\omega_L^*(\widetilde{E}_\lambda\cap F_\lambda)}{\omega_L^*(\Delta_\star^{P_0})}\bigg)^{\frac{\vartheta}{2}}
\lesssim
\Big(\frac{\gamma}{\beta}\Big)^\vartheta, 
\end{align*}
where we have used that $\omega_L^*(\Delta_\star^{P_0})\approx 1$ since $\Delta_\star^{P_0}:=B(x_{P_0}^\star,t_{P_0})\cap\partial\Omega_*$ with $t_{P_0}\approx\ell({P_0})\approx \diam(\pom_*)$, $x_{P_0}^\star\in \partial\Omega_*$, \eqref{qwfafrwe}, Harnack's inequality, and Lemma \ref{lemma:proppde} part $(a)$. 
We can then use Remark \ref{remark:chop-dyadic}, Harnack's inequality, and \eqref{qwfafrwe}, to conclude that
\[
\frac{\omega_L^{X_{Q_0}}(\widetilde{E}_\lambda\cap F_\lambda)}{\omega_L^{X_{Q_0}}(P_0)}
\approx
\frac{\omega_L^{X_{P_0}}(\widetilde{E}_\lambda\cap F_\lambda)}{\omega_L^{X_{P_0}}(P_0)}
\approx
\frac{\omega_L^{X_*}(\widetilde{E}_\lambda\cap F_\lambda)}{\omega_L^{X_*}(P_0)}
\lesssim
\Big(\frac{\gamma}{\beta}\Big)^\vartheta.
\]
Recalling that $P_0=Q_j\in\F_0$, and the definitions of $\widetilde{E}_\lambda$ and $F_\lambda$ in \eqref{def:good-lambda-sets} the previous estimates readily lead to \eqref{good-lambda:goal:local}.

To conclude we need to see how \eqref{good-lambda:goal} yields \eqref{S<N}. With this goal in mind we first observe that 
for every $x\in Q_0$ and $Y\in \Gamma_{Q_0}^{k_0}(x)$ one has that $Y\in \overline{B_{Q_0}^*}\cap\Omega$ (see \eqref{definicionkappa12}) and also
$\delta(Y)\gtrsim 2^{-k_0}\,\ell(Q_0)$. Hence, since $u\in W^{1,2}_{\rm loc} (\Omega)$, one has
\begin{multline}\label{finite-truncate-Su}
\sup_{x\in Q_0}\mathcal{S}_{Q_0}^{k_0}u(x)
=
\sup_{x\in Q_0} \Big(
\iint_{\Gamma_{Q_0}^{k_0}(x)}|\nabla u(Y)|^2\delta(Y)^{1-n}\,dY\Big)^{\frac12}
\\
\lesssim
(2^{-k_0}\,\ell(Q_0))^{\frac{1-n}2}
\Big(
\iint_{B_{Q_0}^*\cap\{Y\in\Omega:\delta(Y)\gtrsim 2^{-k_0}\,\ell(Q_0)\}} |\nabla u(Y)|^2\,dY\Big)^{\frac12}
<\infty.
\end{multline}
On the other hand, given $0<q<\infty$, we can use \eqref{good-lambda:goal} 
\begin{align*}
&(1+\beta)^{-q}\,\|\mathcal{S}_{Q_0}^{k_0}u\|_{L^q(Q_0,\omega_L^{X_{Q_0}})}^{q}
\\
&\qquad
=
\int_0^\infty q\, \lambda^q\, \omega_L^{X_{Q_0}}\big(\big\{x\in Q_0: \mathcal{S}_{Q_0}^{k_0}u(x)>(1+\beta)\,\lambda\big\}\big)\,\frac{d\lambda}{\lambda}
\\
&\qquad
\le 
\int_0^\infty q\, \lambda^q\, \omega_L^{X_{Q_0}}\big(\big\{x\in Q_0: \mathcal{S}_{Q_0}^{k_0}u(x)>(1+\beta)\,\lambda,\ \mathcal{N}_{Q_0}u(x)\le \gamma\,\lambda \big\}\big)\,\frac{d\lambda}{\lambda}
\\
&\qquad\qquad\qquad+
\int_0^\infty q\, \lambda^q\, \omega_L^{X_{Q_0}}\big(\big\{x\in Q_0: \mathcal{N}_{Q_0}u(x)> \gamma\,\lambda \big\}\big)\,\frac{d\lambda}{\lambda}
\\
&\qquad
\lesssim
\Big(\frac{\gamma}{\beta}\Big)^\vartheta\,
\int_0^\infty q\, \lambda^q\, \omega_L^{X_{Q_0}}\big(\big\{x\in Q_0: \mathcal{S}_{Q_0}^{k_0}u(x)>\beta\,\lambda\big\}\big)\,\frac{d\lambda}{\lambda}
\\
&\qquad\qquad\qquad+
\gamma^{-q}\,\|\mathcal{N}_{Q_0}u\|_{L^q(Q_0,\omega_L^{X_{Q_0}})}^{q}
\\
&\qquad
\lesssim
\Big(\frac{\gamma}{\beta}\Big)^\vartheta\,\beta^{-q}\,\|\mathcal{S}_{Q_0}^{k_0}u\|_{L^q(Q_0,\omega_L^{X_{Q_0}})}^{q}+
\gamma^{-q}\,\|\mathcal{N}_{Q_0}u\|_{L^q(Q_0,\omega_L^{X_{Q_0}})}^{q}.
\end{align*}
We can then choose $\gamma$ small enough so that we can hide the first term in the right hand side of the last quantity (which is finite by \eqref{finite-truncate-Su}) and eventually conclude that
\[
\|\mathcal{S}_{Q_0}^{k_0}u\|_{L^q(Q_0,\omega_L^{X_{Q_0}})}^{q}
\lesssim
\|\mathcal{N}_{Q_0}u\|_{L^q(Q_0,\omega_L^{X_{Q_0}})}^{q}.
\]
Since the implicit constant does not depend on $k_0$ and $\mathcal{S}_{Q_0}^k u(x)\nearrow \mathcal{S}_{Q_0}u(x)$ as $k\to\infty$ for every $x\in Q$, the monotone convergence theorem yields at once \eqref{S<N} and the proof Theorem \ref{THEOR:S<N} is complete.

\bibliographystyle{plain}

\bibliography{myref}

\begin{thebibliography}{10}

\bibitem{AHMT-full}
M.~Akman, S.~Hofmann, J.~M. Martell, and T.~Toro.
\newblock Perturbation of elliptic operators in 1-sided {NTA} domains
  satisfying the capacity density condition.
\newblock \url{https://arxiv.org/abs/1901.08261v2}. \textit{Preprint}, 2019.

\bibitem{AHMT-II}
M.~Akman, S.~Hofmann, J.~M. Martell, and T.~Toro.
\newblock Perturbation of elliptic operators in 1-sided {NTA} domains
  satisfying the capacity density condition.
\newblock \url{https://arxiv.org/abs/1901.08261v3}. \textit{Preprint}, 2021.

\bibitem{CFK}
L.~A. Caffarelli, E.~B. Fabes, and C.~E. Kenig.
\newblock Completely singular elliptic-harmonic measures.
\newblock {\em Indiana Univ. Math. J.}, 30(6):917--924, 1981.

\bibitem{CDMT}
M.~Cao, \'O. Domingu\'ez, J.~M. Martell, and P.~Tradacete.
\newblock On the ${A}_\infty$ condition for elliptic operators in 1-sided {NTA}
  domains satisfying the capacity density condition.
\newblock \url{https://arxiv.org/abs/2101.06064}. \textit{Preprint}, 2021.

\bibitem{CHM}
J.~Cavero, S.~Hofmann, and J.~M. Martell.
\newblock Perturbations of elliptic operators in 1-sided chord-arc domains.
  {P}art {I}: {S}mall and large perturbation for symmetric operators.
\newblock {\em Trans. Amer. Math. Soc.}, 371(4):2797--2835, 2019.

\bibitem{CHMT}
J.~Cavero, S.~Hofmann, J.~M. Martell, and T.~Toro.
\newblock Perturbations of elliptic operators in 1-sided chord-arc domains.
  {P}art {II}: non-symmetric operators and {C}arleson measure estimates.
\newblock {\em Trans. Amer. Math. Soc.}, 373(11):7901--7935, 2020.

\bibitem{C}
M.~Christ.
\newblock A {$T(b)$} theorem with remarks on analytic capacity and the {C}auchy
  integral.
\newblock {\em Colloq. Math.}, 60/61(2):601--628, 1990.

\bibitem{CF-1974}
R.~R. Coifman and C.~Fefferman.
\newblock Weighted norm inequalities for maximal functions and singular
  integrals.
\newblock {\em Studia Math.}, 51:241--250, 1974.

\bibitem{D}
B.~E.~J. Dahlberg.
\newblock On the absolute continuity of elliptic measures.
\newblock {\em Amer. J. Math.}, 108(5):1119--1138, 1986.

\bibitem{DJK}
B.~E.~J. Dahlberg, D.~S. Jerison, and C.~E. Kenig.
\newblock Area integral estimates for elliptic differential operators with
  nonsmooth coefficients.
\newblock {\em Ark. Mat.}, 22(1):97--108, 1984.

\bibitem{F}
R.~Fefferman.
\newblock A criterion for the absolute continuity of the harmonic measure
  associated with an elliptic operator.
\newblock {\em J. Amer. Math. Soc.}, 2(1):127--135, 1989.

\bibitem{FKP}
R.~A. Fefferman, C.~E. Kenig, and J.~Pipher.
\newblock The theory of weights and the {D}irichlet problem for elliptic
  equations.
\newblock {\em Ann. of Math. (2)}, 134(1):65--124, 1991.

\bibitem{FP}
J.~Feneuil and B.~Poggi.
\newblock Generalized {C}arleson perturbations of elliptic operators and
  applications.
\newblock \url{https://arxiv.org/abs/2011.06574}. \textit{Preprint}, 2021.

\bibitem{GR}
J.~Garc{\'i}a-Cuerva and J.~L. Rubio~de Francia.
\newblock {\em Weighted norm inequalities and related topics}, volume 116 of
  {\em North-Holland Mathematics Studies}.
\newblock North-Holland Publishing Co., Amsterdam, 1985.
\newblock Notas de Matem\'atica [Mathematical Notes], 104.

\bibitem{HKM}
J.~Heinonen, T.~Kilpel{\"a}inen, and O.~Martio.
\newblock {\em Nonlinear potential theory of degenerate elliptic equations}.
\newblock Dover Publications, Inc., Mineola, NY, 2006.
\newblock Unabridged republication of the 1993 original.

\bibitem{HLMN}
S.~Hofmann, P.~Le, J.~M. Martell, and K.~Nystr\"{o}m.
\newblock The weak-{$A_\infty$} property of harmonic and {$p$}-harmonic
  measures implies uniform rectifiability.
\newblock {\em Anal. PDE}, 10(3):513--558, 2017.

\bibitem{HM1}
S.~Hofmann and J.~M. Martell.
\newblock Uniform rectifiability and harmonic measure {I}: {U}niform
  rectifiability implies {P}oisson kernels in {$L^p$}.
\newblock {\em Ann. Sci. \'Ec. Norm. Sup\'er. (4)}, 47(3):577--654, 2014.

\bibitem{HMM-UR}
S.~Hofmann, J.~M. Martell, and S.~Mayboroda.
\newblock Uniform rectifiability, {C}arleson measure estimates, and
  approximation of harmonic functions.
\newblock {\em Duke Math. J.}, 165(12):2331--2389, 2016.

\bibitem{HMM-trans}
S.~Hofmann, J.~M. Martell, and S.~Mayboroda.
\newblock Transference of scale-invariant estimates from {L}ipschitz to
  {N}on-tangentially accessible to {U}niformly rectifiable domains.
\newblock \url{https://arxiv.org/abs/1904.13116}. \textit{Preprint}, 2019.

\bibitem{HMT1}
S.~Hofmann, J.~M. Martell, and T.~Toro.
\newblock General divergence form elliptic operators on domains with {ADR}
  boundaries, and on 1-sided {NTA} domains.
\newblock Work in Progress, 2014.

\bibitem{HMT-NTA}
S.~Hofmann, J.~M. Martell, and T.~Toro.
\newblock {$A_\infty$} implies {NTA} for a class of variable coefficient
  elliptic operators.
\newblock {\em J. Differential Equations}, 263(10):6147--6188, 2017.

\bibitem{HMMM}
S.~Hofmann, D.~Mitrea, M.~Mitrea, and A.~J. Morris.
\newblock {$L^p$}-square function estimates on spaces of homogeneous type and
  on uniformly rectifiable sets.
\newblock {\em Mem. Amer. Math. Soc.}, 245(1159):v+108, 2017.

\bibitem{HK1}
T.~Hyt{\"o}nen and A.~Kairema.
\newblock Systems of dyadic cubes in a doubling metric space.
\newblock {\em Colloq. Math.}, 126(1):1--33, 2012.

\bibitem{HK2}
T.~Hyt{\"o}nen and A.~Kairema.
\newblock What is a cube?
\newblock {\em Ann. Acad. Sci. Fenn. Math.}, 38(2):405--412, 2013.

\bibitem{JK}
D.~S. Jerison and C.~E. Kenig.
\newblock Boundary behavior of harmonic functions in nontangentially accessible
  domains.
\newblock {\em Adv. in Math.}, 46(1):80--147, 1982.

\bibitem{L88}
J.~L. Lewis.
\newblock Uniformly fat sets.
\newblock {\em Trans. Amer. Math. Soc.}, 308(1):177--196, 1988.

\bibitem{LSW}
W.~Littman, G.~Stampacchia, and H.~F. Weinberger.
\newblock Regular points for elliptic equations with discontinuous
  coefficients.
\newblock {\em Ann. Scuola Norm. Sup. Pisa (3)}, 17:43--77, 1963.

\bibitem{MPT}
E.~Milakis, J.~Pipher, and T.~Toro.
\newblock Harmonic analysis on chord arc domains.
\newblock {\em J. Geom. Anal.}, 23(4):2091--2157, 2013.

\bibitem{MM}
L.~Modica and S.~Mortola.
\newblock Construction of a singular elliptic-harmonic measure.
\newblock {\em Manuscripta Math.}, 33(1):81--98, 1980/81.

\bibitem{Zhao}
Z.~Zhao.
\newblock B{MO} solvability and {$A_{\infty}$} condition of the elliptic
  measures in uniform domains.
\newblock {\em J. Geom. Anal.}, 28(2):866--908, 2018.

\end{thebibliography}

\end{document}